\documentclass[pdflatex,sn-mathphys-num]{sn-jnl}

\usepackage{graphicx}%
\usepackage{multirow}%
\usepackage{amsmath,amssymb,amsfonts}%
\usepackage{amsthm}%
\usepackage{mathrsfs}%
\usepackage[title]{appendix}%
\usepackage{xcolor}%
\usepackage{textcomp}%
\usepackage{manyfoot}%
\usepackage{caption}%
\usepackage{booktabs}%
\usepackage{algorithm}%
\usepackage{algorithmicx}%
\usepackage{algpseudocode}%
\usepackage{listings}%
\usepackage{cases}
\usepackage{multirow}
\usepackage{ulem}

\usepackage[abbreviations]{siunitx}  %
\usepackage{natbib}  %

\theoremstyle{thmstyleone}%
\newtheorem{theorem}{Theorem}[section]
\newtheorem{proposition}{Proposition}[section]%

\theoremstyle{thmstyletwo}%
\newtheorem{example}{Example}[section]%
\newtheorem{remark}{Remark}[section]%

\theoremstyle{thmstyleone}%
\newtheorem{definition}{Definition}[section]%

\newtheorem{assumption}{Assumption}[section]%
\newtheorem{lemma}{Lemma}[section]%
\newtheorem{corollary}{Corollary}[section]%

\begin{document}

\title[Article Title]{An inexact variable metric proximal linearization method for composite optimization on manifolds}


\author[1]{\fnm{Hao} \sur{He}}\email{mahehao@mail.scut.edu.cn}

\author*[1]{\fnm{Ruyu} \sur{Liu}}\email{maruyuliu@mail.scut.edu.cn}

\author[2]{\fnm{Yitian} \sur{Qian}}\email{yitian.qian@polyu.edu.hk}

\author[1]{\fnm{Shaohua} \sur{Pan}}\email{shhpan@scut.edu.cn}

\affil[1]{\orgdiv{School of Mathematics}, \orgname{South China University of Technology}, \orgaddress{\street{Wushan}, \city{Guangzhou},  \country{China}}}

\affil[2]{\orgdiv{Department of Data Science and Artificial Intelligence}, \orgname{The Hong Kong Polytechnic University}, \orgaddress{\city{Hongkong}, \country{China}}}


\abstract{This paper concerns the minimization of the composition of a nonsmooth convex function and a $\mathcal{C}^{1,1}$ mapping $F$ over a $\mathcal{C}^2$-smooth embedded closed submanifold $\mathcal{M}$. For this class of nonconvex and nonsmooth problems, we propose an inexact variable metric proximal linearization method by leveraging its composite structure and the retraction and first-order information of $\mathcal{M}$, which at each iteration seeks an inexact solution to a subspace constrained strongly convex problem by a practical inexactness criterion. Under the boundedness assumption on the iterate sequence, we establish the $O(\epsilon^{-3})$ oracle complexity with a dual fast gradient method as the inner solver, and prove that any cluster point of the iterate sequence is a stationary point. If in addition the constructed potential function has the Kurdyka-Lojasiewicz (KL) property on the set of cluster points, the iterate sequence  converges to a stationary point, and if the potential function has the KL property of exponent $q\in[\frac{1}{2},1)$, the local convergence rate is characterized. We also provide a condition only  involving the original data to identify the KL property of the potential function with an exponent $q\in[0,1)$. Numerical comparisons with the existing methods validate the efficiency of the proposed method.}





\maketitle

\section{Introduction}\label{sec:Intro}

Let $\mathbb{X}$ and $\mathbb{Z}$ be finite-dimensional real vector spaces endowed with an inner product $\langle\cdot,\cdot\rangle$ and its induced norm $\|\cdot\|$, and let $\mathcal{M}$ be a $\mathcal{C}^2$-smooth embedded closed submanifold of $\mathbb{X}$. We are interested in the composite optimization problem over the manifold $\mathcal{M}$
\begin{equation}\label{prob}
\min_{x\in\mathcal{M}}\Theta(x):=f(x)+\vartheta(F(x)),
\end{equation}
where $f:\mathbb{X}\to\overline{\mathbb{R}}:=(-\infty,\infty]$, $\vartheta:\mathbb{Z}\to\mathbb{R}$ and $F:\mathbb{X}\to\mathbb{Z}$ satisfy Assumption \ref{ass0}. 
\begin{assumption}\label{ass0}
 {\bf(i)} $f$ is a proper and lower semicontiuous (lsc) function that is $\mathcal{C}^{1,1}$ on an open convex set $\mathcal{O}\supset\mathcal{M}$ (i.e., differentiable and gradient $\nabla\!f$ is locally Lipschitz on $\mathcal{O}$);
	
 \noindent
 {\bf(ii)} $\vartheta$ is a convex function with a closed-form  proximal mapping;
	
 \noindent
 {\bf(iii)} $F$ is $\mathcal{C}^{1,1}$ on $\mathcal{O}$ (i.e., differentiable and Jacobian $F'(\cdot)$ is locally Lipschitz on $\mathcal{O}$);
	
 \noindent
 {\bf(iv)} $\Theta$ is bounded from below on $\mathcal{M}$, i.e., $\inf_{x\in\mathcal{M}}\Theta(x)>-\infty$.
\end{assumption}

Problem \eqref{prob} is general enough to cover the case that $\vartheta$ is weakly convex on $F(\mathcal{M})$. Indeed, with the weakly convex parameter $\rho$ of $\vartheta$, it can be reformulated as the one with $f(\cdot)\leftarrow f(\cdot)-\frac{\rho}{2}\|F(\cdot)\|^2$ and $\vartheta\leftarrow \vartheta+\frac{\rho}{2}\|\cdot\|^2$. Despite the generality, it often arises from machine learning and scientific computing, as illustrated by the examples below. 
\begin{example}\label{Example1}
 Sparse spectral clustering (see \cite{Lu2016,Park2018}) aims to seek a low-dimensional embedding $X\in\mathbb{R}^{n\times r}$ with $XX^{\top}$ having as many zero entries as possible via the model
\begin{equation}\label{ssc}
 \min_{X\in\mathcal{M}}\,{\rm tr}(X^{\top}AX)+\lambda\,\|XX^{\top}\|_1,
 \end{equation}
 where $\mathcal{M}=\{X\in\mathbb{R}^{n\times r}\,|\, X^{\top}X=I_r\}$ is the Stiefel manifold, $A$ is the normalized Laplacian matrix, $\lambda>0$ is the regularization parameter, and $\|\cdot\|_1$ is the entrywise $\ell_1$-norm of matrices. Obviously, problem \eqref{ssc} is a special case of \eqref{prob} with $\mathbb{X}=\mathbb{R}^{n\times r},\mathbb{Z}=\mathbb{S}^n$ and $f(X)={\rm tr}(X^{\top}AX), \vartheta(Z)=\lambda\|Z\|_1,F(X)=XX^{\top}$ for $X\in\mathbb{X}$ and $Z\in\mathbb{Z}$, where $\mathbb{S}^n$ is the space of all $n\times n$ real symmetric matrices. 
\end{example}
\begin{example}\label{Example2}
 The constrained group sparse principal component analysis (PCA) seeks a low-rank row-sparse embedding $X\!\in\mathbb{R}^{n\times r}$ with some prior information. It is formulated as 
 \begin{equation}\label{rspca}
 \min_{X\in\mathcal{M}}\big\{-{\rm tr}(X^{\top}B^{\top}BX)+\lambda\|X\|_{2,1}\ \ \mathrm{s.t.}\ \ E\circ (X^{\top}B^{\top}BX)=0\big\},
 \end{equation}
 where $\mathcal{M}$ and $\lambda$ are the same as in Example \ref{Example1}, $\|X\|_{2,1}\!:=\sum_{i=1}^n\|X_{i\cdot}\|$ is the row $\ell_{2,1}$-norm of $X$, $E$ is a $r\times r$ matrix with $0$ diagonal and $1$ off-diagonal entries, $B\in\mathbb{R}^{m\times n}$ is a data matrix, and ``$\circ$'' is the Hadamard product. Problem \eqref{rspca} is different from that of \cite{Lu2012} in the regularization term. To cope with its equality constraint, we resort to the $\ell_1$-norm penalty
 \begin{equation}\label{pen-rspca}
 \min_{X\in \mathcal{M}}\,-{\rm tr}(X^{\top}B^{\top}BX)+\lambda\|X\|_{2,1}+\rho\|E\circ(X^{\top}B^{\top}BX)\|_1,
 \end{equation}
 where $\rho>0$ is the penalty parameter. Notice that problem \eqref{pen-rspca} is a special case of \eqref{prob} with $\vartheta(X,Z)=\!\lambda\|X\|_{2,1}+\rho\|Z\|_1$ and $F(X)\!=(X;E\circ(X^{\top}B^{\top}BX))$ for $X\!\in\mathbb{R}^{n\times r}$ and $Z\in\mathbb{S}^r$.
\end{example}
\begin{example}\label{Example3}
 The proper symplectic decomposition is a snapshot-based basis generation method, which finds a symplectic basis matrix $X$ to minimize the projection error of the symplectic projection in the mean over all snapshots. From \cite{Peng2016}, it is characterized as
 \begin{equation}\label{psd-prob}
 \min_{X\in\mathcal{M}}\,\|XX^+A-A\|_F,
 \end{equation} 
 where $\mathcal{M}=\big\{X\in\mathbb{R}^{2n\times 2r}\mid X^{\top}J_{2n}X=J_{2r}\big\}$ with 
 \(
	J_{k}:=\begin{pmatrix}
		0&I_k\\
		-I_k&0
	\end{pmatrix}
\)
 for $k\in\mathbb{N}^*$ is the symplectic Stiefel manifold, $A\in\mathbb{R}^{2n\times 2m}$ is a data matrix, and $X^+:=J_{2r}^{\top}X^{\top}\!J_{2n}$ is the symplectic inverse of $X$. Problem \eqref{psd-prob} with the square of the Frobenius norm was considered in \cite{Gao2024,Jensen2024}. The main advantage of the square of the Frobenius norm lies in its smoothness. In contrast, the Frobenius norm itself is scale-invariant and then exhibits better robustness than its square. We are interested in seeking a sparse proper symplectic decomposition via 
\begin{equation}\label{reg-psdprob}
\min_{X\in\mathcal{M}}\,\|XX^+A-A\|_F+\lambda\|X\|_1,
\end{equation} 
where $\lambda\ge 0$ is the regularization parameter. Clearly, problem \eqref{reg-psdprob} is a special case of \eqref{prob} with $f\equiv 0,\vartheta(Z,X)=\!\|Z\|_F+\|X\|_1,F(X)=(X;XX^+A-A)$ for $(Z,X)\in\mathbb{R}^{2n\times 2m}\times\mathbb{R}^{2n\times 2r}$. 
\end{example} 
\begin{example}\label{Example4}
 Wang et al. \cite{Wang2025} recently considered sparse PCA under a stochastic setting where the underlying probability distribution of the random parameter is uncertain, and formulate it as a distributionally robust optimization (DRO) model based on a constructive approach to capturing uncertainty in the covariance matrix. The DRO model is reformulated as follows 
 \begin{equation}\label{DRO-spca}
 \min_{X\in\mathcal{M}}\,{\rm tr}((I_n\!-\!XX^{\top})\Sigma_n)+\lambda\|X\|_{1}+\rho_n\|(I_n\!-\!XX^{\top})\Sigma_n^{\frac{1}{2}}\|_F,
 \end{equation}
 where $\mathcal{M}$ is the same as in Example \ref{Example1}, $\Sigma_n$ is an $n\times n$ empirical covariance matrix, and $\rho_n>0$ is a constant related to $n$. Problem \eqref{DRO-spca}  is a special case of \eqref{prob} with $f(X)={\rm tr}((I_n\!-\!XX^{\top})\Sigma_n)$, $\vartheta(X,Z)=\lambda\|X\|_{1}+2\rho_n\|Z\|_F$ and $F(X)=(X;(I_n\!-\!XX^{\top})\Sigma_n^{\frac{1}{2}})$ for $X\!\in\mathbb{R}^{n\times r}$ and $Z\in\mathbb{S}^n$. 
\end{example}

As is well known, the element-wise $\ell_1$-norm $\|\cdot\|_1$, row $\ell_{2,1}$-norm $\|\cdot\|_{2,1}$ and Frobeninus norm $\|\cdot\|_F$ of matrices all have a closed-form proximal mapping (see, e.g., \cite{Si2024,XiaoLiu2021}), so the function $\vartheta$ in Examples \ref{Example1}-\ref{Example4} satisfies the requirement of Assumption \ref{ass0} (ii). It is worth pointing out that the closed-form proximal mapping of $\vartheta$ is not required by the theoretical analysis in Sections \ref{sec4}-\ref{sec5}, and is just for numerical tests in Section \ref{sec7}. 

The existing algorithms for manifold optimization are mostly proposed for the special cases of problem \eqref{prob}. Next we mainly review the algorithms developed by the composite structure of \eqref{prob} and the retraction and first-order information of manifolds, which are related to the forthcoming one in this work. For those designed by the retraction and first-order information of manifolds for minimizing a general but abstract nonsmooth function, see Zhang et al. \cite{Zhang2024} and Hosseini et al. \cite{Hosseini2018}; for those proposed by using the retraction but second-order information of manifolds, see Si et al. \cite{Si2024}; and for those without using the retraction of manifolds, see Liu et al. \cite{LiuXiao2024,XiaoLiu2021}.
\subsection{Related works}\label{sec1.1} 

Riemannian proximal gradient (RPG) methods are a class of popular ones based on the retraction and first-order information of manifolds. The first RPG method was proposed by Chen et al. \cite{Chen2020} for problem \eqref{prob} with the Stiefel manifold $\mathcal{M}$ and an identity mapping $F$ (i.e., $F\equiv\mathcal{I}$). At each iteration, it seeks the exact solution $v^k$ of 
\begin{equation}\label{RPG}
\min_{v\in T_{x^k}\mathcal{M}}\,\langle\nabla\!f(x^k), v\rangle+\dfrac{1}{2t}\|v\|^2+\vartheta(x^k\!+v)\ \ {\rm with}\ t>0,
\end{equation}
then finds a step-size $\alpha_k$ with a monotone line search along the direction $v^k$ and retracts $x^k\!+\alpha_kv^k$ onto $\mathcal{M}$ to yield the next iterate, where $T_{x^k}\mathcal{M}$ is the tangent space of $\mathcal{M}$ at $x^k$. Under an assumption a little stronger than Assumption \ref{ass0}, they proved that any cluster point of the iterate sequence is a stationary point, and the RPG method returns an $\epsilon$-stationary point defined by $v^k$ in $O(\epsilon^{-2})$ steps. Later, for \eqref{prob} with a finite dimensional Riemannian manifold $\mathcal{M}$, a continuous nonconvex $\vartheta$ and $F\equiv\mathcal{I}$, Huang and Wei \cite{Huang2022a} developed a RPG method by seeking an exact stationary point $v^k$ of
\begin{equation}\label{RPG1}
\min_{v\in T_{x^k}\mathcal{M}}\,\ell_{x^k}(v):=\langle\nabla\!f(x^k), v\rangle+\frac{\widetilde{L}}{2}\|v\|^2+\vartheta(R_{x^k}(v))
\end{equation}
with $\ell_{x^k}(v^k)\le\ell_{x^k}(0)$, where $R$ is a retraction and $\widetilde{L}>L$ is a constant. Under the $L$-retraction-smoothness of $f$ w.r.t. $R$ and the compactness of a level set of $\Theta$ restricted in $\mathcal{M}$, they obtained the subsequential convergence of the iterate sequence and the iteration complexity $O(\epsilon^{-2})$ for an $\epsilon$-stationary point defined by an exact stationary point of \eqref{RPG1}. If in addition $\Theta$ satisfies the Riemannian KL property, they proved that the iterate sequence converges to a stationary point, and provided the local convergence rate by requiring the Riemannian KL property of exponent.

For the problem considered in \cite{Chen2020}, Wang and Yang \cite{Wang2023,Wang2024} proposed Riemannian proximal quasi-Newton methods by exactly solving \eqref{RPG} with the term $\frac{1}{2t}\|v\|^2$ replaced by a variable metric one $\frac{1}{2}\|v\|_{\mathcal{B}_{k}}^2$ and searching for a step-size $\alpha_k$ along the direction $v^k$ with a nonmonotone line search strategy, where $\mathcal{B}_{k}\!:T_{x^k}\mathcal{M}\to T_{x^k}\mathcal{M}$ is a self-adjoint positive definite (PD) linear operator generated by a damped LBFGS strategy. Under the same restriction on $f$ and $\vartheta$ as in \cite{Chen2020} and the uniformly lower and upper boundedness assumption on $\{\mathcal{B}_{k}\}_{k\in\mathbb{N}}$, they achieved the subsequential convergence of the iterate sequence and the same iteration complexity as in \cite{Chen2020} for an $\epsilon$-stationary point, and if in addition the Riemannian Hessian of $f$ at a cluster point is positive definite, they proved that the iterate sequence converges to it with a linear rate  

The above-mentioned methods all focus on the special case $F\equiv\mathcal{I}$. For problem \eqref{prob} with a linear $F$, Deng and Peng \cite{deng2023manifold} proposed a Riemannian inexact augmented Lagrangian method (RiALM) where each subproblem is inexactly solved, and proved the subsequential convergence of the iterate sequence under the compactness of $\mathcal{M}$ and an assumption slightly stronger than Assumption \ref{ass0}; Deng et al. \cite{deng2025oracle} later developed a RiALM that achieves an oracle complexity of $O(\epsilon^{-3})$ when each subproblem is solved with a Riemannian gradient descent method. Beck and Rosset \cite{Beck2023} proposed a dynamic smoothing approach (DSGM) by replacing $\vartheta$ with its Moreau envelope and solving the smoothing problem with a Riemannian gradient descent method, and achieved the subsequential convergence of the iterate sequence and the iteration complexity $O(\epsilon^{-3})$ on an $\epsilon$-stationary point; Li et al. \cite{LiMa2024} proposed a Riemannian alternating direction method of multiplier (RADMM) method, and proved that it generates an $\epsilon$-stationary point of the constrained reformulation of \eqref{prob} in $O(\epsilon^{-4})$ steps. The convergence results of \cite{Beck2023,LiMa2024} and \cite{deng2025oracle} require the same restriction on $f$ and $\vartheta$ as in  \cite{Chen2020} as well as the compactness of $\mathcal{M}$. For the sparse spectral clustering (SSC) problem in Example \ref{Example1}, Wang et al. \cite{Wang2022} proposed a manifold proximal linear (ManPL) method by combining the ManPG in \cite{Chen2020} and the proximal linear method in the Euclidean space, and achieved the subsequential convergence and the same iteration complexity as in Chen et al. \cite{Chen2020}. For problem \eqref{prob} itself, Xu et al. \cite{Xu2024} proposed a RiALM by seeking at each iteration an approximate stationary point of the augmented Lagrangian subproblem with the Riemannian gradient descent method, and recently they reformulated \eqref{prob} as a Riemannian nonconvex-linear minimax problem and developed a Riemannian alternating descent ascent (RADA) algorithmic framework in \cite{xu2026riemannian}. These two methods were proved to return an $\epsilon$-stationary point with $O(\epsilon^{-3})$ first-order oracle calls under a slightly stronger version of Assumption \ref{ass0}. In addition, for problem \eqref{prob} with additional nonlinear inequality constraints, Zhou et al. \cite{Zhou2023} proposed a RiALM by solving inexactly the augmented Lagrangian subproblems with the manifold semismooth Newton method, and proved the subsequential convergence of the primal variable sequence under a constant positive linear dependent constraint qualification in the Riemannian sense. 

From the above discussion, for the composite problem \eqref{prob}, there is still a lack of algorithms with a full convergence certificate. For its special case $F\equiv\mathcal{I}$, the RPG method \cite{Huang2022a} has the full convergence if $\Theta$ satisfies the Riemannian KL property, but to check whether the property is satisfied by a continuous semialgebraic function is nontrivial since it requires constructing a chart $\phi_{x}$ for $x\in\mathcal{M}$; the proximal quasi-Newton methods \cite{Wang2023,Wang2024} have the full convergence certificate but need a restricted condition for the Riemannian Hessian of $f$. Not only that, their full convergence analysis either requires the exact stationary points \cite{Huang2022a} or optimal solutions \cite{Wang2023,Wang2024} of subproblems, which are unavailable in practice due to computation error or cost. The algorithm proposed recently in \cite{LiZhang2024} is applicable to problem \eqref{prob} with $F\equiv\mathcal{I}$, but its full convergence analysis also requires the exact optimal solutions of subproblems, so the gap still exists between theoretical analysis and practical implementation. For the special case $F\equiv\mathcal{I}$, Huang and Wei \cite{Huang2023} proposed an inexact RPG (IRPG) by solving \eqref{RPG1} at each iteration to achieve $\widetilde{v}^k\in T_{x^k}\mathcal{M}$ satisfying $\|\widetilde{v}^k-v^k\|\le q(\varepsilon_k,\|\widetilde{v}^k\|)$ and $\ell_{x^k}(\widetilde{v}^k)\leq\ell_{x^k}(0)$, where $v^k$ is an exact stationary point of subproblems and $q\!:\mathbb{R}^2\to\mathbb{R}$ is a continuous function with $q(\varepsilon_k,\|\widetilde{v}^k\|)$ for some $\varepsilon_k>0$ to control the accuracy for solving subproblems. They achieved the full convergence of the iterate sequence if $\Theta$ satisfies the Riemannian KL property on the accumulation point set and $q(\varepsilon_k,\|\widetilde{v}^k\|)=\varepsilon_k^2$ with $\sum_{k=1}^{\infty}\varepsilon_k<\infty$. However, its implementable version requires the retraction-convexity of $\vartheta$ on $\mathcal{M}$, and now it is unclear which nonsmooth $\vartheta$ satisfy such a property. Thus, to design an inexact algorithm for \eqref{prob} with a full convergence certificate in theory and a good performance in computation is still an unresolved task. This is precisely the motivation for this work. We notice that after the archive version of this work  (\url{https://arxiv.org/abs/2508.12003}) appeared, Zheng et al. \cite{zheng2025new} made their work public where, they proposed an inexact ManPL (iManPL) with the inexactness criterion HACC similar to ours but did not provide the full convergence of the iterate sequence. For clarity, we summarize the algorithms for problem \eqref{prob} in Table \ref{Compar-tab}.

\begin{table*}[h]
\centering
\caption{\centering Summary of the related algorithms for problem \eqref{prob}}\label{Compar-tab}
\setlength{\abovecaptionskip}{2pt}
\setlength{\belowcaptionskip}{0pt}
\scalebox{0.68}[0.70]{
\begin{tabular}{c|c|c|c|c|c|c} 
\hline    
Methods& Mapping $F$ &Manifold & Subproblems & Oracle complexity & S-convergence & F-convergence\\
\hline
RiALM \cite{deng2023manifold}&Linear&Compact& Inexact &No&Yes&No\\
\hline
DSGM \cite{Beck2023}&Linear&Compact& No &$O(\epsilon^{-3})$&Yes&No\\
\hline
RADMM \cite{LiMa2024}&Linear&Compact& No &$O(\epsilon^{-4})$&No&No\\
\hline
RiALM \cite{deng2025oracle}&Linear&Compact&Inexact&$O(\epsilon^{-3})$&No&No\\
\hline
ManPL \cite{Wang2022}&Nonlinear&Compact&Exact&No&Yes&No\\
\hline
RiALM \cite{Xu2024}&Nonlinear&Embedded& Inexact &$O(\epsilon^{-3})$&No&No\\
\hline
RADA \cite{xu2026riemannian}&Nonlinear&Embedded& No &$O(\epsilon^{-3})$&No&No\\ 
\hline
RiALM \cite{Zhou2023}&Nonlinear&General&Inexact&No&Yes&No\\ 
\hline
iManPL \cite{zheng2025new}&Nonlinear&Compact&Inexact&$O(\epsilon^{-3})$&Yes&No\\
\hline Ours&Nonlinear&Closed&Inexact&$O(\epsilon^{-3})$&Yes&Yes\\
\hline
\end{tabular}}

{\footnotesize where ``S-convergence'' signifies subsequential convergence, and ``F-convergence'' means full convergence.}
\end{table*}
\subsection{Main contributions}\label{sec1.2}

This work aims at developing an efficient inexact algorithm with a full convergence certificate for problem \eqref{prob}. And the main contributions are stated as follows. 

{\bf(i)} We propose a Riemannian inexact variable metric proximal linearization method (RiVMPL), which at each iteration solves a strongly convex subproblem 
\begin{equation}\label{init-subprob}
 \min_ {v\in T_{x^k}\mathcal{M}}\Theta_k(v):=\langle\nabla\!f(x^k), v\rangle+\frac{1}{2}\|v\|^2_{\mathcal{Q}_k}+\vartheta(F(x^k)\!+\!F'(x^k)v)+f(x^k)
\end{equation}
to seek a direction $v^k\in T_{x^k}\mathcal{M}$ with a certain decrease and then yield the next iterate by retracting $x^k+v^k$ onto the manifold $\mathcal{M}$ with a retraction $R$, where $\mathcal{Q}_k:\mathbb{X}\to\mathbb{X}$ is a PD linear operator. The introduction of $\mathcal{Q}_k$ is inspired by the work \cite{tao2023inexact}, which on one hand merges the second-order information of $f$ and $F$ into the subproblems to cope with their possible bad scaling, and on the other hand  simplifies greatly the solving of subproblems (see Section \ref{sec6}). Since the Lipschitz moduli of $\nabla\!f$ and $F'$ at the iterates are usually unknown, our method follows the line in \cite{tao2023inexact} to search for a tight upper estimation for them to formulate the PD linear operator $\mathcal{Q}_k$ and simultaneously solves the associated strongly convex subproblem. To our knowledge, this is the first Riemannian type inexact algorithm proposed by inexactly solving strongly convex subproblems, constructed with a variable metric proximal term $\frac{1}{2}\|v\|^2_{\mathcal{Q}_k}$ and the linearization of $F$ at the iterate. Different from the IRPG \cite{Huang2023}, the inexactness criterion for seeking $v^k$ is easily implementable. Unlike the proximal quasi-Newton methods \cite{Wang2023,Wang2024}, the variable metric linear operator $\mathcal{Q}_k$ is not restricted to be from $T_{x^k}\mathcal{M}$ to $T_{x^k}\mathcal{M}$, so its construction avoids the computation cost of projecting onto $T_{x^k}\mathcal{M}$. Our RiVMPL, as an inexact and variable metric version of ManPL \cite{Wang2022}, keeps its goodness to solve strongly convex subproblems, rather than to solve subproblems with the manifold and a highly nonlinear term like RiALM \cite{Xu2024}.
	
{\bf(ii)} Under the boundedness assumption on the iterate sequence, we establish the $O(\epsilon^{-2})$ iteration complexity and the $O(\epsilon^{-2})$ calls to the subproblem solver for returning an $\epsilon$-stationary point defined with the original variable as in Xu et al. \cite{Xu2024}. The existing iteration complexity for the algorithms, designed by solving a sequence of strongly convex subproblems, mostly focus on an $\epsilon$-stationary point by the optimal solution of subproblems rather than the original variable. When the dual fast first-order method (DFG) in \cite{Necoara2016} with returning the last primal iterate (resp. the average primal iterate) is adopted as an inner solver, the $O(\epsilon^{-4})$ (resp. $O(\epsilon^{-3})$) oracle complexity is obtained for RiVMPL with $\mathcal{Q}_{k}$ specified as in Theorem \ref{oracle}. This oracle complexity is consistent with that of RADMM in \cite{LiMa2024} (resp. RiALM in \cite{Xu2024} and RADA \cite{xu2026riemannian}) applied to \eqref{prob},  except that the former calls a retraction only at the outer iteration and the latter does this at each iteration.     
	
{\bf(iii)} We prove that the whole iterate sequence converges to a stationary point if  the constructed potential function $\Xi_{\widetilde{c}}$ has the KL property on the set of cluster points, and characterize the local convergence rate if the associated composite function $\Xi$ has the KL property of exponent $q\in[\frac{1}{2},1)$ at the interested point. It is worth emphasizing that the potential function value lacks the decrease due to the manifold constraint, so the analysis technique in \cite{tao2023inexact} is inapplicable to RiVMPL. We also provide a condition involving the original data for $\Xi$ to have the KL property of exponent $q\in[0,1)$ at the interested point, which is weaker than the one obtained by applying \cite[Theorem 3.2]{LiPong2018} to $\Xi$. Notice that few criteria are available to check the KL property with an explicit exponent $q\in[0,1)$ for the composite function $\Xi$. Since RiVMPL is an inexact version of the methods in \cite{Wang2023,Wang2024} if $\mathcal{Q}_k$ takes $\mathcal{B}_k$ there and $F\equiv\mathcal{I}$, these results also provide the full convergence certificate for them under the mild KL property, and the local linear convergence rate without restricting the Riemannian Hessian of $f$. Unlike the Riemannian KL property required in \cite{Huang2022a,Huang2023}, there are various chain rules friendly to optimization for identifying the KL property of nonsmooth functions, so to identify the KL property of $\Xi_{\widetilde{c}}$ is an easy task when the expression of $\mathcal{M}$ is available.

We test the performance of RiVMPL, armed with a dual semismooth Newton method as the inner solver, for solving Examples \ref{Example1}-\ref{Example3}. The results indicate that it needs less running time than the RiVMPL armed with the DFG for Example \ref{Example2}, and is comparable with the latter for Examples \ref{Example1} and \ref{Example3}. Numerical comparisons with the double-loop RiALM \cite{Xu2024} and the single-loop RADMM \cite{LiMa2024} show that the RiVMPL with the dual semismooth Newton method is superior to RADMM and RiALM in terms of the normalized mutual information scores (NMIs) for Example \ref{Example1} and the objective value and running time for Example \ref{Example2}, and is comparable with RADMM in the running time (much less than that of RiALM) and comparable even better than RiALM by the objective value (better than the one by RADMM) for Example \ref{Example3}.

\subsection{Notation}\label{sec1.3}

Throughout this paper, a hollow capital represents a finite-dimensional real vector space endowed with an inner product $\langle\cdot,\cdot\rangle$ and its induced norm $\|\cdot\|$, and $\mathbb{R}^{n\times r}$ denotes the space of all $n\times r$ real matrices with the trace inner product and its induced Frobenius norm $\|\cdot\|_F$. For a positive semidefinite (PSD) linear operator $\mathcal{Q}:\mathbb{X}\to\mathbb{X}$, write $\mathcal{Q}\succeq 0$, and define $\|\cdot\|_{\mathcal{Q}}:=\sqrt{\langle\cdot,\mathcal{Q}\cdot\rangle}$. Let $\mathbb{N}$ be the set of natural numbers, and $\mathbb{N}^*$ be the set of positive integers. For any $k\in\mathbb{N}$, write $[k]:=\{0,\ldots,k\}$; and for any $k\in\mathbb{N}^*$, $[k]_{+}:=\{1,\ldots,k\}$. For a closed set $C\subset\mathbb{X}$, $\delta_{C}$ represents its indicator function, i.e., $\delta_{C}(x)=0$ if $x\in C$, otherwise $\delta_{C}(x)=\infty$, and $\Pi_{C}(\cdot)$ denotes the projection mapping onto $C$. For any $x\in\mathbb{X}$, $\mathbb{B}(x,\delta)$ denotes the open ball centered at $x$ with radius $\delta$, and $\overline{\mathbb{B}}(x,\delta)$ represents its closure. For a mapping $g\!:\mathbb{X}\to\mathbb{Y}$ and a point $x\in\mathbb{X}$, if $g$ is differentiable at $x$, $\nabla g(x)$ means the adjoint of $g'(x)\!:\mathbb{X}\to\mathbb{Y}$, the differential of $g$ at $x$; if $g$ is twice differentiable at $x$, $D^2g(x)$ denotes the second-order differential of $g$ at $x$; and if $g$ is locally Lipschitz at $x$, ${\rm lip}\,g(x)$ signifies the Lipschitz modulus of $g$ at $x$. For a mapping $g\!:\mathbb{X}\to\mathbb{Y}$ and a set $C\subset\mathbb{X}$, $g(C):=\{g(x)\,|\,x\in C\}$ denotes a set of $\mathbb{Y}$. For any $x\in\mathcal{M}$, $T_x\mathcal{M}$ and $N_x\mathcal{M}$ signify the tangent and normal spaces of $\mathcal{M}$ at $x$. For a closed proper convex $h\!:\mathbb{Z}\to\overline{\mathbb{R}}$, its proximal mapping and Moreau envelope associated with a parameter $\gamma>0$ are respectively defined as $\mathcal{P}_{\gamma h}(z)\!:=\mathop{\arg\min}_{z'\in\mathbb{Z}}\big\{h(z')+\frac{1}{2\gamma}\|z'-z\|^2\big\}$ and $e_{\gamma h}(z)\!:=\min_{z'\in\mathbb{Z}}\big\{h(z')+\frac{1}{2\gamma}\|z'-z\|^2\big\}$.

 \section{Preliminaries}\label{sec2}

This section first introduces the stationary points of problem \eqref{prob}, and then recalls some preliminary knowledge on manifolds and Kurdyka-Lojasiewicz (KL) property.  

\subsection{Stationary points}\label{sec2.1}

In view of Assumption \ref{ass0} (i)-(iii), $\Theta$ is locally Lipschitz at any $x\in\mathcal{M}$, and it is regular at any $x\in\mathcal{M}$ with  $\partial\Theta(x)=\nabla\!f(x)+\nabla F(x)\partial\vartheta(F(x))$ by \cite[Theorem 10.6]{RW98}, where $\partial\Theta(x)$ is the limiting (or Mordukhovich) subdifferential of $\Theta$ at $x$. Along with the regularity of $\mathcal{M}$, the extended objective function $\widetilde{\Theta}:=\Theta+\delta_{\mathcal{M}}$ of \eqref{prob} is regular with 
\[
\partial\widetilde{\Theta}(x)=\nabla\!f(x)+\nabla\!F(x)\partial\vartheta(F(x))+N_{x}\mathcal{M}\quad\forall x\in\mathcal{M}.
\]
Based on this, we introduce the following concept of stationary points for problem \eqref{prob}.
\begin{definition}\label{def-spoint}
 {\bf(i)} A vector $\overline{x}\in\mathcal{M}$ is called a stationary point of problem \eqref{prob} if 
 \begin{equation}\label{spoint-inclusion}
 0\in\nabla\!f(\overline{x})+\nabla\!F(\overline{x})\partial \vartheta(F(\overline{x}))+N_{\overline{x}}\mathcal{M}=\partial\widetilde{\Theta}(\overline{x}).
 \end{equation}
 {\bf(ii)} For any given $\epsilon>0$, a vector $\overline{x}\in\mathcal{M}$ is called an $\epsilon$-stationary point of \eqref{prob} if there exist $\overline{z}\in\mathbb{Z}$ and $\overline{\xi}\in\partial\vartheta(\overline{z})$ such that 
 \(
 \max\big\{\|\Pi_{T_{\overline{x}}\mathcal{M}}(\nabla\!f(\overline{x})+\!\nabla\!F(\overline{x})\overline{\xi})\|,\|F(\overline{x})-\overline{z}\|\big\}\le\epsilon.
 \) 
\end{definition}
\begin{remark}\label{remark-spoint}
 {\bf(a)} It is not difficult to check that $\overline{x}\in\mathcal{M}$ is a stationary point of \eqref{prob} if and only if there exists a PD linear operator $\mathcal{Q}:\mathbb{X}\to\mathbb{X}$ such that 
 \begin{equation*}
  0=\overline{v}_{\mathcal{Q}}:=\mathop{\arg\min}_{v\in T_{\overline{x}}\mathcal{M}}\,\langle\nabla\!f(\overline{x}), v\rangle+\frac{1}{2}\|v\|_{\mathcal{Q}}^2+\vartheta(F(\overline{x})\!+\!F'(\overline{x})v).
 \end{equation*}
	
 \noindent
 {\bf(b)} The $\epsilon$-stationary point here is precisely the one introduced in Xu et al. \cite{Xu2024}. When $F$ is linear, if $\overline{x}\in\mathcal{M}$ is an $\epsilon$-stationary point, there exist $\overline{z}\in\mathbb{Z}$ and $\overline{\xi}\in\partial\vartheta(\overline{z})$ such that $(\overline{x},\overline{z},\overline{\xi})$ is a $\sqrt{2}\epsilon$-stationary point defined in \cite{LiMa2024} by the constrained reformulation of \eqref{prob}, while if $(\overline{x},\overline{z},\overline{\lambda})$ is an $\epsilon$-stationary point defined in \cite{LiMa2024}, then $\overline{x}$ is a $(1+\|\nabla F(\overline{x})\|)\epsilon$-stationary point. 


\end{remark}  
\subsection{Tangent bundle of $\mathcal{M}$}\label{sec2.2}

Recall that $\mathcal{M}$ is a $\mathcal{C}^{2}$-smooth embedded closed submanifold of $\mathbb{X}$. According to \cite[Definition 3.10]{Boumal2023}, at any $x\in\mathcal{M}$, there exists an open neighborhood $\mathcal{U}_{x}$ of $x$ and a $\mathcal{C}^{2}$-smooth mapping $G_{x}\!:\mathbb{X}\to \mathbb{Y}$ such that for all $u\in\mathcal{M}\cap\mathcal{U}_{x}$, $G_{x}'(u)\!:\mathbb{X}\to \mathbb{Y}$ are surjective and $\mathcal{M}\cap\mathcal{U}_{x}=\big\{u\in\mathcal{U}_{x}\,|\,G_{x}(u)=0\big\}$. Then, for any $u\in\mathcal{M}\cap\mathcal{U}_{x}$, it holds 
\begin{equation}\label{tangent-normal}
T_u\mathcal{M}=\big\{d\in\mathbb{X}\mid G_{x}'(u)d=0\big\}\ \ {\rm and}\ \  N_u\mathcal{M}=\big\{\nabla G_{x}(u)y\mid y\in\mathbb{Y}\big\}.
\end{equation}
Unless otherwise stated, in the rest of this paper, for every $x\in\mathcal{M}$, $G_x$ represents such a $\mathcal{C}^{2}$-smooth  mapping. From \cite[Definition 3.42]{Boumal2023}, the tangent bundle of $\mathcal{M}$ is the disjoint union of the tangent spaces of $\mathcal{M}$, i.e., $T\mathcal{M}\!:=\{(x,d)\in\mathbb{X}\times\mathbb{X}\mid x\in\mathcal{M}, d\in T_{x}\mathcal{M}\}$. Here, ``disjoint'' means that for each tangent vector $d\in T_{x}\mathcal{M}$, the pair $(x,d)$ rather
than simply $d$ is retained. Since $\mathcal{M}$ is an embedded submanifold of $\mathbb{X}$, according to \cite[Theorem 3.43]{Boumal2023}, $T\mathcal{M}$ is a $\mathcal{C}^{1}$-smooth embedded submanifold of $\mathbb{X}\times\mathbb{X}$. Next we prove that $T\mathcal{M}$ is also closed and characterize the normal space of $T\mathcal{M}$.
\begin{lemma}\label{lemma-Tb}
 The tangent bundle $T\mathcal{M}$ of $\mathcal{M}$ is closed, and for any $(x,v)\in T\mathcal{M}$ it holds
 \begin{equation*}
  N_{(x,v)}T\mathcal{M}=\left\{\begin{pmatrix}
	 \nabla G_{x}(x)\xi+[D^2G_{x}(x)v]^*\zeta\\
	\nabla G_{x}(x)\zeta
 \end{pmatrix}\ |\ \xi\in\mathbb{Y},\zeta\in\mathbb{Y}\right\}.
\end{equation*}
In particular, for every $x\in \mathcal{M}$, there exists an open neighborhood $\mathcal{U}_{x}$ of $x$ such that for all $(z,v)\in T\mathcal{M}\cap [\mathcal{U}_{x}\times \mathbb{X}]$,
 \begin{equation}\label{tb-eq1}
  N_{(z,v)}T\mathcal{M}=\left\{\begin{pmatrix}
	 \nabla G_{x}(z)\xi+[D^2G_{x}(z)v]^*\zeta\\
	\nabla G_{x}(z)\zeta
 \end{pmatrix}\ |\ \xi\in\mathbb{Y},\zeta\in\mathbb{Y}\right\}.
\end{equation}
\end{lemma}
\begin{proof}
 Let $\{(x^k,v^k)\}_{k\in\mathbb{N}}\subset T\mathcal{M}$ be a sequence with $(x^k,v^k)\to (x,v)$ as $k\to \infty$. We claim that $(x,v)\in T\mathcal{M}$, so the tangent bundle $T\mathcal{M}$ is closed. Indeed, $x\in \mathcal{M}$ because $\{x^k\}_{k\in\mathbb{N}}\subset\mathcal{M}$ and $\mathcal{M}$ is closed, so it suffices to prove $v\in T_x\mathcal{M}$. Since $v^k\in T_{x^k}\mathcal{M}$ for each $k$ and $x^k\to x$ as $k\to\infty$, from the first equality of \eqref{tangent-normal}, for sufficiently large $k$, $G_{x}'(x^k)v^k=0$. Taking the limit as $k\to\infty$ in this equality and using the smoothness of $G_{x}$ leads to $G_{x}'(x)v=0$, which implies $v\in T_x\mathcal{M}$. The first part of the conclusions follows. For the second part, fix any $(x,v)\in T\mathcal{M}$. Since $x\in\mathcal{M}$, there exists a $\mathcal{C}^{2}$-smooth mapping $G_{x}\!:\mathbb{X}\to \mathbb{Y}$ and an open neighborhood $\mathcal{U}_{x}$ of $x$ such that for all $u\in\mathcal{U}_{x}$, $G_{x}'(u)\!:\mathbb{X}\to \mathbb{Y}$ are surjective and  $\mathcal{M}\cap\mathcal{U}_{x}=\big\{u\in\mathcal{U}_{x}\,|\,G_{x}(u)=0\big\}$. Together with the definition of $T\mathcal{M}$ and the first equality of \eqref{tangent-normal}, it follows
 \begin{equation}\label{temp-TM}
 T\mathcal{M}\cap\big[\mathcal{U}_{x}\times\mathbb{X}\big]=\big\{(u,d)\in\mathcal{U}_{x}\times\mathbb{X}\ |\ G_{x}(u)=0,\,G_{x}'(u)d=0\big\}.
 \end{equation}
 Define $H(u,d):=(G_{x}(u); G_{x}'(u)d)$ for $(u,d)\in\mathcal{U}_{x}\times\mathbb{X}$. It is not hard to check that the mapping $H'(x,v):\mathbb{X}\times\mathbb{X}\to\mathbb{Y}\times\mathbb{Y}$ is surjective. Using \cite[Exercise 6.7]{RW98} leads to   
 \[
	N_{(x,v)}T\mathcal{M}= N_{T\mathcal{M}\cap[\mathcal{U}_{x}\times\mathbb{X}]}(x,v)=\Big\{\nabla H(x,v)\begin{pmatrix}
		\xi\\ \zeta
	\end{pmatrix}\ |\ \xi\in\mathbb{Y},\zeta\in\mathbb{Y}\Big\},
 \]
 where $N_{T\mathcal{M}\cap[\mathcal{U}_{x}\times\mathbb{X}]}(x,v)$ is the normal cone to $T\mathcal{M}\cap[\mathcal{U}_{x}\times\mathbb{X}]$ at $(x,v)$. The desired equality on $N_{(x,v)}T\mathcal{M}$ follows by the expression of $\nabla H(x,v)$. The proof is completed.
\end{proof}

Next we characterize the normal space of $T\mathcal{M}$ at a point $(x,v)$ for $x$ from a compact subset of $\mathcal{M}$. This is used to achieve the relative error condition in Proposition \ref{sdiff-gap}. 
\begin{lemma}\label{Normal-TM}
 Let $\varLambda\subset\mathcal{M}$ be a compact set. Then, there exist an $l\in\mathbb{N}^*$, points $\overline{x}^1,\ldots,\overline{x}^l\in\varLambda$, real numbers $\varepsilon_{\overline{x}^1}>0,\ldots,\varepsilon_{\overline{x}^l}>0$, and $\mathcal{C}^{2}$-smooth mappings $G_{\overline{x}^i}:\mathbb{X}\to \mathbb{Y}$ for $i\in[l]_{+}$ such that for each $i\in[l]_{+}$ and $z\in\mathbb{B}(\overline{x}^i,\varepsilon_{\overline{x}^i})$, $G_{\overline{x}^i}'(z):\mathbb{X}\to\mathbb{Y}$ is surjective, and for every $x\in\varLambda$, there exists an index $j\in[l]_{+}$ such that $x\in\mathcal{M}\cap\mathbb{B}(\overline{x}^j,\varepsilon_{\overline{x}^j}),N_{x}\mathcal{M}=\big\{\nabla G_{\overline{x}^j}(x) y\ |\ y\in\mathbb{Y}\big\}$ and 
 \begin{equation*}
  N_{(x,v)}T\mathcal{M}=\left\{\begin{pmatrix}
			\nabla G_{\overline{x}^j}(x)\xi+[D^2G_{\overline{x}^j}(x)v]^*\zeta\\
			\nabla G_{\overline{x}^j}(x)\zeta
  \end{pmatrix}\ |\ \xi\in\mathbb{Y},\zeta\in\mathbb{Y}\right\}\quad {\rm for}\ v\in T_{x}\mathcal{M}.
 \end{equation*}
\end{lemma}
\begin{proof}
 For each $u\in\mathcal{M}$, there exist $\varepsilon_{u}>0$ and a $\mathcal{C}^{2}$-smooth $G_{u}\!:\mathbb{X}\to \mathbb{Y}$ such that for each  $z\in\mathbb{B}(u,\varepsilon_{u})$, $G_{u}'(z)\!:\mathbb{X}\to \mathbb{Y}$ is surjective and $\mathcal{M}\cap\mathbb{B}(u,\varepsilon_{u})=\big\{z\in\mathbb{B}(u,\varepsilon_{u})\ |\ G_{u}(z)=0\big\}$. Note that $\varLambda\subset\bigcup_{u\in\varLambda}\mathbb{B}(u,\varepsilon_{u})$. By the Heine-Borel covering theorem, there exist an $l\in\mathbb{N}^*$, points $\overline{x}^1,\ldots,\overline{x}^l\in\varLambda$, and real numbers $\varepsilon_{\overline{x}^{1}}>0,\ldots,\varepsilon_{\overline{x}^{l}}>0$ such that $\varLambda\subset\bigcup_{i=1}^l\mathbb{B}(\overline{x}^{i},\varepsilon_{\overline{x}^{i}})$. Let $G_i:=G_{\overline{x}^i}$ for each $i\in[l]_{+}$. Pick any $x\in\varLambda\subset\mathcal{M}$. There is an index $j\in[l]_{+}$ such that 
 \begin{equation}\label{xequa-21}
 x\in\mathcal{M}\cap\mathbb{B}(\overline{x}^j,\varepsilon_{\overline{x}^j})=\big\{z\in\mathbb{B}(\overline{x}^j,\varepsilon_{\overline{x}^j})\,|\,G_{j}(z)=0\big\},
 \end{equation}
 where $G_{j}\!:\mathbb{X}\to\mathbb{Y}$ is a $\mathcal{C}^{2}$-smooth mapping with $G_{j}'(z)$ for $z\in\mathbb{B}(\overline{x}^j,\varepsilon_{\overline{x}^j})$ being surjective. From the openness of $\mathbb{B}(\overline{x}^j,\varepsilon_{\overline{x}^j})$,   $N_{x}\mathcal{M}=N_{\mathcal{M}\cap\mathbb{B}(\overline{x}^j,\varepsilon_{\overline{x}^j})}(x)=\big\{\nabla G_j(x) y\ |\ y\in\mathbb{Y}\big\}$, where $N_{\mathcal{M}\cap\mathbb{B}(\overline{x}^j,\varepsilon_{\overline{x}^j})}(x)$ is the normal cone to $\mathcal{M}\cap\mathbb{B}(\overline{x}^j,\varepsilon_{\overline{x}^j})$ at $x$. For $v\in T_{x}\mathcal{M}$, since $(x,v)\in T\mathcal{M}\cap[\varLambda\times\mathbb{X}]$, using \eqref{xequa-21} and the same arguments as for Lemma \ref{lemma-Tb} yields the result. 
\end{proof}
\begin{corollary}\label{corollary-normalTM}
 For any $\{(x^k,v^k)\}_{k\in\mathbb{N}}\subset T\mathcal{M}$ with bounded $\{x^k\}_{k\in\mathbb{N}}$, there exist a compact set $\Lambda\supset\{x^k\}_{k\in\mathbb{N}}$, an index $l\in\mathbb{N}^*$, points $\overline{x}^{1},\ldots,\overline{x}^{l}\in\varLambda$, real numbers $\varepsilon_{\overline{x}^{1}},\ldots,\varepsilon_{\overline{x}^{l}}>0$, and $\mathcal{C}^{2}$-smooth mappings $G_{\overline{x}^{i}}\!:\mathbb{X}\to \mathbb{Y}$ for $i\in[l]_{+}$ such that for each $i\in[l]_{+}$ and $z\in\mathbb{B}(\overline{x}^{i},\varepsilon_{\overline{x}^{i}})$, $G_{\overline{x}^{i}}'(z)\!:\mathbb{X}\to \mathbb{Y}$ is  surjective, and for each $k\in\mathbb{N}$ there is an index $j_k\in[l]_{+}$ such that 
 \begin{subnumcases}{}
  x^k\in\mathcal{M}\cap\mathbb{B}(\overline{x}^{j_k},\varepsilon_{\overline{x}^{j_k}}),\,N_{x^k}\mathcal{M}=\big\{\nabla G_{\overline{x}^{j_k}}(x^k) y\ |\ y\in\mathbb{Y}\big\},\\
  N_{(x^k,v^k)}T\mathcal{M}=\left\{\begin{pmatrix}
  \nabla G_{\overline{x}^{j_k}}(x^k)\xi+[D^2G_{\overline{x}^{j_k}}(x^k)v^k]^*\zeta\\
  \nabla G_{\overline{x}^{j_k}}(x^k)\zeta
 \end{pmatrix}\ |\ \xi\in\mathbb{Y},\zeta\in\mathbb{Y}\right\}.
 \end{subnumcases}
\end{corollary}
\subsection{Basic properties of retraction}\label{sec2.3}

Retraction is an approximation to the exponential mapping for a Riemannian manifold (see \cite[Definition 4.1]{Absil2008}), and is often used to retract a point on the tangent space of the manifold to the manifold. Its formal definition is stated as follows.
\begin{definition}\label{def-retract}
A smooth mapping $R:T\mathcal{M}\to\mathcal{M}$ is called a retraction if for each $x\in\mathcal{M}$ the restriction    $R_x(\cdot):=R(x,\cdot):T_x\mathcal{M}\to \mathcal{M}$  satisfies $R_x(0)=x$ and $R_x'(0)=\mathcal{I}$.
\end{definition}

The following lemma summarizes the basic properties of retraction. Since its proof is similar to the proofs of \cite[Eqs. (B.3)-(B.4)]{Boumal2019}, we here omit it. 
\begin{lemma}\label{lemma-retract}
 For any compact set $\varLambda\subset\mathcal{M}$, retraction $R$ of $\mathcal{M}$ and $\delta>0$, there exist constants $M_1>0$ and $M_2>0$ such that for all $x\in\varLambda$ and $v\in T_{x}\mathcal{M}\cap \overline{\mathbb{B}}(0,\delta)$, 
 \begin{equation*}
 \|R_{x}(v)-x\|\leq M_1\|v\|\ \ {\rm and}\ \ 
 \|R_{x}(v)-x-v\|\leq M_2\|v\|^2.
 \end{equation*}
\end{lemma}

Now we bound the difference between $\Theta(R_{\overline{x}}(v))$ and $\Theta(\overline{x}+v)$ for $(\overline{x},v)\in T\mathcal{M}$.
\begin{proposition}\label{prop-retract}
Consider a compact set $\varLambda\subset\mathcal{M}$ and a retraction $R$ of $\mathcal{M}$. Then, 
\begin{description}
\item[(i)] for any bounded set $V\subset\mathbb{X}$, if  $f$ and $F$ are continuously differentiable on an open convex set $\mathcal{O}'\supset\mathcal{M}\cup{\rm cl}\,(\varLambda+V)$, then there exists $\gamma>0$ such that for all $x\in\varLambda$ and $v\in T_x{\mathcal{M}}\cap V$, $|\Theta(R_{x}(v))-\Theta(x+v)|\le\!\gamma\|v\|^2$.


\item[(ii)] for each $\overline{x}\in\varLambda$, there exist $\varepsilon_{\overline{x}}>0$ and $M_{\overline{x}}>0$ such that for any $v\in  T_{\overline{x}}\mathcal{M}\cap\overline{\mathbb{B}}(0,\varepsilon_{\overline{x}})$ and $\alpha>M_{\overline{x}}\big[{\rm lip}\,\vartheta(F(\overline{x}))\,{\rm lip}\,F(\overline{x})+{\rm lip}\,f(\overline{x})\big]$,  
 $|\Theta(R_{\overline{x}}(v))-\Theta(\overline{x}\!+\!v)|\le \alpha\|v\|^2$. 
\end{description} 
\end{proposition}
\begin{proof}
{\bf(i)} Since $V$ is bounded, there exists $\delta>0$ such that $V\subset\overline{\mathbb{B}}(0,\delta)$.   Invoking Lemma \ref{lemma-retract} for this $\delta$, there exists $M_{V}>0$ such that for all $x\in\varLambda$ and $v\in T_{x}\mathcal{M}\cap V$, 
\begin{equation}\label{Rx-ineq21}
 \|R_{x}(v)-x-v\|\leq M_{V}\|v\|^2.
\end{equation}  
Let $B:={\rm cl}(\bigcup_{x\in\varLambda}R_{x}(V\cap T_{x}\mathcal{M}))\cup{\rm cl}(\varLambda+V)$. Obviously, $B\subset\mathcal{M}\cup{\rm cl}(\varLambda+V)\subset\mathcal{O}'$.
By the compactness of $\varLambda$ and the continuity of $R_{x}$ for each $x\in\mathcal{M}$, the set $B$ is a compact set in $\mathcal{O}'$. Since $f$ and $F$ are continuously differentiable on $\mathcal{O}'$, the functions $f$ and $\vartheta(F(\cdot))$ are Lipschitz continuous on the set $B$. Then, there exist constants $L_{f,B}>0$ and $L_{\vartheta,B}>0$ such that 
for all $x\in\varLambda$ and $v\in T_{x}\mathcal{M}\cap V$,
\begin{subequations}
\begin{align}\label{ineq-21a}
 &|f(R_{x}(v))-f(x+v)|\le L_{f,B}\|R_{x}(v)-x-\!v\|\le L_{f,B}M_{\varLambda,V}\|v\|^2,\\
 &|\vartheta(F(R_{x}(v)))-\vartheta(F(x+v))|\le L_{\vartheta,B}\|R_{x}(v)-x-\!v\|\le L_{\vartheta,B}M_{\varLambda,V}\|v\|^2, 
 \label{ineq-21b}
\end{align}
\end{subequations}
where the second inequality in \eqref{ineq-21a}-\eqref{ineq-21b} is obtained by using \eqref{Rx-ineq21}. Combining \eqref{ineq-21a}-\eqref{ineq-21b} with the definition of $\Theta$, the desired inequality holds with $\gamma=M_{\varLambda,V}(L_{f,B}+L_{\vartheta,B})$. 

\noindent
{\bf(ii)} Fix any $\epsilon>0$. According to Assumption \ref{ass0} {(i)}, there exists $\delta_1>0$ such that 
 \begin{equation*}
 |f(x)-f(x')|\le \big({\rm lip}\,f(\overline{x})+\epsilon\big)\|x-x'\|\quad\forall x,x'\in\mathbb{B}(\overline{x},\delta_1). 
 \end{equation*}
 From Assumption \ref{ass0} {(ii)-(iii)}, the function $\vartheta(F(\cdot))$ is locally Lipschitz continuous at $\overline{x}$, so there exists $\delta_2\in(0,\delta_1)$ such that for all $z,z'\in\mathbb{B}(\overline{x},\delta_2)$, 
 \begin{equation*}
  |\vartheta(F(z))-\vartheta(F(z'))|\le \big({\rm lip}\,\vartheta(F(\overline{x}))\!+\epsilon\big)\big({\rm lip}\,F(\overline{x})\!+\!\epsilon\big)\|z-z'\|.
 \end{equation*}
  From the continuity of $R_{\overline{x}}$, there exists $\varepsilon_{\overline{x}}\in(0,\delta_2)$ such that for all $v\in T_{\overline{x}}\mathcal{M}\cap\overline{\mathbb{B}}(0,\varepsilon_{\overline{x}})$, $R_{\overline{x}}(v)\in\mathbb{B}(\overline{x},\delta_2)$. Invoking Lemma \ref{lemma-retract} for $\delta=\varepsilon_{\overline{x}}$, there exists a constant $M_{\overline{x}}>0$ such that for all $v\in T_{\overline{x}}\mathcal{M}\cap\overline{\mathbb{B}}(0,\varepsilon_{\overline{x}})$, $\|R_{\overline{x}}(v)-\overline{x}-v\|\leq M_{\overline{x}}\|v\|^2$. Then, for any $v\in T_{\overline{x}}\mathcal{M}\cap\overline{\mathbb{B}}(0,\varepsilon_{\overline{x}})$, 
  \begin{align*}
  &|f(R_{\overline{x}}(v))-f(\overline{x}+v)|\le({\rm lip}\,f(\overline{x})+\epsilon)\|R_{\overline{x}}(v)-\overline{x}\!-\!v\|\le({\rm lip}\,f(\overline{x})+\epsilon)M_{\overline{x}}\|v\|^2,\\
  &|\vartheta(F(R_{\overline{x}}(v)))-\vartheta(F(\overline{x}+v))|\le({\rm lip}\,\vartheta(F(\overline{x}))+\epsilon)({\rm lip}\,F(\overline{x})+\epsilon)M_{\overline{x}}\|v\|^2. 
 \end{align*}
The desired result follows from the last two inequalities and the arbitrariness of $\epsilon>0$. 
\end{proof}

From the proof of Proposition \ref{prop-retract} (ii), we can obtain the following corollary. 
\begin{corollary}\label{corollary-retract}
 Consider a compact set $\varLambda\subset\mathcal{M}$, a retraction $R$ of $\mathcal{M}$ and $\delta>0$. For any $\overline{x}\in\varLambda$, if $f$ and $F$ are continuously differentiable on the set $\Gamma_{\overline{x},\delta}:=\overline{\mathbb{B}}(\overline{x},\delta)\cup R_{\overline{x}}( T_{\overline{x}}\mathcal{M}\cap\overline{\mathbb{B}}(0,\delta))$, then there exist $M_{\delta}>0$ and $\alpha>M_{\delta}(L_{\vartheta,\overline{x}}L_{F,\overline{x}}+L_{\!f,\overline{x}})$ such that for all $v\in T_{\overline{x}}\mathcal{M}\cap\overline{\mathbb{B}}(0,\delta)$, 
 \[
 |\Theta(R_{\overline{x}}(v))-\Theta(\overline{x}+\!v)|\le \alpha\|v\|^2, 
 \]
 where $L_{F,\overline{x}}$ and $L_{\!f,\overline{x}}$ are the Lipschitz constant of $F$ and $f$ on the set $\Gamma_{\overline{x},\delta}$, respectively, and $L_{\vartheta,\overline{x}}$ is the Lipschitz constant of $\vartheta$ on the set $F(\Gamma_{\overline{x},\delta})$.
\end{corollary}

\subsection{Kurdyka-Lojasiewicz property}\label{sec2.4}

To recall the KL property of a proper lsc function $h:\mathbb{X}\to\overline{\mathbb{R}}$, for every $\varpi>0$, we denote $\Upsilon_{\!\varpi}$ by the set of continuous concave functions $\varphi\!:[0,\varpi)\to\mathbb{R}_{+}$ that are continuously differentiable on $(0,\varpi)$ with $\varphi(0)=0$ and $\varphi'(t)>0$ for all $t\in(0,\varpi)$.
\begin{definition}\label{def-KL}
 A proper lsc $h\!:\mathbb{X}\to\overline{\mathbb{R}}$ is said to have the KL property at $\overline{x}\in{\rm dom}\,\partial h$ if there exist $\delta>0,\varpi\in(0,\infty]$ and $\varphi\in\Upsilon_{\!\varpi}$ such that for all $x\in\overline{\mathbb{B}}(\overline{x},\delta)\cap[h(\overline{x})<h<h(\overline{x})+\varpi]$, 
 \[
	\varphi'(h(x)\!-\!h(\overline{x})){\rm dist}(0,\partial h(x))\ge 1;
 \]
 and it is said to have the KL property of exponent $q\in[0,1)$ at $\overline{x}$ if there exist $c>0,\,\delta>0$ and $\varpi\in(0,\infty]$ such that for all $x\in\overline{\mathbb{B}}(\overline{x},\delta)\cap[h(\overline{x})<h<h(\overline{x})+\varpi]$, 
 \[
	c(1\!-\!q)\,{\rm dist}(0,\partial h(x))\ge (h(x)\!-\!h(\overline{x}))^{q}.
 \]
 If $h$ has the KL property (of exponent $q$) at every point of ${\rm dom}\,\partial h$, it is called a KL function (of exponent $q$).
\end{definition}

As discussed in \cite[Section 4]{Attouch2010}, the KL property is ubiquitous and the functions definable in an o-minimal structure over the real field admit this property. By \cite[Lemma 2.1]{Attouch2010}, to demonstrate that a proper lsc function has the KL property (of exponent $q\in[0,1)$), it suffices to check if the property holds at its critical points.


To close this section, we take a look at the influence of the linearization of $F$ at a point on the composite function $\vartheta(F(\cdot))$; see Lemma \ref{local-lip1}. To this end, we define 
\begin{equation}\label{ellF}
 \ell_{F}(x;z):=F(z)+F'(z)(x-z)\quad\ \forall x\in\mathbb{X},z\in\mathcal{M}.
\end{equation}
Since the proof of Lemma \ref{local-lip1} is direct by using Assumption \ref{ass0}, we here omit it.   

\begin{lemma} \label{local-lip1}
Consider any $\overline{x}\in\mathcal{M}$. Then, there exists $\overline{\delta}>0$ such that for any $x\in\mathbb{B}(\overline{x},\overline{\delta})$ and any $\alpha>{\rm lip}\,\vartheta(F(\overline{x}))\,{\rm lip}\,F'(\overline{x})$, $|\vartheta(F(x))-\vartheta(\ell_{F}(x;\overline{x}))|\le\frac{\alpha}{2}\|x-\overline{x}\|^2$.
\end{lemma}

\section{Riemannian inexact VMPL method}\label{sec3}

First, we introduce the idea of the upcoming Riemannian inexact VMPL method. Let $x^k\in\mathcal{M}$ be the current iterate. With the linearization $\ell_{F}(\cdot;x^k)$ of $F$ at $x^k$, we follow the same line as in \cite{tao2023inexact} to construct a local majorization of $\Theta$ at $x^k$. From Lemma \ref{local-lip1} with $\overline{x}=x^k$, for any $x$ sufficiently close to $x^k$ and any $\alpha_{1,k}>{\rm lip}\,\vartheta(F(x^k))\,{\rm lip}\,F'(x^k)$, 
\begin{equation}\label{ineq30-ThetaF}
\vartheta(F(x))\le \vartheta(\ell_{F}(x;x^k))+\frac{\alpha_{1,k}}{2}\|x-x^k\|^2.
\end{equation}
According to Assumption \ref{ass0} (i) and the decent lemma (see \cite[Lemma 5.7]{Beck2017}), for any $x$ close enough to $x^k$ and any $\alpha_{2,k}>{\rm lip}\,\nabla\!f(x^k)$,
\[
f(x)\le f(x^k)+\langle\nabla\!f(x^k),x-x^k\rangle+\frac{\alpha_{2,k}}{2}\|x-x^k\|^2.
\]
Combining the above two inequalities with the expression of $\Theta$, for any $x$ sufficiently close to $x^k$ and any $L_k>\overline{L}_k:={\rm lip}\,\vartheta(F(x^k))\,{\rm lip}\,F'(x^k)+{\rm lip}\,\nabla\!f(x^k)$, we have
\begin{equation}\label{Lk}
 \Theta(x)\le f(x^k)+\langle\nabla\!f(x^k),x-x^k\rangle+\vartheta(\ell_{F}(x;x^k))+\frac{L_k}{2}\|x-x^k\|^2.
\end{equation}
Thus, for any linear mapping $\mathcal{Q}_k\!:\mathbb{X}\to\mathbb{X}$ with $\mathcal{Q}_k\succ\overline{L}_k\mathcal{I}$ and $x$ close enough to $x^k$, 
\begin{equation*}
	\Theta(x)\le\widehat{\Theta}_k(x):=f(x^k)+\langle\nabla\!f(x^k),x-x^k\rangle+\vartheta(\ell_{F}(x;x^k))+\frac{1}{2}\|x-x^k\|_{\mathcal{Q}_k}^2,
\end{equation*} 
which along with $\widehat{\Theta}_k(x^k)=\Theta(x^k)$ implies that $\widehat{\Theta}_k$ is a local majorization of $\Theta$ at $x^k$.  

Now choose a linear operator $\mathcal{Q}_k:\mathbb{X}\to\mathbb{X}$ with $\mathcal{Q}_k\succ\overline{L}_k\mathcal{I}$ for the subproblem \eqref{init-subprob}. Note that $\Theta_k(\cdot-x^k)=\widehat{\Theta}_k(\cdot)$, so the above discussion means that $\Theta_k(\cdot-x^k)$ is a local majorization of $\Theta$ at $x^k$. Inspired by this and Proposition \ref{prop-retract} (ii), our method first seeks a direction $v^k\in T_{x^k}\mathcal{M}$ satisfying a certain decrease by solving \eqref{init-subprob} inexactly, and then moves the iterate $x^k$ along the direction $v^k$ in a full step and retracts $x^k+v^k$ onto the manifold $\mathcal{M}$ with a retraction $R$. The resulting $R_{x^k}(v^k)$ serves as the new iterate. Consider that for some $f$ and $F$ the threshold $\overline{L}_k$ is usually unavailable. Our method at each iteration formulates a PD linear operator $\mathcal{Q}_k$ by searching for a tight upper estimation of $\overline{L}_k$ and meanwhile solves inexactly the associated problem \eqref{init-subprob}. Among others, an implementable criterion is proposed for the inexact solution $v^k\in T_{x^k}\mathcal{M}$ of \eqref{init-subprob}. The inexact solving of strongly convex subproblems, constructed with a proximal linearization of $\Theta$, forms the cornerstone of our method. This interprets why it is called a Riemannian inexact VMPL method. Its iteration steps are described as follows. 
\begin{algorithm}[h]
 \begin{algorithmic}[1]
 \caption{\label{iRVM}{\bf (RiVMPL method for problem \eqref{prob})}}
 \State Input: a retraction $R$, parameters $0<\alpha_{\rm min}\le\alpha_{\rm max},\,\mu_{\rm max}>0,\,\overline{\alpha}>0,\,\overline{\gamma}>0,\sigma>1$ \hspace*{1.0cm} and a starting point $x^0\in \mathcal{M}$. 
		
 \For{$k=0,1,2,\ldots$} 
 		
 \State Choose $\alpha_{k,0}\in[\alpha_{\rm min}, \alpha_{\rm max}]$ and $\mu_k\in(0,\mu_{\rm max}]$.
		
       \For {$j=0,1,2,\ldots$}
		
		\State {Choose a linear operator $\mathcal{Q}_{k,j}$ with $\alpha_{k,j}\mathcal{I}\preceq \mathcal{Q}_{k,j}\preceq (\overline{\alpha}+\alpha_{k,j})\mathcal{I}$. Compute
		\begin{equation}\label{subprobj}
		\qquad\min_{v\in T_{x^k}\mathcal{M}}\Theta_{k,j}(v)\!:=\langle\nabla\!f(x^k), v \rangle+\frac{1}{2}\|v\|_{\mathcal{Q}_{k,j}}^2+\vartheta(\ell_{F}(x^k\!+\!v;x^k))+f(x^k)
		\end{equation}
		\hspace*{0.92cm} to seek a pair $(v^{k,j},\Theta_{k,j}^{\rm LB})\in T_{x^k}\mathcal{M}\times(-\infty,\Theta_{k,j}(\overline{v}^{k,j})]$ such that
		\begin{equation}\label{inexact-cond}
		\Theta_{k,j}(v^{k,j})\leq \Theta_{k,j}(0)\ \ {\rm and}\ \ \Theta_{k,j}(v^{k,j})-\Theta_{k,j}^{\rm LB}\leq \frac{\mu_k}{2}\|v^{k,j}\|^2,
		\end{equation}
        \hspace*{1.0 cm}where $\overline{v}^{k,j}$ denotes the unique optimal solution of \eqref{subprobj}.  
		}
		\State If $\Theta(R_{x^k}(v^{k,j}))\le\! \Theta_{k,j}(v^{k,j})\!-\!\frac{\overline{\gamma}}{2}\|v^{k,j}\|^2$ or $\|v^{k,j}\|\!=0$, go to Step 8; else set \hspace*{1.1cm}$\alpha_{k,j+1}:=\sigma\alpha_{k,j}$.		
		\EndFor
        
 \State If $\|v^{k,j}\|=0$, stop; else set $j_k:=j,v^k:=v^{k,j_k},x^{k+1}:=R_{x^k}(v^{k,j_k}),\overline{v}^k:=\overline{v}^{k,j_k}$,  
 \hspace*{0.4cm} $\mathcal{Q}_k:=\mathcal{Q}_{k,j_k}$ and $\Theta_k:=\Theta_{k,j_k}$.
 \EndFor
\end{algorithmic}
\end{algorithm}
\begin{remark}\label{remark-alg}
{\bf(a)} By Lemma \ref{well-def} later, Algorithm \ref{iRVM} is well defined, i.e., the inner for-end loop stops within a finite number of steps. The inner loop aims at capturing a tight upper estimation of $\overline{L}_k$ to formulate the PD linear operator $\mathcal{Q}_{k}$ and seeking an inexact minimizer of the associated subproblem \eqref{subprobj} with a certain decrease. As will be seen in Section \ref{sec6}, the linear operator $\mathcal{Q}_{k,j}$ allows us to incorporate the first-order information of $F$ at $x^k$ into the term $\frac{1}{2}\|v\|_{\mathcal{Q}_{k,j}}^2$, which can potentially accelerate the computation of subproblems. The subproblems of ManPL \cite{Wang2022} involve a proximal term $\frac{1}{2t}\|v\|^2$ for a constant $t>0$. When $\overline{L}_k$ is unavailable, a big gap between $t^{-1}$ and $\overline{L}_k$ may occur, which will lead to a bad direction. 
	
\noindent 
{\bf(b)} The inexactness criterion in \eqref{inexact-cond} consists of two conditions. The first one aims to ensure that $v^{k,j}$ for large enough $j$ are descent directions of the objective function $\Theta$, by recalling that  $\Theta(x^k\!+\!v^{k,j})\le\Theta_{k,j}(v^{k,j})$ for such $j$ and $\Theta_{k,j}(0)=\Theta(x^k)$. The second one is intended to control the inexactness tolerance of $v^{k,j}$ since it implies $\Theta_{k,j}(v^{k,j})-\Theta_{k,j}(\overline{v}^{k,j})\le\frac{\mu_k}{2}\|v^{k,j}\|^2$. By the strong duality for strongly convex programs, a dual or primal-dual method is especially adapted to seek a pair $(v^{k,j},\Theta_{k,j}^{\rm LB})\in T_{x^k}\mathcal{M}\times(-\infty,\Theta(\overline{v}^{k,j})]$ satisfying \eqref{inexact-cond}. The details can be found in Section \ref{sec6}. Consequently, the proposed inexactness criterion is easily implementable. 
	
\noindent
 {\bf(c)} For some $k\in\mathbb{N}$ and $j\in[j_k]$, $v^{k,j}=0$ if and only if $\overline{v}^{k,j}=0$, which by Remark \ref{remark-spoint} (a) means that  $x^{k,j}$ is now a stationary point of \eqref{prob}. Indeed, $v^{k,j}=0$ implies $\overline{v}^{k,j}=0$ by Proposition \ref{prop-xk} (i) later; conversely, if $\overline{v}^{k,j}=0$, the first inequality in \eqref{inexact-cond} becomes $\Theta_{k,j}(v^{k,j})\le \Theta_{k,j}(\overline{v}^{k,j})$, so $v^{k,j}=0$ follows by the strong convexity of $\Theta_{k,j}$ and the optimality of $\overline{v}_{k,j}$. This explains why Algorithm \ref{iRVM} stops when $\|v^{k,j}\|=0$ for some $k\in\mathbb{N}$ and $j\in[j_k]$.
 
 \noindent
 {\bf(d)} For each $k\in\mathbb{N}$ and $j\in[j_k]$, from the optimality of $\overline{v}^{k,j}$ to the subproblem \eqref{subprobj}, we have
\begin{equation}\label{optim-cond}
0\in\nabla\!f(x^k)+\mathcal{Q}_{k,j}\overline{v}^{k,j}+\nabla F(x^k)\partial\vartheta(\ell_F(x^k\!+\overline{v}^{k,j};x^k))+N_{x^k}\mathcal{M}.
\end{equation}
This optimality condition will be used in the subsequent complexity and convergence analysis. 
\end{remark}

\begin{lemma}\label{well-def}
 Suppose that $v^{k,j}\ne 0$ for all $k\in\mathbb{N}$ and $j\in\mathbb{N}$. The following assertions hold.
 \begin{description}
 \item[(i)] For each $k\in\mathbb{N}$ and $j\in\mathbb{N}$, any $(u,\Theta_{k,j}^{\rm LB})\in T_{x^k}\mathcal{M}\times(-\infty,\Theta_{k,j}(\overline{v}^{k,j})]$ sufficiently close to $(\overline{v}^{k,j},\Theta_{k,j}(\overline{v}^{k,j}))$ satisfies the criterion \eqref{inexact-cond}.  

 \item[(ii)] For each $k\in\mathbb{N}$, the inner for-end loop stops within a finite number of steps.
 \end{description}	
\end{lemma}
\begin{proof}
{\bf(i)} Fix any $k,j\in\mathbb{N}$. By Remark \ref{remark-alg} (c), $\overline{v}^{k,j}\ne 0$. The strong convexity of $\Theta_{k,j}$ implies that $\Theta_{k,j}(\overline{v}^{k,j})<\Theta_{k,j}(0)$, so for any $v\in T_{x^k}\mathcal{M}$ close enough to $\overline{v}^{k,j}$, $\Theta_{k,j}(v)<\Theta_{k,j}(0)$ follows the continuity of $\Theta_{k,j}$. In addition, consider the function $h_{k,j}(v):=\Theta_{k,j}(v)-\frac{\mu_k}{2}\|v\|^2$ for $v\in\mathbb{X}$. Then, for those $\Theta_{k,j}^{\rm LB}\in[\Theta_{k,j}(\overline{v}^{k,j})\!-\frac{\mu_k}{4}\|\overline{v}^{k,j}\|^2,\Theta_{k,j}(\overline{v}^{k,j})]$, 
it holds
\[
 h_{k,j}(\overline{v}^{k,j})-\Theta_{k,j}^{\rm LB}\le h_{k,j}(\overline{v}^{k,j})-\Theta_{k,j}(\overline{v}^{k,j})+\frac{\mu_k}{4}\|\overline{v}^{k,j}\|^2=-\frac{\mu_k}{4}\|\overline{v}^{k,j}\|^2<0.
 \]
 Then the continuity of $h_{k,j}$ implies $h_{k,j}(v)-\Theta_{k,j}^{\rm LB}<0$ for any $v\in T_{x^k}\mathcal{M}$ close enough to $\overline{v}^{k,j}$.  
 
 \noindent
 {\bf(ii)} Suppose on the contrary that the conclusion does not hold for some $k\in\mathbb{N}$. From Step 6 of Algorithm \ref{iRVM} and $\sigma>1$, it follows that $\lim_{j\to\infty}\alpha_{k,j}=\infty$ and for all $j\in\mathbb{N}$,
 \begin{equation}\label{aim-ineq31}
 \Theta(R_{x^k}(v^{k,j}))>\Theta_{k,j}(v^{k,j})-\frac{\overline{\gamma}}{2}\|v^{k,j}\|^2.
 \end{equation}
 For each $j\in\mathbb{N}$, according to the inexactness criterion for $v^{k,j}$, we immediately obtain
 \begin{equation*}
 \Theta_{k,j}(0)\ge\Theta_{k,j}(v^{k,j})=\langle\nabla\!f(x^{k}), v^{k,j} \rangle+\dfrac{1}{2}\|v^{k,j}\|_{\mathcal{Q}_{k,j}}^2+\vartheta(\ell_{F}(x^{k}\!+\!v^{k,j};x^{k}))+f(x^{k}).
 \end{equation*}
 Note that $\vartheta$ is bounded from below by an affine function due to Assumption \ref{ass0} (ii). The above inequality, along with $\Theta_{k,j}(0)<\infty,\mathcal{Q}_{k,j}\succeq\alpha_{k,j}\mathcal{I}$ and $\lim_{j\to\infty}\alpha_{k,j}=\infty$, implies that $\lim_{j\to{\infty}}v^{k,j}=0$. By \eqref{Lk}, there exist $L_{k}>\overline{L}_k$ and some $\overline{j}\in\mathbb{N}$ such that for each $j\ge\overline{j}$,
 \begin{align*}
  \Theta(x^{k}\!+\!v^{k,j})&\le f(x^{k})+\langle\nabla\!f(x^{k}),v^{k,j}\rangle+\vartheta(\ell_{F}(x^{k}\!+\!v^{k,j};x^{k}))+\frac{1}{2}L_k\|v^{k,j}\|^2\\
  &\le \Theta_{k,j}(v^{k,j})+\frac{1}{2}(L_{k}- \alpha_{k,j})\|v^{k,j}\|^2.
 \end{align*}
 In addition, from Step 5 of Algorithm \ref{iRVM}, $v^{k,j}\in T_{x^k}\mathcal{M}$ for all $j\in\mathbb{N}$. Since $x^k\in\mathcal{M}\subset\mathcal{O}$, there exists $\delta_k>0$ such that $\overline{\mathbb{B}}(x^k,\delta_k)\subset\mathcal{O}$. Then, applying Proposition \ref{prop-retract} (i) with $\varLambda=\{x^k\}$, $V=\overline{\mathbb{B}}(0,\delta_k)$ and $\mathcal{O}'=\mathcal{O}$ and using $\lim_{j\to{\infty}}v^{k,j}=0$, there exists $\widehat{\alpha}_{k}>0$ such that for all $j\ge\overline{j}$ (if necessary by increasing $\overline{j}$),
 \[
  \Theta(R_{x^{k}}(v^{k,j}))
  \le\Theta(x^{k}+v^{k,j})+\frac{\widehat{\alpha}_{k}}{2}\|v^{k,j}\|^2. 
 \]
 Combining the above two inequalities with \eqref{aim-ineq31} yields $(\alpha_{k,j}-L_{k}-\widehat{\alpha}_{k}-\overline{\gamma})\|v^{k,j}\|^2< 0$ for all $j\ge\overline{j}$, which is impossible by recalling that $\lim_{j\to\infty}\alpha_{k,j}=\infty$. The proof is completed.
\end{proof} 

Lemma \ref{well-def} (ii) states that, as long as the current iterate $x^k$ is not a stationary point, the inner loop of Algorithm \ref{iRVM} necessarily stops within a finite number of steps. The specific number of steps can be quantified when $f$ and $F$ are $\mathcal{C}^{1,1}$ on a larger set. 
\begin{lemma}\label{step-num}
 Fix any $k\in\mathbb{N}$ for $v^k\ne 0$. Suppose that  $f$ and $F$ are $\mathcal{C}^{1,1}$ on an open convex set containing $\mathcal{M}\cup\overline{\mathbb{B}}(x^k,\delta_k)$ with $\delta_k:=\max_{z\in\mathcal{X}_k}\|z-x^k\|$. Then, the inner loop stops once
 \begin{equation}\label{jineq}
  j>\big\lceil(\log\sigma)^{-1}\log[\alpha_{k,0}^{-1}(\overline{\gamma}+L_{\vartheta,k}(L_{\nabla F, k}\!+\!2M_kL_{F,k})+L_{\nabla\!f,k}(1\!+\!2M_{k}))]\big\rceil,
 \end{equation}  
 where $L_{f,k}, L_{\nabla\!f,k}, L_{F,k}$ and $L_{\nabla\!F, k}$ are respectively the Lipschitz constant of $f,\nabla\!f,F$ and $\nabla F$ on  $\widehat{\mathcal{X}}_k\!:=\overline{\mathbb{B}}(x^k,\delta_k)\cup R_{x^k}(\mathcal{X}_k\!-\!x^k)$ with $\mathcal{X}_k\!:=x^k+\{v\in\mathbb{X}\ |\ v\in T_{x^k}\mathcal{M},q_{k}(v)\leq q_k(0)\}$, where 
 \[
	q_k(v):=\langle\nabla\!f(x^k), v \rangle+\frac{\alpha_{\min}}{2}\|v\|^2+\vartheta(\ell_{F}(x^k\!+v;x^k))+f(x^k), 
 \]
 and $L_{\vartheta,k}$ is the Lipschitz constant of $\vartheta$ on the set $F(\widehat{\mathcal{X}}_k)$, and $M_{k}$ is the constant from Corollary \ref{corollary-retract} with $\varLambda=\{x^k\},\overline{x}=x^k$ and $\delta=\delta_k$. 
\end{lemma}
\begin{proof}
Since the function $q_k$ is coercive, the set $\mathcal{X}_k$ is compact by its definition, so is the set $\widehat{\mathcal{X}}_k$ by the continuity of $R_{x^k}$. Since $R_{x^k}(\mathcal{X}_k\!-\!x^k)\subset\mathcal{M}$, the given assumption on $f$ and $F$ imply that $f,\nabla\!f,F$ and $\nabla F$ are Lipshitz continuous on $\widehat{\mathcal{X}}_k$, while $\vartheta$ is Lipschitz continuous on $F(\widehat{\mathcal{X}}_k)$. Then, the constants $L_{f,k},L_{\nabla\!f,k},L_{F,k},L_{\nabla F, k}$ and $L_{\vartheta,k}$ are well defined. Note that $x^k+v^{k,j}\in\mathcal{X}_k\subset\widehat{\mathcal{X}}_k$ for all $j\in\mathbb{N}$ by the iteration of Algorithm \ref{iRVM}. From the Lipschitz continuity of $\vartheta(F(\cdot))$ on $\widehat{\mathcal{X}}_k$, similar to the conclusion of Lemma \ref{local-lip1}, for each $j\in\mathbb{N}$, we have 
\begin{equation}\label{vtheta-ineq31}
 \vartheta(F(x^k\!+v^{k,j}))\le \vartheta(\ell_{F}(x^k\!+v^{k,j};x^k))+\frac{1}{2}L_{\vartheta,k}L_{\nabla F, k}\|v^{k,j}\|^2. 
\end{equation}
By the definition of $\Theta$ and the descent lemma for $f$ (see \cite[Theorem 5.7]{Beck2017}), for each $j\in\mathbb{N}$,  
\begin{align*} 
\Theta(x^k\!+v^{k,j})
 &\le f(x^k)+\langle\nabla\!f(x^k),v^{k,j}\rangle+\frac{1}{2}L_{\nabla\!f,k}\|v^{k,j}\|^2+\vartheta(F(x^k\!+v^{k,j}))\\ 
 &\le f(x^k)+\langle\nabla\!f(x^k),v^{k,j}\rangle+\!\vartheta(\ell_{F}(x^k\!+v^{k,j};x^k))+\!\frac{L_{\vartheta,k}L_{\nabla F, k}\!+\!L_{\nabla\!f,k}}{2}\|v^{k,j}\|^2\\
 &\le\Theta_{k,j}(v^{k,j})+\!\frac{1}{2}(L_{\vartheta,k}L_{\nabla F, k}\!+\!L_{\nabla\!f,k})\|v^{k,j}\|^2-\!\frac{1}{2}\sigma^j\alpha_{k,0}\|v^{k,j}\|^2,
 \end{align*}
 where the second inequality follows from \eqref{vtheta-ineq31} and the third inequality is due to the expression of $\Theta_{k,j}$ and $\mathcal{Q}_{k,j}\succeq\alpha_{k,j}\mathcal{I}$. Notice that $v^{k,j}\in T_{x^k}\mathcal{M}\cap \overline{\mathbb{B}}(0,\delta_k)$. From Corollary \ref{corollary-retract} with $\varLambda=\{x^k\},\overline{x}=x^k$ and $\delta=\delta_k$, it follows
 \begin{equation*}
  \Theta(R_{x^{k}}(v^{k,j}))
	\le\Theta(x^{k}\!+\!v^{k,j})+M_{k}(L_{\vartheta,k}L_{F,k}+L_{\nabla\!f,k})\|v^{k,j}\|^2\quad\forall j\in\mathbb{N}.
 \end{equation*}
 Putting the above two inequalities together, for each $j\in\mathbb{N}$, it holds 
 \begin{equation*}
 \Theta(R_{x^{k}}(v^{k,j}))\leq \Theta_{k,j}(v^{k,j})-\frac{\sigma^j\alpha_{k,0}-[L_{\vartheta,k}(L_{\nabla F, k}\!+\!2M_kL_{F,k})+L_{\nabla\!f,k}(1\!+\!2M_{k})]}{2}\|v^{k,j}\|^2.
 \end{equation*}
 When $j$ satisfies \eqref{jineq}, $\Theta(R_{x^k}(v^{k,j}))\leq \Theta_{k,j}(v^{k,j})-\frac{\overline{\gamma}}{2}\|v^{k,j}\|^2$, i.e., the inner loop stops. 
\end{proof}

Now we present useful properties of the sequences $\{(v^k,\overline{v}^k)\}_{k\in\mathbb{N}}$ and $\{\Theta(x^k)\}_{k\in\mathbb{N}}$. 
\begin{proposition}\label{prop-xk}
 Let $\{(x^k,v^k,\overline{v}^k)\}_{k\in \mathbb{N}}$ be the sequence generated by Algorithm \ref{iRVM}. Then,
 \begin{description}
 \item[(i)] for each $k\in\mathbb{N}$ and $j\in[j_k]$, $\|v^{k,j}-\overline{v}^{k,j}\|\le(\alpha_{\rm min}^{-1}\mu_{\rm max})^{1/2}\|v^{k,j}\|$;

 \item[(ii)] for each $k\in \mathbb{N}$, $\Theta(x^{k+1})\le\Theta(x^k)-\frac{\overline{\gamma}}{2}\|v^k\|^2$, so the sequence $\{\Theta(x^k)\}_{k\in\mathbb{N}}$ converges to a limit, denoted by $\varsigma^*$;

 \item[(iii)] $\lim_{k\to\infty} v^k=0$ and $\lim_{k\to\infty}\overline{v}^k=0$. 
 \end{description}	
\end{proposition}
\begin{proof}
 {\bf(i)} 
 Fix any $k\in\mathbb{N}$ and $j\in[j_k]$. Note that $0\in\partial(\Theta_{k,j}+\delta_{T_{x^{k}}\mathcal{M}})(\overline{v}^{k,j})=\partial\Theta_{k,j}(\overline{v}^{k,j})+N_{x^{k}}\mathcal{M}$, so there exists $\xi^{k,j}\in\partial\Theta_{k,j}(\overline{v}^{k,j})\cap N_{x^{k}}\mathcal{M}$. Together with $v^{k,j},\overline{v}^{k,j}\in T_{x^{k}}\mathcal{M}$, the strong convexity of  $\Theta_{k,j}+\delta_{T_{x^{k}}\mathcal{M}}$ and the second inequality in \eqref{inexact-cond}, it follows
 \begin{align*}\label{ineq-sconvex}
 \frac{1}{2}\|v^{k,j}-\overline{v}^{k,j}\|^2_{\mathcal{Q}_{k,j}}&=\langle\xi^{k,j},v^{k,j}-\overline{v}^{k,j}\rangle+\frac{1}{2}\|v^{k,j}-\overline{v}^{k,j}\|^2_{\mathcal{Q}_{k,j}}\\
 &\le \Theta_{k,j}(v^{k,j})-\Theta_{k,j}(\overline{v}^{k,j})\le\Theta_{k,j}(v^{k,j})-\Theta_{k,j}^{\rm LB}\leq\frac{1}{2}\mu_{\rm max}\|v^{k,j}\|^2.
 \end{align*} 
 The desired result follows from the last inequality by noting that $\mathcal{Q}_{k,j}\succeq\alpha_{\rm min}\mathcal{I}$. 
	
 \noindent
 {\bf(ii)-(iii)} From Steps 6 and 8 of Algorithm \ref{iRVM}, for each $k\in\mathbb{N}$, it holds
 \begin{equation*}
 \Theta(x^{k+1})=\Theta(R_{x^k}(v^{k,j_k}))\leq \Theta_k(v^{k,j_k})-\frac{\overline{\gamma}}{2}\|v^{k,j_k}\|^2,
\end{equation*}
 which along with $\Theta_k(v^{k,j_k})\leq \Theta_k(0)=\Theta(x^k)$ by the first inequality of  \eqref{inexact-cond} implies that
 \begin{equation*}
 \Theta(x^{k+1})\le\Theta(x^k) -\frac{\overline{\gamma}}{2}\|v^{k,j_k}\|^2=\Theta(x^k)-\frac{\overline{\gamma}}{2}\|v^k\|^2\quad\forall k\in\mathbb{N}.
 \end{equation*}
 This shows that $\{\Theta(x^k)\}_{k\in\mathbb{N}}$ is nonincreasing, so is convergent by Assumption \ref{ass0} (iv). Then, the last inequality implies that $\lim_{k\to\infty} v^k=0$, and $\lim_{k\to\infty}\overline v^k=0$ holds by item (i). 
\end{proof}

In the next two sections, we establish the iteration complexity for Algorithm \ref{iRVM} and analyze the full convergence of its iterate sequence  under the following assumption: 
\begin{assumption}\label{ass1}
 The sequence $\{x^k\}_{k\in \mathbb{N}}\subset\mathcal{M}$ is bounded, so the closedness of $\mathcal{M}$ implies the existence of a compact set $\varLambda_0\subset\mathcal{M}$ such that $\{x^k\}_{k\in \mathbb{N}}\subset\varLambda_0$. 
\end{assumption}
Assumption \ref{ass1} is rather mild and automatically holds if the manifold $\mathcal{M}$ is compact. It is implied by the boundedness on the level set $\{x\in\mathcal{M}\ |\ \Theta(x)\le\Theta(x^0)\}$, a common assumption in the literature on manifold optimization (see, e.g., \cite{Huang2022a,Huang2023}).
\begin{proposition}\label{bound-alpk}
 Under Assumption \ref{ass1}, the following assertions hold. 
\begin{description}

\item[(i)] There exists $b_{v}>0$ such that for all $k\in\mathbb{N}$ and $j\in[j_k]$, $\max\{\|v^{k,j}\|,\|\overline{v}^{k,j}\|\}\le b_{v}$.

\item[(ii)] There is $b_{\alpha}>0$ such that for all $k$ and $j\in[j_k]$, $\alpha_{k,j}\le b_{\alpha}$, so $\|\mathcal{Q}_{k,j}\|\le\alpha^*\!:= b_{\alpha}+\overline{\alpha}$. 
\end{description}	
\end{proposition}
\begin{proof}
 {\bf(i)} For each $k\in\mathbb{N}$ and $j\in[j_k]$, from the first inequality of \eqref{inexact-cond}, we have
 \begin{equation}\label{temp-ineq41}
 q_k(v^{k,j})\le\Theta_{k,j}(v^{k,j})\le\Theta_{k,j}(0)=\Theta(x^k)\le\Theta(x^0), 
 \end{equation}
 where $q_k(\cdot)$ is the function defined in Lemma \ref{step-num}. In view of Assumption \ref{ass0} (ii), for each $k\in\mathbb{N}$, there exists $\zeta^k\in\partial\vartheta(F(x^k))$ such that for all $j\in[j_k]$, 
 \begin{equation}\label{subdiff-vtheta}
 \vartheta(\ell_F(x^k+v^{k,j};x^k))\ge\vartheta(F(x^k))+\langle\nabla F(x^k)\zeta^k,v^{k,j}\rangle.    
 \end{equation}
 Together with the expression of $q_k$ and the above \eqref{temp-ineq41}, for each $k\in\mathbb{N}$ and $j\in[j_k]$, it holds
 \[
 \langle\nabla\!f(x^k)+\nabla F(x^k)\zeta^k,v^{k,j}\rangle+\frac{1}{2}\alpha_{\rm min}\|v^{k,j}\|^2+\vartheta(F(x^k))+f(x^k)\le\Theta(x^0).
 \]
 Since the multivalued mapping $\partial\vartheta:\mathbb{Z}\rightrightarrows\mathbb{Z}$ is locally bounded by \cite[Theorem 9.13 (d)]{RW98}, the boundedness of $\{\zeta^k\}_{k\in\mathbb{N}}$ follows item (i) and \cite[Proposition 5.15]{RW98}. Then, the function $\mathbb{X}\ni v\mapsto\langle\nabla\!f(x^k)+\nabla F(x^k)\zeta^k,v\rangle+\frac{\alpha_{\rm min}}{2}\|v\|^2$ is coercive. Thus, from the last inequality and Assumptions \ref{ass1} and \ref{ass0},  there exists $b_{v}>0$ such that $\|v^{k,j}\|\le b_{v}$ for all $k\in\mathbb{N}$ and $j\in[j_k]$. The result follows Proposition \ref{prop-xk} (i) by enlarging $b_{v}$ if necessary.
	
 \noindent
 {\bf(ii)} Suppose on the contrary that the conclusion does not hold. There exist an infinite index set $\mathcal{K}\subset\mathbb{N}$ and an index sequence $\{i_k\}_{k\in\mathcal{K}}$ with $1\le i_k\le j_k$ for all $k\in\mathcal{K}$ such that $\lim_{\mathcal{K}\ni k\to\infty}\alpha_{k,i_k}=\infty$. For each $k\in\mathbb{N}$, let $\widehat{\alpha}_{k}:=\alpha_{k,i_k-1}$. From Step 6 of Algorithm \ref{iRVM}, 
 \begin{equation}\label{aim-ineq41}
 \Theta(R_{x^{k}}(v^{k,i_k-1}))>\Theta_{k,i_k-1}(v^{k,i_k-1})-\frac{1}{2}\overline{\gamma}\|v^{k,i_k-1}\|^2\quad\forall k\in\mathbb{N}.
 \end{equation}
 For each $k\in\mathbb{N}$, from the first inequality in \eqref{inexact-cond}, $\mathcal{Q}_{k,i_k-1}\succeq\widehat{\alpha}_k\mathcal{I}$ and the above \eqref{subdiff-vtheta}, we get
 \begin{align*}
 \Theta(x^k)&=\Theta_{k,i_k-1}(0)\ge \Theta_{k,i_k-1}(v^{k,i_k-1})\nonumber\\
 &\ge \vartheta(\ell_{F}(x^{k}\!+\!v^{k,i_k-1};x^{k}))-\|\nabla f(x^k)\|\|v^{k,i_k-1}\|+\frac{1}{2}\widehat{\alpha}_{k}\|v^{k,i_k-1}\|^2+f(x^{k})\nonumber\\
 &\ge\vartheta(F(x^k))-\big(\|\nabla\!f(x^k)\|+\|\nabla F(x^k)\zeta^k\|\big)\|v^{k,i_k-1}\|+\frac{1}{2}\widehat{\alpha}_{k}\|v^{k,i_k-1}\|^2+f(x^k).
 \end{align*}
where the third inequality follows from \eqref{subdiff-vtheta}. Since the sequences $\{x^k\}_{k\in\mathbb{N}},\{\zeta^k\}_{k\in\mathbb{N}}$ and $\{\Theta(x^k)\}_{k\in\mathbb{N}}$ are bounded, the last inequality and $\lim_{\mathcal{K}\ni k\to\infty}\widehat{\alpha}_{k}=\infty$ imply that $\lim_{\mathcal{K}\ni k\to \infty}v^{k,i_k-1}=0$. Now invoking the previous \eqref{Lk} with $x=x^{k}+v^{k,i_k-1}$ for sufficiently large $k$ and recalling $\mathcal{Q}_{k,i_k-1}\succeq\widehat{\alpha}_k\mathcal{I}$, it follows 
 \begin{align*}
 \Theta(x^{k}\!+\!v^{k,i_k-1})&\le f(x^{k})\!+\!\langle\nabla\!f(x^{k}),v^{k,i_k-1}\rangle+\vartheta(\ell_{F}(x^{k}\!+\!v^{k,i_k-1};x^{k}))+\frac{1}{2}(\overline{L}_k\!+\!1)\|v^{k,i_k-1}\|^2\\
 &\le \Theta_{k,i_k-1}(v^{k,i_k-1})+\frac{1}{2}(\overline{L}_k\!+\!1-\widehat{\alpha}_k)\|v^{k,i_k-1}\|^2.
 \end{align*}
From Assumption \ref{ass1},  $\varLambda_0\subset\mathcal{M}\subset\mathcal{O}$, there exists $\delta>0$ such that ${\rm cl}(\varLambda_0+\mathbb{B}(0,\delta))\subset\mathcal{O}$. Using $\lim_{\mathcal{K}\ni k\to \infty}v^{k,i_k-1}=0$ and applying Proposition \ref{prop-retract} (i) with $\varLambda=\varLambda_0, V=\mathbb{B}(0,\delta)$ and $\mathcal{O}'=\mathcal{O}$, there exists a constant $\gamma>0$ such that for sufficiently large $k\in\mathcal{K}$,
 \begin{equation*}
 \Theta(R_{x^{k}}(v^{k,i_k-1}))\leq\Theta(x^{k}\!+\!v^{k,i_k-1})+\gamma\|v^{k,i_k-1}\|^2.
 \end{equation*}
 From the last two inequalities, it immediately follows that for sufficiently large $k\in\mathcal{K}$,
 \[
  \Theta(R_{x^{k}}(v^{k,i_k-1}))\le \Theta_{k,i_k-1}(v^{k,i_k-1})-\frac{1}{2}(\widehat{\alpha}_k-\overline{L}_k-1-2\gamma)\|v^{k,i_k-1}\|^2.
 \]
 The boundedness of $\{x^k\}_{k\in\mathbb{N}}\subset\mathcal{M}\subset\mathcal{O}$ and \cite[Theorem 9.2]{RW98} imply that the sequences $\{{\rm lip}\,f(x^k)\}_{k\in\mathbb{N}}$, $\{{\rm lip}\,\nabla\!f(x^k)\}_{k\in\mathbb{N}}$, $\{{\rm lip}\,F(x^k)\}_{k\in\mathbb{N}}$, $\{{\rm lip}\,F'(x^k)\}_{k\in\mathbb{N}}$ and $\{{\rm lip}\,\vartheta(F(x^k))\}_{k\in\mathbb{N}}$ are all bounded, so is the sequence $\{\overline{L}_k\}_{k\in\mathcal{K}}$. Recalling that $\lim_{\mathcal{K}\ni k\to\infty}\widehat{\alpha}_{k}=\infty$, the last inequality yields a contradiction to \eqref{aim-ineq41}. The desired conclusion holds.
\end{proof} 
\section{Iteration complexity}\label{sec4}

This section establishes the iteration complexity of Algorithm~\ref{iRVM} for finding an \(\epsilon\)-stationary point. To this end, we modify Algorithm~\ref{iRVM} by replacing its outer termination condition in Step~8 with \(\|v^{k,j}\|\le\chi\epsilon\), and its inner-loop stopping criterion in Step~6 by \(\Theta(R_{x^k}(v^{k,j}))\le \Theta_{k,j}(v^{k,j})-\frac{\overline{\gamma}}{2}\|v^{k,j}\|^2\) or \(\|v^{k,j}\|\le\chi\epsilon\), where \(\chi\) is the constant defined in Theorem \ref{complexity} below. Note that these modified criteria are employed solely for the theoretical analysis of this section rather than the practical implementation of Algorithm~\ref{iRVM} in Section \ref{sec7}, as \(\chi \) depends on the unknown parameters \(\alpha ^{*}\) and \(c_{\nabla F}\). 
\begin{theorem}\label{complexity}
 Let $c_{\nabla\!F}\!:=\sup\limits_{k\in\mathbb{N}}\|F'(x^k)\|$ and $\chi\!:=\frac{[\max\{\alpha^*,c_{\nabla\!F}\}]^{-1}}{(\alpha_{\rm min}^{-1}\mu_{\rm max})^{1/2}+1}$. Under Assumption \ref{ass1}, 
 \begin{description}
 \item[(i)] if $\|v^{k,j}\|\le\chi\epsilon$
 for some $k\in\mathbb{N}$ and $j\in[j_k]$, then $x^k$ is an $\epsilon$-stationary point of \eqref{prob};

 \item[(ii)] Algorithm \ref{iRVM} returns an $\epsilon$-stationary point within at most $K:=\lceil \frac{2(\Theta(x^0)-\varsigma^*)}{\overline{\gamma}\chi^2\epsilon^2}\rceil$ steps.
 \end{description}
\end{theorem}
\begin{proof}
{\bf(i)} Let $z^{k,j}:=\ell_F(x^k+\overline{v}^{k,j};x^k)$. From \eqref{optim-cond}, there exists  $\xi^{k,j}\in\partial\vartheta(z^{k,j})$ such that 
 \[
  \|\Pi_{T_{x^k}\mathcal{M}}(\nabla\!f(x^k)+\nabla F(x^k)\xi^{k,j})\|=\|\Pi_{T_{x^k}\mathcal{M}}(\mathcal{Q}_{k,j}\overline{v}^{k,j})\|\le \|\mathcal{Q}_{k,j}\overline{v}^{k,j}\|\le \alpha^*\|\overline{v}^{k,j}\|,
 \]
 where the second inequality is due to Proposition \ref{bound-alpk} (ii). In addition, by the expression of $z^{k,j}$, we have $\|F(x^k)-z^{k,j}\|\le\|F'(x^k)\|\|\overline{v}^{k,j}\|\le c_{\nabla\!F}\|\overline{v}^{k,j}\|$. The two sides imply that  
 \begin{align}\label{vbarkj-vkj}
 &\max\big\{\|\Pi_{T_{x^k}\mathcal{M}}(\nabla\!f(x^k)+\nabla F(x^k)\xi^{k,j})\|,\|F(x^k)-z^{k,j}\|\big\}\nonumber\\ 
 &\le\max\{\alpha^*,c_{\nabla\!F}\}\|\overline{v}^{k,j}\|\le\chi^{-1}\|v^{k,j}\|,
 \end{align}
 where the second inequality is due to Proposition \ref{prop-xk} (i). The result then follows.
 
\noindent
 {\bf(ii)} We argue that Algorithm \ref{iRVM} returns an iterate $x^k$ with the associated $v^k$ satisfying $\|v^{k}\|\le\chi\epsilon$ within at most $K$ iterations. If not, $\|v^{k}\|>\chi\epsilon$ for all $k\in[K]$. From Proposition \ref{prop-xk} (ii), 
 \begin{equation*}
 (K\!+\!1)\chi^2\epsilon^2<\sum_{k=0}^{K}\|v^k\|^2\le 2\overline{\gamma}^{-1}\sum_{k=0}^{K}\big[\Theta(x^k)-\Theta(x^{k+1})\big]\le 2\overline{\gamma}^{-1}(\Theta(x^0)-\varsigma^*).
 \end{equation*}
 Then, $K\!+\!1<2[\overline{\gamma}\chi^2]^{-1}(\Theta(x^0)-\varsigma^*)\epsilon^{-2}$, a contradiction to the definition of $K$. 
\end{proof} 

For each $k\in\mathbb{N}$, let $q_{k},\mathcal{X}_k$ and $\widehat{\mathcal{X}}_k$ be the same as in Lemma \ref{step-num}. Under Assumption \ref{ass1}, it is easy to argue that the set $\bigcup_{k\in\mathbb{N}}\big\{v\in\mathbb{X}\ |\ q_k(v)\le \Theta(x^0)\big\}$ is bounded, so is the set $\bigcup_{k\in\mathbb{N}}\big\{v\in\mathbb{X}\ |\ q_k(v)\le q_k(0)\big\}$ because $q_k(0)=\Theta(x^k)\le\Theta(x^0)$ by Proposition \ref{prop-xk} (ii). Then, the set $\bigcup_{k\in\mathbb{N}}\mathcal{X}_k$ is bounded, so is $\widehat{\mathcal{X}}:={\rm cl}\bigcup_{k\in\mathbb{N}}\widehat{\mathcal{X}}_k$. With these notation, we present the complexity of calls to the inner solver in the following corollary.  
\begin{corollary}\label{inner-complexity}
 Suppose $f$ and $F$ are $\mathcal{C}^{1,1}$ on the space $\mathbb{X}$. Under Assumption \ref{ass1}, Algorithm \ref{iRVM} returns an $\epsilon$-stationary point within at most $j_{\rm max}K$ calls to the subproblem solver with
 \begin{equation*}
 j_{\rm max}:=\big\lceil(\log\sigma)^{-1}\log[\alpha_{\rm min}^{-1}(\overline{\gamma}+\widehat{L}_{\vartheta}\widehat{L}_{\nabla\!F}\!+\!\widehat{L}_{\nabla\!f}+2\widehat{M})]\big\rceil,
 \end{equation*}
 where  $\widehat{L}_{f},\widehat{L}_{\nabla\!f},\widehat{L}_{F}$ and $\widehat{L}_{\nabla\!F}$ are the Lipschitz constant of $f,\nabla\!f,F$ and $\nabla F$ on the compact set $\widehat{\mathcal{X}}$, $\widehat{L}_{\vartheta}$ is the Lipschitz constant  of $\vartheta$ on the set $F(\widehat{\mathcal{X}})$, and $\widehat{M}$ is the constant from Proposition \ref{prop-retract} (i) with $\varLambda=\varLambda_0$ in Assumption \ref{ass1} and $V:=\sup_{k\in\mathbb{N}}\max_{z\in{\rm cl}\bigcup_{k\in\mathbb{N}}\mathcal{X}_k}\|z-x^k\|$.
\end{corollary} 
\begin{proof}
For each $k\in\mathbb{N}$ and $j\in[j_k]$, from the proof of Lemma \ref{step-num} and the given assumption, 
\begin{equation*}
 \Theta(x^k\!+v^{k,j})\le \Theta_{k,j}(v^{k,j})+\!\frac{1}{2}(\widehat{L}_{\vartheta}\widehat{L}_{\nabla\!F}\!+\!\widehat{L}_{\nabla\!f})\|v^{k,j}\|^2-\!\frac{1}{2}\sigma^j\alpha_{\rm min}\|v^{k,j}\|^2.
\end{equation*}
Using $v^{k,j}\in T_{x^k}\mathcal{M}\cap V$ and  Proposition \ref{prop-retract} (i) with $\varLambda=\varLambda_0, \mathcal{O}'=\mathbb{X}$ and the set $V$ yields 
\begin{equation*}
 \Theta(R_{x^{k}}(v^{k,j}))	\le\Theta(x^{k}\!+\!v^{k,j})+\widehat{M}\|v^{k,j}\|^2.
\end{equation*}
The last two inequalities together imply that for each $k\in\mathbb{N}$ and $j\in[j_k]$, 
\begin{equation*}
 \Theta(R_{x^{k}}(v^{k,j}))\leq \Theta_{k,j}(v^{k,j})-\frac{1}{2}[\sigma^j\alpha_{\rm min}-(\widehat{L}_{\vartheta}\widehat{L}_{\nabla\!F}\!+\!\widehat{L}_{\nabla\!f}+2\widehat{M})]\|v^{k,j}\|^2.
 \end{equation*}
 Thus, for each $k\in\mathbb{N}$, when $j>j_{\rm max}$, $\Theta(R_{x^k}(v^{k,j}))\leq \Theta_{k,j}(v^{k,j})-\frac{\overline{\gamma}}{2}\|v^{k,j}\|^2$, i.e., the inner loop stops. Along with Theorem \ref{complexity} (ii), the complexity of calls to the inner solver follows.  
\end{proof}

When the PD linear operator $\mathcal{Q}_{k,j}$ is specified as in Theorem \ref{oracle} below and the subproblems are solved with the dual first-order methods in \cite{Necoara2016}, Algorithm \ref{iRVM} can return an $\epsilon$-stationary point with at most $O(\epsilon^{-4})$ calls oracles. Here, we assume that we have access to oracles that compute $f(x),\nabla f(x),F(x),F'(x)v,\nabla F(x)w$, the projection mapping $\Pi_{T_{x}\mathcal{M}}(\xi)$, and $\mathcal{P}_{\gamma\vartheta}(z)$ for all $x,v,\xi\in\mathbb{X}$ and $z,w\in\mathbb{Z}$. Such an oracle complexity is consistent with that of the primal-dual RADMM in Li et al. \cite{LiMa2024} applied to \eqref{prob} except that every iteration of the latter calls a retraction but does not require $F'(x)v$, while Algorithm \ref{iRVM} calls a retraction only at each iteration of the outer loop. 
\begin{theorem}\label{oracle}
 Suppose that $f,F$ are $\mathcal{C}^{1,1}$ on $\mathbb{X}$ and Assumption \ref{ass1} holds, and that $\Theta^{\rm LB}_{k,j}$ for each $k\in\mathbb{N}$ and $j\in[j_k]$ is the dual objective value returned by Algorithm DFG in \cite{Necoara2016} applied to solve the subproblems with a starting point from a bounded set $Z\subset\mathbb{Z}$ for yielding the last primal iterate (resp. the average primal iterate). Then, Algorithm \ref{iRVM} with $\mu_k\ge\mu_{\rm min}>0$ and $\mathcal{Q}_{k,j}=\alpha_{k,j}\mathcal{I}+\beta_k\nabla F(x^k)F'(x^k)$ for a positive bounded $\{\beta_k\}_{k\in\mathbb{N}}$ returns an $\epsilon$-stationary point by calling the oracles at most $O(\epsilon^{-4})$ (resp. at most $O(\epsilon^{-3})$) times.  
\end{theorem}
\begin{proof}
 For each $k\in\mathbb{N}$ and $j\in[j_k]$, since $\mathcal{Q}_{k,j}=\alpha_{k,j}\mathcal{I}+\beta_k\nabla F(x^k)F'(x^k)$, the dual of \eqref{subprobj} is
 \begin{equation}\label{dualkj}
 \Theta_{k,j}(\overline{v}^{k,j})=\max_{\zeta\in\mathbb{Z}}\Psi_{k,j}(\zeta)\ \ {\rm with}\ \ \Psi_{k,j}(\zeta):=\min_{v\in T_{x^k}\mathcal{M},z\in\mathbb{Z}}\mathcal{L}_{k,j}(v,z,\zeta)
 \end{equation}
 where $\mathcal{L}_{k,j}$ is the Lagrange function of \eqref{subprobj} and for $(v,z,\zeta)\in\mathbb{X}\times\mathbb{Z}\times\mathbb{Z}$ it has the form 
\begin{equation}\label{la-func}
    \mathcal{L}_{k,j}(v,z,\zeta):=\langle\nabla\!f(x^k),v\rangle+\frac{\alpha_{k,j}}{2}\|v\|^2+\frac{\beta_k}{2}\|z-F(x^k)\|^2+\vartheta(z)+\langle\zeta,\ell_F(x^k+v;x^k)-z\rangle.
\end{equation}
 For each $k\in\mathbb{N}$ and $j\in[j_k]$, let $\overline{\zeta}^{k,j}$ be an arbitrary optimal solution of \eqref{dualkj}, and let $\overline{v}^{k,j}=-\alpha_{k,j}^{-1}\Pi_{T_{x^k}\mathcal{M}}(\nabla F(x^k)\overline{\zeta}^{k,j}\!+\!\nabla\!f(x^k))$ and  $\overline{z}^{k,j}=\mathcal{P}_{\!\beta_{k}^{-1}\vartheta}(F(x^k)+\beta_{k}^{-1}\overline{\zeta}^{k,j})$. It is easy to check $(\overline{v}^{k,j},\overline{z}^{k,j})\in\mathop{\arg\min}_{v\in T_{x^k}\mathcal{M},z\in\mathbb{Z}}\mathcal{L}_{k,j}(v,z,\overline{\zeta}^{k,j})$, whose optimality condition implies 
 \begin{equation}\label{temp-optcond}
 \overline{\zeta}^{k,j}\in\beta_k(\overline{z}^{k,j}-F(x^k))+\partial\vartheta(\overline{z}^{k,j}).
 \end{equation}
 Furthermore, $\overline{v}^{k,j}$ is the optimal solution of \eqref{subprobj}. We claim that there exists $c_{\!d}^*$ such that $\|\overline{\zeta}^{k,j}\|\le c_{\!d}^*$ for all $k\in\mathbb{N}$ and $j\in[j_k]$. Indeed, from $\nabla\Psi_{k,j}(\overline{\zeta}^{k,j})=0$, it follows that $\overline{z}^{k,j}=F'(x^k)\overline{v}^{k,j}+F(x^k)$. Notice that $\sup_{k\in\mathbb{N}}\sup_{j\in[j_k]}\|\overline{v}^{k,j}\|\le b_{v}$ by Proposition \ref{bound-alpk} (i). By Assumption \ref{ass1}, there exists $\widetilde{b}_{v}>0$ such that $\sup_{k\in\mathbb{N}}\sup_{j\in[j_k]}\|\overline{z}^{k,j}\|\le c_{\nabla\!F}b_{v}+\!\|F(x^k)\|\le \widetilde{b}_{v}$. Thus, from the inclusion \eqref{temp-optcond}, the local boundedness of $\partial\vartheta$ and \cite[Proposition 5.15]{RW98}, the claimed $c_{\!d}^*$ necessarily exists. Now, for each $k\in\mathbb{N}$ and $j\in[j_k]$, by letting $\mathcal{Y}_{k,j}^*$ be the solution set of the dual problem \eqref{dualkj}, we have 
 $\sup_{\zeta\in Z}\|\zeta-\Pi_{\mathcal{Y}_{k,j}^*}(\zeta)\|\le c_{Z}+c_{\!d}^*\ \ {\rm with}\ \  c_{Z}:=\sup_{\zeta\in Z}\|\zeta\|$. 
 
 From \cite[Theorem 1]{Nesterov2005}, the Lipschitz constant of $\nabla\Psi_{k,j}$ is $\sigma_{k,j}^{-1}\|[F'(x^k)\ \ \mathcal{I}]\|^2$, where $\sigma_{k,j}$ is the strongly convex modulus of $\Theta_{k,j}$. Obviously, $\sigma_{k,j}\ge\alpha_{\rm min}$. For each $k\in\mathbb{N}$ and $j\in[j_k]$, invoking \cite[Eqs. (39),(45)]{Necoara2016} with $x^0=\zeta^{1}\in Z,L_{d}=\sigma_{k,j}^{-1}\|[F'(x^k)\ \ \mathcal{I}]\|^2\le\frac{(1+c_{\nabla\!F})^2}{\alpha_{\rm min}}$ and $R_{d}=\|\zeta^{1}\!-\!\Pi_{\mathcal{Y}_{k,j}^*}(\zeta^1)\|\le c_{Z}\!+\!c_{\!d}^*$, we conclude that Algorithm DFG returns $v^{k,j}\in T_{x^k}\mathcal{M}$ and $\Theta_{k,j}^{\rm LB}<\Theta_{k,j}(\overline{v}^{k,j})$ such that $\Theta_{k,j}(v^{k,j})-\Theta^{\rm LB}_{k,j}\le \epsilon_{k,j}$ within at most $\big\lceil\frac{6(c_{Z}+c_{d}^*)^2(1+c_{\nabla\!F})^2}{\alpha_{\rm min}\epsilon_{k,j}}+\frac{\sqrt{2}(1+c_{\nabla F})(c_{Z}+c_{d}^*)}{\sqrt{\alpha_{\rm min}\epsilon_{k,j}}}\big\rceil$ iterations. Notice that 
 the inexactness condition \eqref{inexact-cond} holds once $\Theta_{k,j}(v^{k,j})-\Theta^{\rm LB}_{k,j}\le\epsilon_{k,j}$ with $\epsilon_{k,j}=\min\{\Theta_{k,j}(0)-\Theta^{\rm LB}_{k,j},\frac{\mu_k}{2}\|v^{k,j}\|^2\}$. Further, by the strong convexity of $\Theta_{k,j}$, it holds
 \begin{align*}
 \Theta_{k,j}(0)-\Theta^{\rm LB}_{k,j}\ge \Theta_{k,j}(0)-\Theta_{k,j}(\overline{v}^{k,j})\ge \frac{1}{2}\alpha_{\min}\|\overline{v}^{k,j}\|^2.
 \end{align*}
Together with $\mu_k\ge\mu_{\rm min}$, it follows that Algorithm DFG returns $v^{k,j}$ satisfying \eqref{inexact-cond} within at most $\frac{12(c_{Z}+c_{\!d}^*)^2(1+c_{\nabla\!F})^2}{\min\{\alpha_{\rm min}^2\|\overline{v}^{k,j}\|^2,\alpha_{\rm min}{\mu_{\rm min}}\|v^{k,j}\|^2\}}+\frac{2(c_{Z}+c_{\!d}^*)(1+c_{\nabla\!F})}{\sqrt{\min\{\alpha_{\rm min}^2\|\overline{v}^{k,j}\|^2,{\alpha_{\rm min}\mu_{\rm min}}\|v^{k,j}\|^2\}}}$ iterations. Now suppose that Algorithm \ref{iRVM} first returns an $\epsilon$-stationary point at the $K$th step. By Definition \ref{def-spoint} (ii), $\max\big\{\|\Pi_{T_{x^k}\mathcal{M}}(\nabla\!f(x^k)+\nabla F(x^k)\xi^{k,j})\|,\|F(x^k)-z^{k,j}\|\big\}>\epsilon$ for each $k\in[K\!-\!1]$ and $j\in[j_k]$. Along with the previous \eqref{vbarkj-vkj} and the definition of $\chi$, it follows that $\min\{\|v^{k,j}\|,\|\overline{v}^{k,j}\|\}>\chi\epsilon$ for each $k\in[K\!-\!1]$ and $j\in[j_k]$. Let $J_{k}:=\{j\in[j_k]\ |\ \alpha_{\rm min}\|\overline{v}^{k,j}\|^2\le\mu_{\rm min}\|v^{k,j}\|^2\}$ and $\overline{J}_k=[j_k]\backslash J_k$ for each $k\in\mathbb{N}$. Then, 
 \begin{align*}
  &\sum_{k=0}^{K-1}\sum_{j=0}^{j_k}\frac{12(c_{Z}\!+\!c_{\!d}^*)^2(1\!+\!c_{\nabla\!F})^2}{\min\{\alpha_{\rm min}^2\|\overline{v}^{k,j}\|^2,\alpha_{\rm min}\mu_{\rm min}\|v^{k,j}\|^2\}}+\frac{2(c_{Z}\!+\!c_{\!d}^*)(1\!+\!c_{\nabla\!F})}{\sqrt{\min\{\alpha_{\rm min}^2\|\overline{v}^{k,j}\|^2,{\alpha_{\rm min}\mu_{\rm min}}\|v^{k,j}\|^2\}}}\\
 &\le \sum_{k=0}^{K-1}\Big[|J_k|\frac{12(c_{Z}\!+\!c_{\!d}^*)^2(1+c_{\nabla\!F})^2}{\alpha_{\rm min}^2\chi^2\epsilon^2}+|\overline{J}_{\!k}|\frac{12(c_{Z}\!+\!c_{\!d}^*)^2(1+c_{\nabla\!F})^2}{\alpha_{\rm min}\mu_{\rm min}\chi^2\epsilon^2}\Big]\\
 &\quad +\sum_{k=0}^{K-1}\Big[|J_k|\frac{2(c_{Z}+c_{\!d}^*)(1\!+\!c_{\nabla\!F})}{\alpha_{\rm min}\chi\epsilon}+|\overline{J}_{\!k}|\frac{2(c_{Z}\!+\!c_{\!d}^*)(1\!+\!c_{\nabla\!F})}{\sqrt{\alpha_{\rm min}\mu_{\rm min}}\chi\epsilon}\Big]\\
 &\le \frac{12Kj_{\rm max}(c_{Z}\!+\!c_{\!d}^*)^2(1\!+\!c_{\nabla\!F})^2}{\min\{\alpha_{\rm min}^2,\alpha_{\rm min}\mu_{\rm min}\}\chi^2\epsilon^2}+\frac{2Kj_{\rm max}(c_{Z}\!+\!c_{\!d}^*)(1\!+\!c_{\nabla\!F})}{\min\{\alpha_{\rm min},\sqrt{\alpha_{\rm min}\mu_{\rm min}}\}\chi\epsilon},
\end{align*}
 where the last inequality is due to Corollary \ref{inner-complexity}. The result follows Theorem \ref{complexity} (ii).  
 
 When $\Theta^{\rm LB}_{k,j}$ for each $k\in\mathbb{N}$ and $j\in[j_k]$ is the dual objective value returned by Algorithm DFG in \cite{Necoara2016} applied to solve the subproblems with a starting point from $Z$ for yielding the average primal iterate, invoking \cite[Eqs. (39), (56)-(57)]{Necoara2016} and following the same arguments as above yields the desired oracle complexity. Thus, we complete the proof.
\end{proof}
\section{Convergence analysis}\label{sec5}

Taking the limit as $\mathcal{K}\ni k\to\infty$ in the inclusion \eqref{optim-cond} and using the outer semicontinuity of $\partial\vartheta$ and $\partial\delta_{\mathcal{M}}$ and Proposition \ref{prop-xk} (iii), we get the subsequential convergence result. 
\begin{theorem}\label{theorem1}
 Under Assumption \ref{ass1}, the set of cluster points of $\{x^k\}_{k\in \mathbb{N}}$, denoted as $\Omega(x^0)$, is nonempty, and every $x^*\in\Omega(x^0)$ is a stationary point of problem \eqref{prob}.
\end{theorem}

The rest of this section focuses on the full convergence of $\{x^k\}_{k\in\mathbb{N}}$ under the KL framework. This requires to construct an appropriate potential function. To this end,
let $\{e_1,\ldots,e_p\}$ be an orthonormal base of $\mathbb{X}$ and $\{d_1,\ldots,d_r\}$ an orthonormal base of $\mathbb{Z}$, and define the linear mappings $\mathcal{E}:\mathbb{R}^p\to\mathbb{X}$ and $\mathcal{F}:\mathbb{R}^r\to\mathbb{Z}$ by $\mathcal{E}\widehat{x}:=\sum_{i=1}^{p}\widehat{x}_ie_i$ for $\widehat{x}\in\mathbb{R}^p$ and $\mathcal{F}\widehat{z}:=\sum_{i=1}^{r}\widehat{z}_id_i$ for $\widehat{z}\in\mathbb{R}^r$. Then, for any $x\in \mathbb{X}$ (resp. $z\in\mathbb{Z}$), there exists a unique $\widehat{x}\in \mathbb{R}^p$ (resp. $\widehat{z}\in \mathbb{R}^r$) such that $x=\mathcal{E}\widehat{x}$ (resp. $z=\mathcal{F}\widehat{z}$). Define the linear mapping $\mathcal{E}^{\dag}\!:\mathbb{X}\to \mathbb{R}^p$ by $\mathcal{E}^{\dag}x=(\langle e_1,x\rangle,\ldots,\langle e_p,x\rangle)^{\top}\in\mathbb{R}^p$. Obviously, $\mathcal{E}\mathcal{E}^{\dag}=\mathcal{I}$. Then, under the given bases of $\mathbb{X}$ and $\mathbb{Z}$, any linear mapping $\mathcal{A}\!:\mathbb{X}\to\mathbb{Z}$ has a unique matrix representation  $A\!:=\mathcal{F}^{\dag}\mathcal{A}\mathcal{E}=[\mathcal{F}^{\dag}\mathcal{A}e_1\ \cdots\ \mathcal{F}^{\dag}\mathcal{A}e_p]\in\mathbb{R}^{r\times p}$, where $\mathcal{F}^{\dag}:\mathbb{Z}\to\mathbb{R}^r$ is the linear mapping defined in the same way as for $\mathcal{E}^{\dag}$. Thus, for any $v\in\mathbb{X}$, $\mathcal{A}v=\mathcal{F}\mathcal{F}^{\dag}\mathcal{A}\mathcal{E}\mathcal{E}^{\dag}v=\mathcal{F}A\mathcal{E}^{\dag}v$. In the sequel, for each $k\in\mathbb{N}$, let $Q^k\in\mathbb{R}^{r\times p}$ be the unique matrix representation of $F'(x^k):\mathbb{X}\to\mathbb{Z}$ under the above bases, so that 
\begin{equation}\label{gradFk}
 F'(x^k)v=\mathcal{F}Q^k\mathcal{E}^{\dag}v\ \ {\rm for}\ v\in\mathbb{X}\ \ {\rm and}\ \ \nabla F(x^k)z=(\mathcal{E}^{\dag})^{*}(Q^k)^{\top}\mathcal{F}^*z\ \ {\rm for}\ z\in\mathbb{Z}.
\end{equation}

Inspired by the structure of the objective function of subproblem \eqref{subprobj}, we define
\begin{equation}\label{Xi-fun}
\Xi(w)\!:=f(x)+\langle s,v\rangle+\vartheta(F(x)\!+\mathcal{F}Q\mathcal{E}^{\dag}v)+\delta_{T\mathcal{M}}(x,v)+\!\frac{\alpha^*}{2}\|v\|^2
\end{equation}
for $w=(x,v,Q,s)\in\mathbb{W}:=\mathbb{X}\times\mathbb{X}\times\mathbb{R}^{r\times p}\times\mathbb{X}$, 
where $\alpha^*$ is the same as in Proposition \ref{bound-alpk} (ii). Such  $\Xi$ can make the subsequent full convergence analysis independent of the twice continuous differentiability of $F$. Due to the manifold constraint, the potential function $\Xi$ does not enjoy decrease. This precludes using the recipe developed in \cite{Attouch2013,Bolte2014} and the analysis technique in \cite{tao2023inexact} to achieve the convergence of $\{x^k\}_{k\in\mathbb{N}}$. We get around this difficulty by building the bridge between the objective value of \eqref{prob} and the value of the potential function at the iterates. In the rest of this section, for each $k\in\mathbb{N}$, let $w^k\!:=(x^k,\overline{v}^k,Q^k,\nabla\!f(x^k))$. Then, $\{(x^k,\overline{v}^k)\}_{k\in\mathbb{N}}\subset\mathcal{M}\times T_{x^k}\mathcal{M}$ and $\{Q^k\}_{k\in\mathbb{N}}$ is bounded under Assumption \ref{ass1}. Furthermore, the first equality in \eqref{gradFk} implies that $\ell_{F}(x^k+\overline{v}^k;x^k)=F(x^k)+\mathcal{F}Q^k\mathcal{E}^{\dag}\overline{v}^k$, and consequently, for each $k\in\mathbb{N}$, 
\begin{equation}\label{Xiwk}
\Xi(w^k)=f(x^k)+\langle\nabla\!f(x^k),\overline{v}^k\rangle+\vartheta(\ell_{F}(x^k\!+\!\overline{v}^k;x^k))+\frac{\alpha^*}{2}\|\overline{v}^k\|^2.
\end{equation}
The following lemma discloses the relationship between $\Xi(w^k)$ with $\Theta(x^{k+1})$. 
\begin{lemma}\label{lemma-XiTheta}
 Under Assumption \ref{ass1}, there exist a compact convex set $\Gamma\subset\mathcal{O}$ with $\{x^k\}_{k\in\mathbb{N}}\subset\Gamma$, a compact set $D\subset\mathbb{Z}$ with $D\supset\{\ell_{F}(x^k\!+v^k;x^k)\}_{k\in\mathbb{N}}\cup\{\ell_{F}(x^{k+1};x^k)\}_{k\in\mathbb{N}}$, and an index $\overline{k}\in\mathbb{N}$ such that with $\widetilde{c}:=(1+b_{\vartheta}b_{\nabla\!F}+L_{\nabla\!f})M_1^2+2M_2(c_{\nabla\!f}\!+c_{\nabla\!F}L_{\vartheta})+\mu_{\rm max}$,
 \[
	\Theta(x^{k+1})\le\Xi(w^k)+\frac{\widetilde{c}}{2}\|v^k\|^2\quad{\rm for\ all}\ k\ge\overline{k},
 \]  
 where $L_{\nabla\!f}$ and $L_{\vartheta}$ are the Lipschitz constant of $\nabla\!f$ and $\vartheta$ on the sets $\Gamma$ and $D$, respectively, 
 \begin{equation}\label{ineq-const}
 b_{\vartheta}:=\sup_{k\in\mathbb{N}}{\rm lip}\,\vartheta(F(x^k)),\,b_{\nabla\!F}:=\sup_{k\in\mathbb{N}}{\rm lip}\,F'(x^k)\ \ {\rm and}\ \ c_{\nabla\!f}:=\sup_{k\in\mathbb{N}}\|\nabla\!f(x^k)\|, 
\end{equation}
and $M_1$ and $M_2$ are the constants from Lemma \ref{lemma-retract} for $\varLambda=\varLambda_0$ in Assumption \ref{ass1} and $\delta=1$.
\end{lemma}
\begin{proof}
 Since the sequence $\{x^k\}_{k\in\mathbb{N}}\subset\mathcal{M}\subset\mathcal{O}$ is bounded by Assumption \ref{ass1}, there exists a compact convex set $\Gamma\subset\mathcal{O}$ such that $\{x^k\}_{k\in\mathbb{N}}\subset\Gamma$.
 Let $D\subset\mathbb{Z}$ be a compact set containing $\{\ell_{F}(x^k\!+\!v^k;x^k)\}_{k\in\mathbb{N}}$ and $\{\ell_{F}(x^{k+1};x^k)\}_{k\in\mathbb{N}}$. Obviously, such a set $D$ exists by the boundedness of the two sequences. Along with Assumption \ref{ass0}, the constants $L_{\nabla\!f}$ and $L_{\vartheta}$ are well defined. In addition, from \cite[Theorem 9.2]{RW98} and Assumption \ref{ass0}, the sequences $\{{\rm lip}F'(x^k)\}_{k\in\mathbb{N}}$ and $\{{\rm lip}\,\vartheta(F(x^k))\}_{k\in\mathbb{N}}$ are bounded from above, so the constants in \eqref{ineq-const} are well defined. From the definition of $\Theta$ and the descent lemma for $f$, for each $k\in\mathbb{N}$, 
 \begin{equation*}
 \Theta(x^{k+1})\le f(x^k)+\langle\nabla\!f(x^k),x^{k+1}\!-x^k\rangle+\frac{L_{\nabla\!f}}{2}\|x^{k+1}\!-x^k\|^2+\vartheta(F(x^{k+1})).
 \end{equation*} 
From the second inequality of \eqref{inexact-cond} and Proposition \ref{bound-alpk} (ii), for each $k\in\mathbb{N}$, it holds 
\begin{align*}
 0\le\Theta_{k}(\overline{v}^k)-\Theta_k(v^k)+\frac{\mu_k}{2}\|v^k\|^2&\le\langle\nabla\!f(x^k),\overline{v}^k\!-\!v^k\rangle+\frac{\alpha^*}{2}\|\overline{v}^k\|^2+\frac{\mu_k}{2}\|v^k\|^2\\
 &\quad +\vartheta(\ell_{F}(x^k\!+\overline{v}^k;x^k))-\vartheta(\ell_{F}(x^k\!+v^k;x^k)).
 \end{align*}
 Combining the above two inequalities with equation \eqref{Xiwk}, for each $k\in\mathbb{N}$, we have
 \begin{align}\label{ineq42-Theta}
 \Theta(x^{k+1})&\le f(x^k)+\langle\nabla\!f(x^k),x^{k+1}-x^k\rangle+\frac{L_{\nabla\!f}}{2}\|x^{k+1}-x^k\|^2+\vartheta(F(x^{k+1}))+\frac{\mu_k}{2}\|v^k\|^2\nonumber\\
 &\quad+\langle\nabla\!f(x^k),\overline{v}^k\!-\!v^k\rangle+\frac{\alpha^*}{2}\|\overline{v}^k\|^2+\vartheta(\ell_{F}(x^k\!+\overline{v}^k;x^k))-\vartheta(\ell_{F}(x^k\!+v^k;x^k))\nonumber\\
 &=\Xi(w^k)+\langle\nabla\!f(x^k),x^{k+1}\!-\!x^k\!-\!v^k\rangle+\frac{L_{\nabla\!f}}{2}\|x^{k+1}\!-\!x^k\|^2\nonumber\\
 &\quad\ +\vartheta(F(x^{k+1}))-\vartheta(\ell_{F}(x^k\!+v^k;x^k))+\frac{{\mu_k}}{2}\|v^k\|^2.
 \end{align}
 Since $\lim_{k\to\infty}v^k=0$ by Proposition \ref{prop-xk} (iii) and $x^{k+1}=R_{x^k}(v^k)$ for all $k$, invoking Lemma \ref{lemma-retract} with $\varLambda=\varLambda_0$ and $\delta=1$, there exist $M_1>0,M_2>0$ and  $\overline{k}\in\mathbb{N}$ such that for all $k\ge\overline{k}$, 
 \begin{equation}\label{diffxk}
 \|x^{k+1}-x^k\|\le M_1\|v^k\|\ \ {\rm and}\ \ \|x^{k+1}-x^k-v^k\|\le M_2\|v^k\|^2.
 \end{equation}
 Clearly, $\lim_{k\to\infty}(x^{k+1}-x^k)=0$. From the limit and \eqref{ineq30-ThetaF}, if necessary by increasing $\overline{k}$,  
 \[
 \vartheta(F(x^{k+1}))\le\vartheta(\ell_{F}(x^{k+1};x^k))+\frac{1}{2}\big[{\rm lip}\,\vartheta(F(x^k))\,{\rm lip}\,F'(x^k)\!+\!1\big]\|x^{k+1}\!-x^k\|^2\quad\forall k\ge\overline{k}.
 \]
 The Lipschitz continuity of $\vartheta$ on $D$ with Lipschitz constant $L_{\vartheta}$ implies that for all $k\ge\overline{k}$, 
 \[
  \vartheta(\ell_{F}(x^{k+1};x^k))- \vartheta(\ell_{F}(x^{k}\!+v^k;x^k))
	\le L_{\vartheta}\|F'(x^k)(x^{k+1}\!-x^k-v^k)\|.
 \]
 Adding the last two inequalities together and using equations \eqref{ineq-const} and \eqref{diffxk} leads to  
 \begin{equation*}
 \vartheta(F(x^{k+1}))\le \vartheta(\ell_{F}(x^{k}\!+v^k;x^k))+\frac{1}{2}[(1\!+\!b_{\vartheta}b_{\nabla\!F})M_1^2+2c_{\nabla\!F}L_{\vartheta}M_2]\|v^k\|^2\quad\forall k\ge\overline{k}. 
 \end{equation*} 
 Combining this inequality with the above \eqref{ineq42-Theta}-\eqref{diffxk}, for each $k\in\mathbb{N}$, it holds that
 \begin{align*}
 \Theta(x^{k+1})&\le \Xi(w^k)+c_{\nabla\!f}M_2\|v^k\|^2+\!\frac{1}{2}(L_{\nabla\!f}M_1^2+\mu_{\rm max})\|v^k\|^2\\  &\quad+\frac{1}{2}[(1\!+\!b_{\vartheta}b_{\nabla\!F})M_1^2+2c_{\nabla\!F}L_{\vartheta}M_2]\|v^k\|^2,
 \end{align*}
 which by the definition of $\widetilde{c}$ implies the desired result. The proof is completed. 
\end{proof}

In view of Lemma \ref{lemma-XiTheta}, now we define the potential function $\Xi_{\widetilde{c}}:\mathbb{W}\times\mathbb{X}\to\overline{\mathbb{R}}$ by 
\begin{equation}\label{Xict}
 \Xi_{\widetilde{c}}(u):=\Xi(w)+\frac{\widetilde{c}}{2}\|\eta\|^2\quad\forall u=(w,\eta)\in\mathbb{U}:=\mathbb{W}\times\mathbb{X}, 
\end{equation}
where $\widetilde{c}$ is the constant appearing in Lemma \ref{lemma-XiTheta}. For each $k\in\mathbb{N}$, write $u^k:=(w^k,v^k)$. Denote the sets of cluster points of $\{w^k\}_{k\in\mathbb{N}}$ and $\{u^k\}_{k\in\mathbb{N}}$ as $W^*$ and $U^*$, respectively. 
\begin{proposition}\label{prop-Xi}
 Under Assumption \ref{ass1}, the following assertions hold.  
 \begin{description}
 \item[(i)] $W^*$ and $U^*$ are nonempty and compact, every $w^*\!=(x^*,\overline{v}^*,Q^*,s^*)\in W^*$ satisfies $\overline{v}^*=0$, $x^*\in\mathcal{M}$ and $s^*=\nabla\!f(x^*)$, and $u^*\in U^*$ if and only if $u^*=(w^*,0)$ for some $w^*\in W^*$;

 \item[(ii)] $W^*\subset(\partial\Xi)^{-1}(0)$;

 \item[(iii)] $\lim_{k\to\infty}\Xi(w^k)=\varsigma^*=\lim_{k\to\infty}\Xi_{\widetilde{c}}(u^k)$, where $\varsigma^*$ is the same as in Proposition \ref{prop-xk} (ii);

 \item[(iv)] $\Xi(w)=\varsigma^*$ for all $w\in W^*$ and $\Xi_{\widetilde{c}}(u)=\varsigma^*$ for all $u\in U^*$.
 \end{description}	
\end{proposition}
\begin{proof}
 {\bf(i)-(ii)} Item (i) is direct by Assumptions \ref{ass0} and \ref{ass1} and Proposition \ref{prop-xk} (iii). We prove item (ii). By the expression of $\Xi$, at any $w\!=(x,v,Q,s)\in\mathcal{M}\times T_{x}\mathcal{M}\times\mathbb{R}^{r\times p}\times\mathbb{X}$, 
 \begin{equation}\label{subdiff-Xi}
\!\partial\Xi(w)=\!\left[\begin{pmatrix}
 N_{(x,v)}T\mathcal{M}\\ 0   
\end{pmatrix}\!+\!\begin{pmatrix}
  \nabla\!f(x)\\
   s+\alpha^*v \\
    0
 \end{pmatrix}\!+\!\nabla P(x,v,Q)\partial\vartheta(P(x,v,Q))\right]\times\!\big\{v\big\},
\end{equation}
where $P(x,v,Q):=F(x)\!+\mathcal{F}Q\mathcal{E}^{\dag}v$ for $(x,v,Q)\in\mathbb{X}\times\mathbb{X}\times\mathbb{R}^{r\times p}$. A simple calculation yields
\begin{equation}\label{subdiff-H}
 \nabla P(x,v,Q)\Delta z=\big[\nabla F(x)\Delta z;(\mathcal{E}^{\dag})^*Q^{\top}\mathcal{F}^*\Delta z;\mathcal{F}^*\Delta zv^{\top}(\mathcal{E}^{\dag})^*\big]\quad\forall\Delta z\in\mathbb{Z}.
\end{equation}
Pick any $w^*=(x^*,\overline{v}^*,Q^*,s^*)\in W^*$. Then $x^*\in\mathcal{M},\,\overline{v}^*=0$ and $s^*=\nabla\!f(x^*)$ by item (i). While from the first part of Lemma \ref{lemma-Tb}, it follows $N_{(x^*,\overline{v}^*)}T\mathcal{M}=N_{x^*}\mathcal{M}\times N_{x^*}\mathcal{M}$. Now invoking the equality \eqref{subdiff-Xi} with $w=w^*$ and noting that $P(x^*,\overline{v}^*,Q^*)=F(x^*)$, we have
\[
\!\partial\Xi(w^*)=\!\left[\begin{pmatrix}
 N_{x^*}\mathcal{M}\\ N_{x^*}\mathcal{M}\\0   
\end{pmatrix}\!+\!\begin{pmatrix}
  \nabla\!f(x^*)\\
   \nabla\!f(x^*)\\
    0
 \end{pmatrix}\!+\!\begin{pmatrix}
  \nabla\!F(x^*)\\
   \nabla\!F(x^*)\\
    0
 \end{pmatrix}\partial\vartheta(F(x^*))\right]\times\!\big\{0\big\}.
\]
Recall that $x^*$ is a stationary point of \eqref{prob} by Theorem \ref{theorem1}. From Definition \ref{def-spoint} (i), we have
\begin{equation*}
 0\in\nabla\!f(x^*)+\nabla\!F(x^*)\partial \vartheta(F(x^*))+N_{x^*}\mathcal{M}.
\end{equation*}
The last two equations imply $0\in\partial\Xi(w^*)$, and the desired inclusion $W^*\subset(\partial\Xi)^{-1}(0)$ holds. 

\noindent
{\bf(iii)-(iv)} For each $k\in\mathbb{N}$, from the expression of $\Xi(w^k)$ in  \eqref{Xiwk} and the definition $\Theta_k$, 
\begin{align}\label{ineq41-Xi-Theta}
 \Xi(w^k)&=\Theta_k(\overline{v}^k)+\frac{1}{2}(\alpha^*\|\overline{v}^k\|^2-\|\overline{v}^k\|^2_{\mathcal{Q}_k})\leq \Theta_k(\overline{v}^k)+\frac{\alpha^*}{2}\|\overline{v}^k\|^2\nonumber\\
 &\leq \Theta_k(0)+\frac{\alpha^*}{2}\|\overline{v}^k\|^2=\Theta(x^k)+\frac{\alpha^*}{2}\|\overline{v}^k\|^2.
 \end{align}
 On the other hand, from Lemma \ref{lemma-XiTheta}, $\Xi(w^k)\ge\Theta(x^{k+1})-\frac{\widetilde{c}}{2}\|v^k\|^2$ for all $k\ge\overline{k}$. Passing to the limit $k\to\infty$ in the inequality and \eqref{ineq41-Xi-Theta}, and using Proposition \ref{prop-xk} (ii)-(iii) leads to item (ii). For item (iii), pick any $w^*=(x^*,\overline{v}^*,Q^*,s^*)\in W^*$. Then there exists an index set $\mathcal{K}\subset\mathbb{N}$ such that $w^*=\lim_{\mathcal{K}\ni k\to\infty}w^k$. 
Taking the limit as $\mathcal{K}\ni k\to\infty$ in \eqref{Xiwk} and using $\lim_{k\to\infty}\overline{v}^k=0$ and Assumption \ref{ass0} leads to $\Xi(w^*)=f(x^*)+\vartheta(F(x^*))=\Theta(x^*)=\lim_{\mathcal{K}\ni k\to\infty}\Theta(x^k)=\varsigma^*$. From the definition of $\Xi_{\widetilde{c}}$ and the last part of item (i), $\Xi_{\widetilde{c}}(u)=\varsigma^*$ for all $u\in U^*$. 
\end{proof}
\subsection{Full convergence}\label{sec5.1}

To achieve the full convergence of $\{x^k\}_{k\in\mathbb{N}}$ under the KL framework, we also need to bound the approximate stationarity of $u^k$ to the minimization of $\Xi_{\widetilde{c}}$ by using $\|v^k\|$, i.e., to characterize the relative error condition for the potential function $\Xi_{\widetilde{c}}$. This is a nontrivial task due to the difficulty caused by the manifold constraint. 
\begin{proposition}\label{sdiff-gap}
 Under Assumption \ref{ass1}, there exists a constant $\widehat{\gamma}>0$ such that for all $k\ge\overline{k}$, ${\rm dist}(0,\partial\,\Xi_{\widetilde{c}}(u^k))\le\!(\widehat{\gamma}+\widetilde{c})\|v^k\|$, where $\overline{k}$ is the same as the one in Lemma \ref{lemma-XiTheta}.
\end{proposition}
\begin{proof}
From $\{(x^k,\overline{v}^k)\}_{k\in\mathbb{N}}\subset T\mathcal{M}$ and Corollary \ref{corollary-normalTM}, there exist a compact set $\varLambda\supset\{x^k\}_{k\in\mathbb{N}}$, an index $l\in\mathbb{N}^*$, $\overline{x}^{1},\ldots,\overline{x}^{l}\in\varLambda$, real numbers $\varepsilon_{\overline{x}^{1}}>0,\ldots,\varepsilon_{\overline{x}^{l}}>0$, and $\mathcal{C}^{2}$-smooth mappings $G_i:=G_{\overline{x}^{i}}:\mathbb{X}\to \mathbb{Y}$ for $i\in[l]_{+}$ such that for each $i\in[l]_{+}$ and $x\in\mathbb{B}(\overline{x}^{i},\varepsilon_{\overline{x}^{i}})$, $G_{i}'(x):\mathbb{X}\to \mathbb{Y}$ is surjective, and for each $k\in\mathbb{N}$ there is an index $j_k\in[l]_{+}$ such that 
\begin{subnumcases}{}\label{Normalk}
 x^k\in\mathcal{M}\cap\mathbb{B}(\overline{x}^{j_k},\varepsilon_{\overline{x}^{j_k}}),\,N_{x^k}\mathcal{M}=\big\{\nabla G_{j_k}(x^k) y\ |\ y\in\mathbb{Y}\big\},\qquad\qquad\qquad\\
 \label{Tbundlek}
 N_{(x^k,\overline{v}^k)}T\mathcal{M}=\left\{\begin{pmatrix}
			\nabla G_{j_k}(x^k)\xi+[D^2G_{j_k}(x^k)\overline{v}^k]^*\zeta\\
			\nabla G_{j_k}(x^k)\zeta
 \end{pmatrix}\ |\ \xi\in\mathbb{Y},\zeta\in\mathbb{Y}\right\}.
\end{subnumcases}
For each $k\in\mathbb{N}$, by \eqref{optim-cond} and  \eqref{Normalk}, there exist $\xi^k\in \partial \vartheta(\ell_{F}(x^k\!+\overline{v}^k;x^k))$ and $y^k\in\mathbb{Y}$ such that 
\begin{equation}\label{zeros}
 \nabla\!f(x^k)+\mathcal{Q}_k\overline{v}^k+\nabla F(x^k)\xi^k+\nabla G_{j_k}(x^k)y^k=0.
\end{equation}
Note that $\ell_{F}(x^k\!+\overline{v}^k;x^k)=P(x^k,\overline{v}^k,Q^k)$, where the mapping $P$ is the same as the one in the proof of Proposition \ref{prop-Xi}. Hence,  $\xi^k\in\partial\vartheta(P(x^k,\overline{v}^k,Q^k))$ for each $k\in\mathbb{N}$.
Next we use equations \eqref{subdiff-Xi}-\eqref{subdiff-H} and \eqref{Normalk}-\eqref{zeros} to prove the conclusion by the following two steps.

\noindent
{\bf Step 1: to prove that $\{y^k\}_{k\in\mathbb{N}}$ is bounded}. Since $\vartheta$ is locally Lipschitz on $\mathbb{Z}$, the boundedness of $\{\ell_{F}(x^k\!+\!\overline{v}^k;x^k)\}_{k\in\mathbb{N}}$ and \cite[Theorem 9.13 $\&$ Proposition 5.15]{RW98} imply the boundedness of $\{\xi^k\}_{k\in \mathbb{N}}$, so does that of the sequence $\{\nabla\!f(x^k)\!+\!\mathcal{Q}_k\overline{v}^k+\!\nabla F(x^k)\xi^k\}_{k\in\mathbb{N}}$. Suppose on the contrary that $\{y^k\}_{k\in\mathbb{N}}$ is unbounded, i.e., there exists an infinite index set $\mathcal{K}\subset\mathbb{N}$ such that $\lim_{\mathcal{K}\ni k\to\infty}\|y^k\|=\infty$. If necessary by taking a subset of $\mathcal{K}$, we can assume that $\lim_{\mathcal{K}\ni k\to\infty}\frac{y^k}{\|y^k\|}=\overline{y}$ for some $\overline{y}\in\mathbb{Y}$ with $\|\overline{y}\|=1$. Since the index set $[l]_{+}$ is finite, from $j_k\in[l]_{+}$ for each $k$, there necessarily exist an index set $\mathcal{K}_1\subset\mathcal{K}$ and an index $i_0\in[l]_{+}$ such that for each $k\in\mathcal{K}_1$, $x^k\in\mathcal{M}\cap\mathbb{B}(\overline{x}^{i_0},\varepsilon_{\overline{x}^{i_0}})$ and $G_{j_k}=G_{i_0}$. From the boundedness of $\{x^k\}_{k\in\mathbb{N}}$, if necessary by shrinking the index set $\mathcal{K}_1$, we can assume that $\lim_{\mathcal{K}_1\ni k\to\infty}x^k=x^*\in\mathcal{M}\cap\overline{\mathbb{B}}(\overline{x}^{i_0},\varepsilon_{\overline{x}^{i_0}})$, where the inclusion is due to the closedness of $\mathcal{M}$. Together with the above \eqref{zeros}, for sufficiently large $k\in\mathcal{K}_1$, 
\[
 \frac{\nabla\!f(x^k)+\nabla F(x^k)\xi^k+\mathcal{Q}_k\overline{v}^k}{\|y^k\|}+\nabla G_{i_0}(x^k)\frac{y^k}{\|y^k\|}=0.
\]
 Passing to the limit $\mathcal{K}_1\ni k\to\infty$ in this equality and using the smoothness of $G_{i_0}$ leads to $\nabla G_{i_0}(x^*)\overline{y}=0$, which along with the surjectivity of $G_{i_0}'(x^*)\!:\mathbb{X}\to\mathbb{Y}$ yields $\overline{y}=0$, a contradiction to $\|\overline{y}\|=1$. The boundedness of $\{y^k\}_{k\in\mathbb{N}}$ follows. 

 \noindent
{\bf Step 2: to establish the desired bound.}
 For each $k\in\mathbb{N}$, we introduce the notation
 \begin{align*}
 \omega_{x}^k:=\nabla\!f(x^k)+\nabla F(x^k)\xi^k\!+\!\big[\nabla G_{j_k}(x^k)\!+\!(D^2G_{j_k}(x^k)\overline{v}^k)^*\big]y^k,\qquad\qquad\\
 \omega_{v}^k:=\nabla f(x^k)+(\mathcal{E}^{\dag})^*(Q^k)^{\top}\mathcal{F}^*\xi^k+\alpha^*\overline{v}^k+\nabla G_{j_k}(x^k)y^k\ \ {\rm and}\ \ 
 \omega_Q^k:=\mathcal{F}^*\xi^k(\overline{v}^k)^{\top}(\mathcal{E}^{\dag})^*. 
 \end{align*} 
 Recall that $\xi^k\in\partial\vartheta(P(x^k,\overline{v}^k,Q^k))$ and $w^k=(x^k,\overline{v}^k,Q^k,\nabla\!f(x^k))$ for each $k\in\mathbb{N}$. Comparing with \eqref{subdiff-Xi}-\eqref{subdiff-H} and \eqref{Normalk}-\eqref{Tbundlek}, we have $(\omega_{x}^k,\omega_{v}^k,\omega_{Q}^k,\overline{v}^k)\in\partial\Xi(w^k)$. Further, by \eqref{zeros}, 
 \begin{equation}\label{temp-ghk}
 \omega_{x}^k=(D^2G_{j_k}(x^k)\overline{v}^k)^*y^k-\mathcal{Q}_k\overline{v}^k,\,  \omega_{v}^k=\alpha^*\overline{v}^k-\mathcal{Q}_k\overline{v}^k\ \ {\rm and}\ \ \omega_Q^k:=\mathcal{F}^*\xi^k(\overline{v}^k)^{\top}(\mathcal{E}^{\dag})^*.
 \end{equation}
 where the second equality is obtained by using the second equality of \eqref{gradFk}. Since $\{x^k\}_{k\in\mathbb{N}}$ and $\{y^k\}_{k\in \mathbb{N}}$ are bounded, from the $\mathcal{C}^2$-smoothness of $G_i$ for each $i\in[l]_{+}$, there exists $\widehat{c}_1>0$ such that  $\|(D^2G_{j_k}(x^k)\overline{v}^k)^*y^k\|\le \widehat{c}_1\|\overline{v}^k\|$ for all $k\in\mathbb{N}$. Together with the first equality of \eqref{temp-ghk}, Proposition \ref{bound-alpk} (ii) and Proposition \ref{prop-xk} (i), it follows for each $k\in\mathbb{N}$,
 \[
 \|\omega_{x}^k\|\leq (\widehat{c}_1+\alpha^*)\|\overline{v}^k\|\le\!(\widehat{c}_1+\alpha^*)\left[\sqrt{\alpha_{\rm min}^{-1}\mu_{\rm max}}+1\right]\|v^k\|. 
 \]
 Combining the second equality of \eqref{temp-ghk} with Proposition \ref{bound-alpk} (ii) and Proposition \ref{prop-xk} (i) yields $\|\omega_{v}^k\|\le 2\alpha^*\|\overline{v}^k\|\le 2\alpha^*\big[\sqrt{\alpha_{\rm min}^{-1}\mu_{\rm max}}+1\big]\|v^k\|$ for each $k$. By the boundedness of $\{\xi^k\}_{k\in\mathbb{N}}$ and Proposition \ref{prop-xk} (i), there exists a constant $\widehat{c}_2>0$ such that $\|\omega_{Q}^k\|\le\|\mathcal{F}^*\xi^k(\overline{v}^k)^{\top}(\mathcal{E}^{\dag})^*\|\le\widehat{c}_2\|\overline{v}^k\|\le\widehat{c}_2\big[\sqrt{\alpha_{\rm min}^{-1}\mu_{\rm max}}+1\big]\|v^k\|$ for each $k$. Thus, for each $k\ge\overline{k}$, it holds
\[
 {\rm dist}(0,\partial\,\Xi_{\widetilde{c}}(u^k))\le{\rm dist}(0,\partial\,\Xi(w^k))+\widetilde{c}\|v^k\|\leq \|(\omega_{x}^k,\omega_{v}^k,\omega_{Q}^k,\overline{v}^k)\|+\widetilde{c}\|v^k\|\leq(\widehat{\gamma}+\widetilde{c})\|v^k\|
\]
with $\widehat{\gamma}=\big[\sqrt{\alpha_{\rm min}^{-1}\mu_{\rm max}}\!+1\big]\sqrt{(\widehat{c}_1+\alpha^*)^2+4(\alpha^*)^2+\widehat{c}_2^2+1}$. The proof is finished. 
\end{proof}

Although $\{\Xi_{\widetilde{c}}(u^k)\}_{k\in\mathbb{N}}$ lacks decrease, we can prove the full convergence of $\{x^k\}_{k\in\mathbb{N}}$ by combining the KL
inequality on $\Xi_{\widetilde{c}}$ with the decrease of $\{\Theta(x^k)\}_{k\in\mathbb{N}}$ skillfully. 
\begin{theorem}\label{theorem-converge}
 Under Assumption \ref{ass1}, if the function $\Xi_{\widetilde{c}}$ has the KL property on the set $U^*$, then $\sum_{k=1}^{\infty}\|x^{k+1}\!-x^k\|<\infty$, so the sequence $\{x^k\}_{k\in\mathbb{N}}$ converges to some $x^*\in\Omega(x^0)$, where $\Omega(x^0)$ is defined in Theorem \ref{theorem1}.
\end{theorem}
\begin{proof}
 If there exists some $k_0\in\mathbb{N}$ such that $\Theta(x^{k_0})=\Theta(x^{k_0+1})$, from Proposition \ref{prop-xk} (ii), we have $v^{k_0}=0$, which by Remark \ref{remark-alg} (c) shows that Algorithm \ref{iRVM} finds a stationary point $x^{k_0}$ of problem \eqref{prob} within a finite number of steps. Otherwise, by Proposition \ref{prop-xk} (ii), it suffices to consider that $\Theta(x^k)>\Theta(x^{k+1})>\varsigma^*$ for all $k\in\mathbb{N}$, where $\varsigma^*$ is the same as in Proposition \ref{prop-xk} (ii). In this case, from Lemma \ref{lemma-XiTheta}, $\Xi_{\widetilde{c}}(u^k)>\varsigma^*$ for all $k\ge\overline{k}$. By Proposition \ref{prop-Xi} (i), the set $U^*$ is nonempty and compact, and $\Xi_{\widetilde{c}}(u)=\varsigma^*$ for all $u\in U^*$. Thus, from the KL property of $\Xi_{\widetilde{c}}$ and \cite[Lemma 6]{Bolte2014}, there exist some $\delta>0$, $\varpi>0$ and $\varphi\in\Upsilon_{\varpi}$ such that for all $u\in[\varsigma^*<\Xi_{\widetilde{c}}<\varsigma^*+\varpi]\cap\{u\in\mathbb{U}\mid {\rm dist}(u,U^*)\leq \delta\}$, $\varphi'(\Xi_{\widetilde{c}}(u)-\varsigma^*){\rm dist}(0,\partial\, \Xi_{\widetilde{c}}(u))\geq 1$. Recall that $\lim_{k\to\infty}\Xi_{\widetilde{c}}(u^k)=\varsigma^*$	and $\lim_{k\to\infty}{\rm dist}(u^k,U^*)=0$. There exists $\widehat{k}\ge\overline{k}\!+\!1$ such that $u^{k-1}\in[\varsigma^*<\Xi_{\widetilde{c}}<\varsigma^*+\varpi]\cap\{u\in\mathbb{U}\mid {\rm dist}(u,U^*)\leq \delta\}$ for all $k\ge\widehat{k}$. Then, for all $k\ge\widehat{k}$,  
 \begin{equation*}
 \varphi'(\Xi_{\widetilde{c}}(u^{k-1})-\varsigma^*){\rm dist}(0,\partial\, \Xi_{\widetilde{c}}(u^{k-1}))\geq 1.
\end{equation*}
 Together with Lemma \ref{lemma-XiTheta}, the nonincreasing of $\varphi'$ on $(0,\infty)$, and Proposition \ref{sdiff-gap}, we obtain 
 \begin{equation}\label{temp43-converge}
 \varphi'(\Theta(x^{k})-\varsigma^*)\ge\varphi'(\Xi_{\widetilde{c}}(u^{k-1})-\varsigma^*)\ge\frac{1}{(\widehat{\gamma}+\widetilde{c})\|v^{k-1}\|}\quad\forall k\ge\widehat{k}. 
 \end{equation}
 From the concavity of $\varphi$ and Proposition \ref{prop-xk} (ii), it then follows that for each $k\ge\widehat{k}$, 
 \begin{equation*}
 \Delta_{k,k+1}:=\varphi(\Theta(x^{k})-\varsigma^*)-\varphi(\Theta(x^{k+1})-\varsigma^*)\ge \frac{\Theta(x^{k})-\Theta(x^{k+1})}{(\widehat{\gamma}+\widetilde{c})\|v^{k-1}\|}\ge\frac{\gamma\|v^k\|^2}{2(\widehat{\gamma}+\widetilde{c})\|v^{k-1}\|},
 \end{equation*}
 which implies that $\|v^k\|\le\sqrt{2(\widehat{\gamma}+\widetilde{c})\gamma^{-1}\|v^{k-1}\|\Delta_{k,k+1}}$. Along with $2\sqrt{ab}\le a+b$ for $a,b\ge 0$, we obtain $2\|v^k\|\le 2(\widehat{\gamma}+\widetilde{c})\gamma^{-1}\Delta_{k,k+1}+\|v^{k-1}\|$ for all $k\ge\widehat{k}$. Summing this inequality from any $k\ge\widehat{k}$ to any $l\ge k$ yields that  
 \(
	2\textstyle{\sum_{i=k}^{l}}\|v^i\|\le \sum_{i=k}^{l}\|v^{i-1}\|+2(\widehat{\gamma}+\widetilde{c})\gamma^{-1}\sum_{i=k}^{l}\Delta_{i,i+1}.
 \)
 From the definition of $\Delta_{i,i+1}$ and the nonnegativity of $\varphi$, it follows 
 \begin{equation}\label{vk-ineq51}
 \textstyle{\sum_{i=k}^{l}}\|v^i\|\le \|v^{k-1}\|+2(\widehat{\gamma}+\widetilde{c})\gamma^{-1}\varphi(\Theta(x^k)-\varsigma^*).
 \end{equation}
 Recall that $x^{k+1}=R_{x^k}(v^k)$ for each $k$ and $\lim_{k\to\infty}v^k=0$. Using Lemma \ref{lemma-retract} with $\varLambda=\varLambda_0$ in Assumption \ref{ass1}, there exist $\widetilde{k}\ge\widehat{k}$ and $M_1',M_2'>0$ such that for all $k\ge\widetilde{k}$,
 \begin{equation}\label{vk-ineq52}
 \|x^{k+1}-x^k\|\le M_1'\|v^k\|,\,\|x^{k}-x^{k-1}-v^{k-1}\|\le M_2'\|v^{k-1}\|^2\ {\rm and}\ M_2'\|v^{k-1}\|\le \frac{1}{2}.
 \end{equation}
  By the latter two inequalities in \eqref{vk-ineq52}, $\|v^{k-1}\|\le \|x^{k}-x^{k-1}\|+M_2'\|v^{k-1}\|^2\le \|x^{k}-x^{k-1}\|+\frac{1}{2}\|v^{k-1}\|$, which implies that $\|v^{k-1}\|\le 2\|x^k-x^{k-1}\|$ for all $k\ge\widetilde{k}$.
 Together with the first inequality in \eqref{vk-ineq52} and the above \eqref{vk-ineq51}, for any $k\ge\widetilde{k}$ and $l\ge k$, 
 \begin{align}\label{temp45-rate}
 \sum_{i=k}^{l}\|x^{i+1}-x^i\|\le \sum_{i=k}^{l}M_1'\|v^i\|\le 2M_1'\|x^k-x^{k-1}\|+\frac{2M_1'(\widehat{\gamma}+\widetilde{c})}{\gamma}\varphi(\Theta(x^{k})-\varsigma^*).
\end{align}
Passing to the limit $l\to\infty$ in this inequality leads to $\sum_{k=1}^{\infty}\|x^{k+1}\!-x^k\|<\infty$. 
\end{proof}

From the expression of $\Xi_{\widetilde{c}}$ and Section \ref{sec2.4}, $\Xi_{\widetilde{c}}$ is a KL function if $\Xi$ is definable in an o-minimal structure over the real field, which is easily identified when the expression of $\mathcal{M}$ is known. For example, for the $\mathcal{M}$ in Examples \ref{Example1}-\ref{Example4}, $T\mathcal{M}$ is a semialgebraic set, so $\delta_{T\mathcal{M}}$ is a semialgebraic function, which means that $\Xi$ is definable in an o-minimal structure over the real field if $f,\vartheta$ and $F$ are all definable in this o-minimal structure.  
\subsection{Local convergence rate}\label{sec5.2}

When $\Xi$ has the KL property of exponent $q\in[\frac{1}{2},1)$ on $W^*$, we can prove that $\{x^k\}_{k\in\mathbb{N}}$ converges to $x^*$ in a linear rate for $q=\frac{1}{2}$ and a sublinear rate for $q\in(\frac{1}{2},1)$.
\begin{theorem}\label{coverge-rate}
 Under Assumption \ref{ass1}, if $\Xi$ has the KL property of exponent $q\in[\frac{1}{2},1)$ on the set $W^*$, then $\{x^k\}_{k\in\mathbb{N}}$ converges to some $x^*\in\Omega(x^0)$, where $\Omega(x^0)$ is defined in Theorem \ref{theorem1}, and furthermore,
 \begin{description}
 \item[(i)] when $q=\frac{1}{2}$, there exist $\varrho,\widehat{\varrho}\in(0,1)$ and $\gamma_1>0$ such that for sufficiently large $k$, 
 \[
  \Theta(x^k)-\varsigma^*\le\varrho(\Theta(x^{k-1})-\varsigma^*)\ \ {\rm and}\ \ \|x^k-x^*\|\le\gamma_1\widehat{\varrho}^k;
\]
		
\item[(ii)] when $q\in(\frac{1}{2},1)$, there exist $\gamma_1>0$ and $\gamma_2>0$ such that for sufficiently large $k$,
\[
 \Theta(x^k)-\varsigma^*\le\gamma_1 k^{\frac{1}{1-2q}}\ \ {\rm and}\ \ \|x^k-x^*\|\le\gamma_2k^{\frac{1-q}{1-2q}}.
\]
\end{description}	
\end{theorem}
\begin{proof}
 Note that $\Xi$ is continuous relative to its domain and the function $\mathbb{X}\ni v\mapsto\frac{\widetilde{c}}{2}\|v\|^2$ has the KL property of exponent $\frac{1}{2}$. From \cite[Theorem 3.3]{LiPong2018}, $\Xi_{\widetilde{c}}$ has the KL property of exponent $q\in[\frac{1}{2},1)$ on $U^*$. The convergence of $\{x^k\}_{k\in\mathbb{N}}$ follows Theorem \ref{theorem-converge}. To achieve items (i)-(ii), by the proof of Theorem \ref{theorem-converge}, it suffices to consider the case that $\Theta(x^k)>\Theta(x^{k+1})>\varsigma^*$ for each $k\in\mathbb{N}$, and now inequality \eqref{temp43-converge} holds with $\mathbb{R}_{+}\ni t\mapsto\varphi(t)=ct^{1-q}$ for some $c>0$, i.e.,
 \begin{equation}\label{temp44-rate}
 (\Theta(x^{k})-\varsigma^*)^{q}\le(\Xi_{\widetilde{c}}(u^{k-1})-\varsigma^*)^{q}\le c(1-q)(\widehat{\gamma}+\widetilde{c})\|v^{k-1}\|\quad\forall k\ge\widehat{k}. 
 \end{equation} 
 Together with Proposition \ref{prop-xk} (ii), for each $k\ge\widehat{k}$, by letting  $\omega_k:=\Theta(x^k)-\varsigma^*$, it holds  
 \begin{equation}\label{ineq46-rate}
 \omega_k^{2q}\le c_1\big(\omega_{k-1}-\omega_k\big)\ \ {\rm with}\ \ c_1:=2\overline{\gamma}^{-1}[c(1-q)(\widehat{\gamma}+\widetilde{c})]^2.
 \end{equation} 
 In addition, passing to the limit $l\to\infty$ in \eqref{temp45-rate} and using $\varphi(t)=ct^{1-q}$ for $t\ge 0$ leads to 
 \[
	\Delta_{k}:=\sum_{j=k}^{\infty}\|x^{j+1}-x^j\|\le 2M_1'\|x^{k}-x^{k-1}\|+2cM_1'(\widehat{\gamma}+\widetilde{c})\overline{\gamma}^{-1}(\Theta(x^{k})-\varsigma^*)^{1-q}\quad\forall k\ge\widetilde{k}.
 \]
 Note that $\lim_{k\to\infty}\|x^{k+1}-x^{k}\|=0$ and $0<\frac{1-q}{q}\le 1$. If necessary by increasing $\widetilde{k}$, we have $\|x^{k}-x^{k-1}\|\le\|x^{k}-x^{k-1}\|^{\frac{1-q}{q}}$ for each $k\ge\widetilde{k}$. Then, for all $k\ge\widetilde{k}$, 
 \begin{align}\label{ineq47-rate}  
 \Delta_{k}&\le 2M_1'\|x^{k}-x^{k-1}\|^{\frac{1-q}{q}}+2cM_1'(\widehat{\gamma}+\widetilde{c})\overline{\gamma}^{-1}(\Theta(x^{k})-\varsigma^*)^{1-q}\nonumber\\
 &\le2M_1'\|x^{k}-x^{k-1}\|^{\frac{1-q}{q}}+2M_1'\overline{\gamma}^{-1}[c(\widehat{\gamma}+\widetilde{c})]^{1+\frac{1}{q}}(1-q)^{\frac{1}{q}}\|v^{k-1}\|^{\frac{1-q}{q}}\nonumber\\
 &\le 2M_1'\|x^{k}-x^{k-1}\|^{\frac{1-q}{q}}+4M_1'\overline{\gamma}^{-1}[c(\widehat{\gamma}+\widetilde{c})]^{1+\frac{1}{q}}(1-q)^{\frac{1}{q}}\|x^k-x^{k-1}\|^{\frac{1-q}{q}}\nonumber\\
 &=c_2(\Delta_{k-1}-\Delta_k)^{\frac{1-q}{q}}\ \ {\rm with}\ 
 c_2:=2M_1'+4M_1\overline{\gamma}^{-1}[c(\widehat{\gamma}+\widetilde{c})]^{1+\frac{1}{q}}(1-q)^{\frac{1}{q}}.
 \end{align}
 where the second inequality follows from \eqref{temp44-rate} and the third inequality is due to $\|v^{k-1}\|\le 2\|x^k-x^{k-1}\|$ for all $k\ge\widetilde{k}$ by the proof of Theorem \ref{theorem-converge} and $\frac{1-q}{q}\le 1$. With \eqref{ineq46-rate} and \eqref{ineq47-rate}, it is easy to obtain the conclusion of (i)-(ii) by following the same arguments as those for \cite[Theorem 2]{Attouch2009}, so we here omit the details. 
\end{proof}

The KL property of $\Xi$ with an exponent $q\in[\frac{1}{2},1)$ on $W^*$ is the key to achieve the local convergence rate of $\{x^k\}_{k\in\mathbb{N}}$. Unlike the KL property, the KL property with an explicit exponent $q$ is rare for manifold optimization except for special objective functions or manifolds (see Liu et al. \cite{LiuSo2019}). Recall that $T\mathcal{M}$ is an embedded closed submanifold, and $\Xi$ is a composite function over this manifold. It is important to provide a reasonable condition for $\Xi$ to satisfy the KL property of exponent $q$ on $W^*$. From the remark after Definition \ref{def-KL}, it suffices to provide a condition for $\Xi$ to satisfy the KL property of exponent $q$ on the set $W^*\cap(\partial\Xi)^{-1}(0)=W^*$, where the equality is due to Proposition \ref{prop-Xi} (ii). Observe that $\Xi=\phi\circ H$ with the function $\phi$ defined by 
\begin{equation}\label{phi-fun}
 \phi(w,z)\!:=f(x)+\langle s,v\rangle+\vartheta(z)+\delta_{T\mathcal{M}}(x,v)+\frac{\alpha^*}{2}\|v\|^2\quad{\rm for}\ (w,z)\in\mathbb{W}\times\mathbb{Z},
\end{equation}
and $H(w):=(w;P(x,v,Q))$ for $w=(x,v,Q,s)\in\mathbb{W}$, where $P$ is the same as the one in the proof of Proposition \ref{prop-Xi} (ii). Since $\phi$ is almost separable, checking its KL property of exponent $q$ is much easier than doing that of $\Xi$. Inspired by this, we  use the KL property of $\phi$ with an exponent $q$ to derive a condition for checking that of $\Xi$. 
\begin{proposition}\label{KL-prop}
 Consider any $w^*=(x^*,v^*,Q^*,s^*)\in W^*$ and $q\in[0,1)$. Suppose that $\phi$ in \eqref{phi-fun} satisfies the KL property of exponent $q$ at $(w^*,F(x^*))$, and that the  implication 
 \begin{equation}\label{KL-imply}
  \Delta_{x}+\nabla F(x^*)\Delta_{z}=0\
 \Longrightarrow\ \|(\Delta_{x};\Delta_{z})\|=0
 \end{equation}
 holds for any $(\Delta_{x},\Delta_{z})\in\big\{(\zeta,\xi)\in(N_{x^*}\mathcal{M}+\tau\nabla\!f(x^*))\times\tau\partial \vartheta(F(x^*))\ |\ \tau>0\big\}$. Then, the function $\Xi$ satisfies the KL property of exponent $q$ at $w^*$.
 \end{proposition}
\begin{proof} 
 By Proposition \ref{prop-Xi} (i), $x^*\in\mathcal{M}, v^*=0$ and $s^*=\nabla\!f(x^*)$. Suppose that $\Xi$ does not have the KL property of exponent $q$ at $w^*$. By Definition \ref{def-KL}, there exists a sequence $\{\widetilde{w}^k\}_{k\in\mathbb{N}}\subset\mathbb{W}$ converging to $w^*$ with $\widetilde{w}^k=(\widetilde{x}^k,\widetilde{v}^k,\widetilde{Q}^k,\widetilde{s}^k)$ and $\Xi(w^*)<\Xi(\widetilde{w}^k)<\Xi(w^*)+\frac{1}{k}$ such that 
 \begin{equation}\label{partialXi-ineq}
 {\rm dist}(0,\partial \Xi(\widetilde{w}^k))< \dfrac{1}{k}(\Xi(\widetilde{w}^k)-\Xi(w^*))^{q}=\frac{1}{k}\big[\phi(\widetilde{w}^k,\widetilde{z}^k)-\phi(w^*,z^*)\big]^{q}\quad\forall k\in\mathbb{N},
 \end{equation}
 where $\widetilde{z}^k\!:=P(\widetilde{x}^k,\widetilde{v}^k,\widetilde{Q}^k)$ and $z^*\!:=P(x^*,v^*,Q^*)=F(x^*)+\mathcal{F}Q^*\mathcal{E}^{\dag}v^*\!=F(x^*)$. 
 Obviously, for each $k\in\mathbb{N}$, $(\widetilde{x}^k,\widetilde{v}^k,\widetilde{Q}^k,\widetilde{s}^k)\in\mathcal{M}\times T_{\widetilde{x}^k}\mathcal{M}\times\mathbb{R}^{r\times p}\times\mathbb{X}$. Invoking the limit $\lim_{k\to\infty}\widetilde{x}^k=x^*$ and  \eqref{tb-eq1}, there exist $\overline{k}_1\in\mathbb{N}$ and $\zeta_1^k,\zeta_2^k\in\mathbb{Y}$ such that for each $k\ge\overline{k}_1$, 
\begin{equation}\label{tb-normalk}
\begin{pmatrix}
   \nabla G_{x^*}(\widetilde{x}^k)\zeta_1^k+\!(D^2G_{x^*}(\widetilde{x}^k)\widetilde{v}^k)^*\zeta^k_2\\
   \nabla G_{x^*}(\widetilde{x}^k)\zeta^k_2
   \end{pmatrix}\in N_{(\widetilde{x}^k,\widetilde{v}^k)}T\mathcal{M}.
\end{equation}
Together with \eqref{subdiff-Xi}-\eqref{subdiff-H} and \eqref{partialXi-ineq}, for each $k\ge\overline{k}_1$ 
there exists $\xi^k\in\partial\vartheta(\widetilde{z}^k)$ such that
 \begin{align}\label{temp42-KL}
  &\left\|\begin{pmatrix}
 \nabla\!f(\widetilde{x}^k)+\nabla F(\widetilde{x}^k)\xi^k+\nabla G_{x^*}(\widetilde{x}^k)\zeta_1^k+(D^2G_{x^*}(\widetilde{x}^k)\widetilde{v}^k)^*\zeta^k_2\\
	\widetilde{s}^k+\alpha^*\widetilde{v}^k+(\mathcal{E}^{\dag})^*(\widetilde{Q}^k)^{\top}\mathcal{F}^*\xi^k+\nabla G_{x^*}(\widetilde{x}^k)\zeta^k_2\\
   \mathcal{F}^*\xi^k(\widetilde{v}^k)^{\top}(\mathcal{E}^{\dag})^*\\ 
	\widetilde{v}^k
  \end{pmatrix}\right\|\nonumber\\
  &<\frac{1}{k}\big[\phi(\widetilde{w}^k,\widetilde{z}^k)-\phi(w^*,z^*)\big]^{q}.
 \end{align}
 Notice that $\nabla F(x^k)\xi^k=(\mathcal{E}^{\dag})^*(\widetilde{Q}^k)^{\top}\mathcal{F}^*\xi^k$ for each $k\ge\overline{k}_1$. The inequality \eqref{temp42-KL} implies the boundedness of $\{(\widetilde{s}^k+\alpha^*\widetilde{v}^k+\nabla F(\widetilde{x}^k)\xi^k+\nabla G_{x^*}(\widetilde{x}^k)\zeta^k_2)\}_{k\ge\overline{k}_1}$, while the local boundedness of $\partial\vartheta$ and \cite[Proposition 5.15]{RW98} imply that the sequence $\{\xi^k\}_{k\ge\overline{k}_1}$ is bounded. Then, the surjectivity of the mapping $G'_{x^*}(x^*)\!:\mathbb{X}\to\mathbb{Y}$ implies the boundedness of $\{\zeta_2^k\}_{k\ge\overline{k}_1}$, which along with the above \eqref{temp42-KL} implies that of the sequence $\{\zeta_1^k\}_{k\ge\overline{k}_1}$. For each $k\ge\overline{k}_1$, write 
 \begin{equation*}
 \Delta_1^k:=\nabla G_{x^*}(\widetilde{x}^k)\zeta_1^k+\!(D^2G_{x^*}(\widetilde{x}^k)\widetilde{v}^k)^*\zeta^k_2+\!\nabla\!f(\widetilde{x}^k)\ {\rm and}\ 
 \Delta_2^k:= \nabla G_{x^*}(\widetilde{x}^k)\zeta^k_2+\widetilde{s}^k\!+\!\alpha^*\widetilde{v}^k.
 \end{equation*}
 By the definition of $\phi$, at any $(w,z)$ with $w= (x,v,Q,s)\in\mathcal{M}\times T_{x}\mathcal{M}\times\mathbb{R}^{r\times p}\times\mathbb{X}$ and $z\in\mathbb{Z}$,
 \begin{equation*}
 \partial\phi(w,z)=\left[N_{(x,v)}T\mathcal{M}\!+\!\begin{pmatrix}
 			\nabla\!f(x)\\
			s+\alpha^*v 
 \end{pmatrix}\right]\times\partial\vartheta(z)\times\big\{v\big\}.
 \end{equation*}
 By virtue of \eqref{tb-normalk}, we have 
 $(\Delta_1^k;\Delta _2^k;\xi^k;\widetilde{v}^k)\in\partial\phi(\widetilde{w}^k,\widetilde{z}^k)$ for all $k\ge\overline{k}_1$. Together with the KL property of $\phi$ with exponent $q$ at $(w^*,z^*)$, there exists $\overline{c}_1>0$ such that for all $k\ge\overline{k}_1$,
 \begin{equation*}
 \overline{c}_1(\phi(\widetilde{w}^k,\widetilde{z}^k)-\phi(w^*,z^*))^{q}\leq {\rm dist}(0,\partial\phi(\widetilde{w}^k,\widetilde{z}^k))\le\|(\Delta_1^k;\Delta_2^k;\xi^k;\widetilde{v}^k)\|.
 \end{equation*}
 For each $k\ge\overline{k}_1$, let $t_k:=\|(\Delta_1^k;\Delta_2^k;\xi^k;\widetilde{v}^k)\|$ and $\widehat{\Delta}_1^k:=\frac{\Delta_1^k}{t_k},\widehat{\Delta}_2^k:=\frac{\Delta_2^k}{t_k}, \widehat{\xi}^k:=\frac{\xi^k}{t_k},\widehat{v}^k:=\frac{\widetilde{v}^k}{t_k}$.
 Using these notation and combining the last inequality with the above \eqref{temp42-KL}, for each $k\ge\overline{k}_1$, 
 \begin{equation}\label{temp-ineq46}
  \left\|\big(
 \widehat{\Delta}_1^k+\nabla F(\widetilde{x}^k)\widehat{\xi}^k;
 \widehat{\Delta}_2^k+\!\nabla F(\widetilde{x}^k)\widehat{\xi}^k;\mathcal{F}^*\xi^k(\widehat{v}^k)^{\top}(\mathcal{E}^{\dag})^*;
 \widehat{v}^k\big)\right\|<\frac{1}{\overline{c}_1k}.  \end{equation}
If necessary by taking a subsequence, we assume  $\lim_{k\to\infty}(\widehat{\Delta}_1^k;\widehat{\Delta}_2^k;\widehat{\xi}^k;\widehat{v}^k)=(\widehat{\Delta}_1;\widehat{\Delta}_2;\widehat{\xi}^*;\widehat{v}^*)$ with $\|(\widehat{\Delta}_1;\widehat{\Delta}_2;\widehat{\xi}^*;\widehat{v}^*)\|=1$. Along with the last inequality and $\lim_{k\to\infty}\widetilde{x}^k=x^*$,  
\begin{equation}\label{Delta-ineq}
\widehat{\Delta}_1+\nabla F(x^*)\widehat{\xi}^*=0,\, \widehat{\Delta}_2+\nabla F(x^*)\widehat{\xi}^*=0,\,  \widehat{v}^*=0\ \ {\rm and}\ \ \|(\widehat{\Delta}_1;\widehat{\Delta}_2;\widehat{\xi}^*)\|=1.
\end{equation}
In addition, from the definitions of $\widehat{\Delta}_1^k$ and $\widehat{\Delta}_2^k$, for each $k\ge\overline{k}_1$, by \eqref{tb-normalk}, it holds
\begin{align*}
 \begin{pmatrix}
  \widehat{\Delta}_1^k\\
  \widehat{\Delta}_2^k-\alpha^*\widehat{v}^k
 \end{pmatrix}&=t_k^{-1}\left[\begin{pmatrix}
 \nabla G_{x^*}(\widetilde{x}^k)\zeta_1^k+\!(D^2G_{x^*}(\widetilde{x}^k)\widetilde{v}^k)^*\zeta^k_2\\
  \nabla G_{x^*}(\widetilde{x}^k)\zeta^k_2
 \end{pmatrix}+\begin{pmatrix}
 \nabla\!f(\widetilde{x}^k)\\
 \widetilde{s}^k
 \end{pmatrix}\right]\nonumber\\
 &\in N_{(\widetilde{x}^k,\widetilde{v}^k)}T\mathcal{M}+t_k^{-1}\big(\nabla\!f(\widetilde{x}^k);\widetilde{s}^k\big).
\end{align*}
Note that the sequence $\{t_k\}_{k\ge\overline{k}}$ is bounded. We proceed the proof by two cases as below.
 
\noindent
{\bf Case 1: $\liminf_{k\to\infty}t_k>0$.} Note that $\{t_k\}_{k\ge\overline{k}}$ is bounded. If necessary by taking a subsequence, we can assume that $\lim_{k\to\infty}t_k=t_*>0$. Note that $\widehat{\xi}^k\in t_k^{-1}(\partial\vartheta(\widetilde{z}^k))$ for all $k\ge\overline{k}_1$ and $\lim_{k\to\infty}\widehat{\xi}^k=\widehat{\xi}^*$. We have $\widehat{\xi}^*\in t_*^{-1}\partial \vartheta(z^*)=t_*^{-1}\partial\vartheta(F(x^*))$. Passing to the limit $k\to\infty$ in the last inclusion and using the outer semicontinuity of the mapping $N_{(\cdot,\cdot)}T\mathcal{M}$ leads to 
\begin{align*}
 (\widehat{\Delta}_1,\widehat{\Delta}_2)&\in N_{(x^*,0)}T\mathcal{M}+t_*^{-1}(\nabla\!f(x^*),\nabla\!f(x^*))\\
 &=(N_{x^*}\mathcal{M}+t_*^{-1}\nabla\!f(x^*))\times (N_{x^*}\mathcal{M}+t_*^{-1}\nabla\!f(x^*)),
\end{align*}
where the equality is due to Lemma \ref{Normal-TM}. Along with the first two equalities in \eqref{Delta-ineq} and the implication \eqref{KL-imply}, we get $\|(\widehat{\Delta}_1;\widehat{\Delta}_2;\widehat{\xi}^*)\|=0$, a contradiction to $\|(\widehat{\Delta}_1;\widehat{\Delta}_2;\widehat{\xi}^*)\|=1$. 

\noindent
{\bf Case 2: $\liminf_{k\to\infty}t_k=0$.} In this case, since $\widehat{\xi}^k=t_k^{-1}\partial\vartheta(\widetilde{z}^k)$ and $\lim_{k\to\infty}\widehat{\xi}^k=\widehat{\xi}^*$, we have $\widehat{\xi}^*\in\limsup_{z\to F(x^*),\tau\downarrow 0}\tau^{-1}\partial \vartheta(z)$. Recall that the mapping $\partial\vartheta:\mathbb{Z}\rightrightarrows\mathbb{Z}$ is locally bounded. We infer that the outer limit set $\limsup_{z\to F(x^*),\tau\downarrow 0}\tau^{-1}\partial \vartheta(z)$ is empty or $0$. Thus, it holds $\widehat{\xi}^*=0$. Together with the first two equalities in \eqref{Delta-ineq}, we obtain $\|(\widehat{\Delta}_1;\widehat{\Delta}_2;\widehat{\xi}^*)\|=0$, which is a contradiction to $\|(\widehat{\Delta}_1;\widehat{\Delta}_2;\widehat{\xi}^*)\|=1$. Thus, we finish the proof. 
\end{proof}
\begin{remark}\label{remark-KLcond}
We notice that \cite[Theorem 3.2]{LiPong2018} identifies the KL property of exponent $q\in[0,1)$ for the composite  function $\Xi=\phi\circ H$ at $w^*\in W^*$ in terms of the KL property of $\phi$ with exponent $q$ at $(w^*,F(x^*))$ and the surjectivity of the mapping $H'(w^*)\!:\mathbb{W}\to\mathbb{U}$. By using $v^*=0$ and $\nabla F(x^*)=(\mathcal{E}^{\dag})^*(Q^*)^{\top}\mathcal{F}^*$ for any $w^*\in W^*$, it is not hard to verify that the latter is equivalent to requiring the implication \eqref{KL-imply} to hold for all $(\Delta_{x},\Delta_{z})\in\mathbb{X}\times\mathbb{Z}$. Obviously, the set $\big\{(\zeta,\xi)\in(N_{x^*}\mathcal{M}+\tau\nabla\!f(x^*))\times\tau\partial \vartheta(F(x^*))\ |\ \tau>0\}$ is much smaller than the whole space $\mathbb{X}\times\mathbb{Z}$. Then, Proposition \ref{KL-prop} provides a weaker condition than \cite[Theorem 3.2]{LiPong2018} for identifying the KL property of $\Xi$ with exponent $q$, as will be illustrated by the following example.

 Consider problem \eqref{prob} with $F(x)=(x_1+x_2,2x_2,\ldots, nx_n)^{\top}$, $f\equiv 0$, $\vartheta(z)=\sum_{i=1}^nh(z_i)$ with $h(t)=t^2\!-\!1$ for $t\notin\![-1,1]$ and $0$ for $t\in\![-1,1]$, and $\mathcal{M}=\!\{x\in\mathbb{R}^n\,|\,\|x\|=1\}$. Since $f\equiv 0$, when constructing the potential function $\Xi$, we remove the term $\langle s,v\rangle$. Then, $\Xi(w)=\vartheta(F(x)+Qv)+\delta_{T\mathcal{M}}(x,v)+\frac{\alpha^*}{2}\|v\|^2$ for $w=(x,v,Q)\in\mathbb{W}=\mathbb{R}^n\times\mathbb{R}^n\times\mathbb{R}^{n\times n}$. For $w^*=(x^*,v^*,Q^*)$ with $x^*=(1,0,\ldots,0)^{\top}\in\mathcal{M},v^*=0$ and any $Q^*\in\mathbb{R}^{n\times n}$, it is easy to check that $w^*\in(\partial\Xi)^{-1}(0)$. For this example, the implication \eqref{KL-imply} does not hold for  $(\Delta_{x},\Delta_{z})$ with $\Delta_{x}=-\nabla F(x^*)\Delta_{z}$ and $\Delta_{z}\in\mathbb{Z}\backslash\{0\}$, i.e., the criterion in \cite[Theorem 3.2]{LiPong2018} fails in checking the KL property of $\Xi$ with exponent $\frac{1}{2}$. Next we claim that $\Xi$ has the KL property of exponent $\frac{1}{2}$ at $w^*$ by Proposition \ref{KL-prop}.  To this end, consider any $(\Delta_{x},\Delta_{z})\in \big\{(\zeta,\xi)\in N_{x^*}\mathcal{M}\times\tau\partial \vartheta(F(x^*))\ |\ \tau>0\}$. Then there exist $\tau_1\in\mathbb{R}$ and $\tau_2\ge0$ such that $\Delta_{x}=(\tau_1,0,\ldots,0)^{\top}$ and $\Delta_{z}=(\tau_2,0,\ldots,0)^{\top}$. While by the expression of $F$, if $\Delta_{x}+\nabla F(x^*)\Delta_{z} = 0$ holds, then  $\tau_1=\tau_2=0$, which shows that the implication \eqref{KL-imply} holds. By Proposition \ref{KL-prop}, it suffices to argue that $\phi(w,z)=\vartheta(z)+\delta_{T\mathcal{M}}(x,v)+\frac{\alpha^*}{2}\|v\|^2$ for $w\in\mathbb{W}$ and $z\in\mathbb{R}^n$ has the KL property of exponent $1/2$ at $(w^*,F(x^*))$. Since $\vartheta$ is a finite piecewise linear-quadratic convex function, it is a KL function of exponent $\frac{1}{2}$, so we only need to argue that $g(w):=\delta_{T\mathcal{M}}(x,v)+\frac{\alpha^*}{2}\|v\|^2$ for $w\in\mathbb{W}$ has the KL property of exponent $\frac{1}{2}$ at $(x^*,v^*,Q^*)$. By Lemma \ref{lemma-Tb}, at any $(x,v,Q)\in{\rm dom}\,g$,
$\partial g(x,v,Q) =\{(ax+bv;bx+\alpha^*v)\ |\ a,b\in\mathbb{R}\}$. Fix $\delta=1$ and $\varpi=1$. Pick any $(x,v,Q)\in\overline{\mathbb{B}}((x^*,v^*,Q^*),\delta)\cap[g(x^*,v^*,Q^*)<g<g(x^*,v^*,Q^*)+\varpi]$. Obviously, $x\in\mathcal{M}$ and $v\in T_{x}\mathcal{M}$. Then, by noting that $g(x,v,Q)=\frac{\alpha^*}{2}\|v\|^2$ and $x\in\mathcal{M}$, 
\begin{align*}
{\rm dist}^2(0,\partial g(x,v,Q))&=\min_{a,b\in\mathbb{R}}\big[a^2+(b^2+(\alpha^*)^2)\|v\|^2+b^2\big]\\
	&=(\alpha^*)^2\|v\|^2\geq 2\alpha^*\big[g(x,v)-g(x^*,v^*)\big],
\end{align*}		
which shows that the function $g$ has the KL property with exponent $\frac{1}{2}$ at $w^*$.
\end{remark}

\begin{remark}\label{remark2-KLcond}
 Although Proposition \ref{KL-prop} provides a condition for identifying the KL property of $\Xi$ with an explicit exponent $q\in[0,1)$, the condition is not always satisfied. In this case, if $\Xi$ is semialgebraic, by \cite[Example 1 (b)]{Attouch2009} there exists a constant $q\in[0,1)$ such that $\Xi$ satisfies the KL property of exponent $q$, which implies that Algorithm \ref{iRVM} at least has a local sublinear convergence rate though the explicit rate is unknown. Obviously, for Examples \ref{Example1}-\ref{Example4}, the function $\Xi$ is semi-algebraic if $f$ and $F$ are semialgebraic.
 \end{remark}
\section{Dual SNCG for the subproblems}\label{sec6}

This section is devoted to the solving of subproblem \eqref{subprobj} with $\mathcal{Q}_{k,j}$ from Theorem \ref{oracle}. Such $\mathcal{Q}_{k,j}$ is helpful to handle the term $\vartheta(\ell_{F}(x^k\!+v;x^k))$ in \eqref{subprobj} and induce a desirable dual. Fix any $k\in\mathbb{N}$ and $j\in[j_k]$. From \eqref{dualkj}, the dual of subproblem \eqref{subprobj} is 
\begin{equation*}
 \max_{\zeta\in\mathbb{Z}}\Psi_{k,j}(\zeta):=\min_{v\in T_{x^k}\mathcal{M},z\in\mathbb{Z}}\mathcal{L}_{k,j}(v,z,\zeta),
 \end{equation*}
where $\mathcal{L}_{k,j}$ is the Lagrange function defined by \eqref{la-func}. Let $\Phi_{k,j}:=-\Psi_{k,j}$. An elementary calculation yields the dual of \eqref{subprobj} in the minimization form as follows
\begin{equation}\label{EDsubprobj}
\min_{\zeta\in \mathbb{Z}}\Phi_{k,j}(\zeta)=\frac{1}{2\alpha_{k,j}}\|\Pi_{T_{x^k}\mathcal{M}}(\nabla\!F(x^k)\zeta+\!\nabla\! f(x^k))\|^2+\dfrac{1}{2\beta_{k}}\|\zeta\|^2-e_{\beta_{k}^{-1}\vartheta}(F(x^k)+\beta_{k}^{-1}\zeta).
\end{equation}
Due to the strong convexity of \eqref{subprobj}, the strong duality holds for \eqref{subprobj} and \eqref{EDsubprobj}. Since $\Phi_{k,j}$ is a smooth convex function, one can achieve an optimal solution of \eqref{EDsubprobj} and then the unique optimal solution $\overline{v}^{k,j}$ of \eqref{subprobj} by finding a root of the system
\begin{equation}\label{root}
\nabla\Phi_{k,j}(\zeta)=\alpha_{k,j}^{-1}F'(x^k)\Pi_{T_{x^k}\mathcal{M}}(\nabla F(x^k)\zeta+\!\nabla\! f(x^k))+\mathcal{P}_{\!\beta_{k}^{-1}\vartheta}(F(x^k)+\beta_{k}^{-1}\zeta)-F(x^k)=0.
\end{equation}
That is, if $\zeta^*$ is a root of \eqref{root}, then $\overline{v}^{k,j}\!=-\alpha_{k,j}^{-1}\Pi_{T_{x^k}\mathcal{M}}(\nabla F(x^k)\zeta^*\!+\nabla\!f(x^k))$ is the unique optimal solution of \eqref{subprobj}. Let $v^{l}\!:=-\alpha_{k,j}^{-1}\Pi_{T_{x^k}\mathcal{M}}(\nabla F(x^k)\zeta^{l}+\!\nabla\!f(x^k))$ for each $l\in\mathbb{N}$, where $\{\zeta^{l}\}_{l\in\mathbb{N}}$ is an iterate sequence yielded by a solver to the dual \eqref{EDsubprobj}. Clearly, $v^{l}\in T_{x^k}\mathcal{M}$. By the weak duality, $\Theta_{k,j}(\overline{v}^{k,j})-f(x^k)\ge -\Phi_{k,j}(\zeta^{l})$ for all $l\in\mathbb{N}$. Then, with $\Theta_{k,j}^{\rm LB}$ taking the value $f(x^k)-\Phi_{k,j}(\zeta^{l})$, the pair $(v^{l},\Theta_{k,j}^{\rm LB})$ satisfies \eqref{inexact-cond} whenever
\begin{equation}\label{stop-algA}
\Theta_{k,j}(v^{l})\leq \Theta_{k,j}(0)\ \ {\rm and}\ \ \Theta_{k,j}(v^{l})+\Phi_{k,j}(\zeta^{l})-f(x^k)\le\frac{\mu_k}{2}\|v^{l}\|^2. 
\end{equation}
This means that any solver to the dual \eqref{EDsubprobj} can be terminated at the $l$th iteration whenever the vector $v^{l}$ corresponding to $\zeta^{l}$ is such that \eqref{stop-algA} holds. In other words, the conditions in \eqref{stop-algA} can be used as a stopping criterion for any solver to the dual \eqref{EDsubprobj}. 

The convexity of $\vartheta$ implies the Lipschitz continuity of the mapping $\mathcal{P}_{\!\beta_{k}^{-1}\vartheta}$. Further, from  \cite[Proposition 3.1 (b)]{Ioffe2009} and \cite[Theorem 1]{Bolte2009}, $\mathcal{P}_{\!\beta_{k}^{-1}\vartheta}$ is semismooth if $\vartheta$ is definable in an o-minimal structure over the real field, satisfied by the function $\vartheta$ corresponding to Examples \ref{Example1}-\ref{Example3}. Together with \cite[Proposition 7.4.4]{Facchinei2003}, the associated system \eqref{root} is semismooth. This inspires us to seek a root of \eqref{root} with the semismooth Newton method. From \cite{Hiriart-Urruty1984}, the generalized Hessian of $\Phi_{k,j}$ at $\zeta$ is defined by $\partial^2\Phi_{k,j}(\zeta)\!:=\partial_{C}(\nabla\Phi_{k,j})(\zeta)$, where $\partial_{C}(\nabla\Phi_{k,j})(\zeta)$ is the Clarke's Jacobian of $\nabla\Phi_{k,j}$ at $\zeta$. Since the characterization of $\partial_{C}(\nabla\Phi_{k,j})(\zeta)$ is not an easy task, we replace it with
\begin{equation}\label{GHessian}
 \widehat{\partial}^2\Phi_{k,j}(\zeta)d\!:=\alpha_{k,j}^{-1}F'(x^k)\Pi_{T_{x^k}\mathcal{M}}(\nabla\! F(x^k)d)+\!\beta_{k}^{-1}\partial_C \mathcal{P}_{\!\beta_{k}^{-1}\vartheta}(F(x^k)+\!\beta_{k}^{-1}\zeta)d\ \ \forall d\in\mathbb{Z}.
\end{equation}
From \cite[Page 75]{Clarke1983}, for any $d\in\mathbb{Z}$, $\partial^2\Phi_{k,j}(\zeta)d\subset\widehat{\partial}^2\Phi_{k,j}(\zeta)d$. The iterations of the semismooth Newton method for \eqref{EDsubprobj} are described in Algorithm \ref{SNCG}. Its global and local convergence analysis can be found in \cite[Proposition 3.3 \& Theorem 3.4]{ZhaoSun2010}. 
\begin{algorithm}[h]
 \begin{algorithmic}[1]
 \renewcommand{\thealgorithm}{A}
 \caption{\label{SNCG}{\bf (Semismooth Newton-CG)}}
 \State Input: $k,j\in\mathbb{N}$, $\overline{\eta}\in(0,1),\tau\in(0,1], \tau_1,\tau_2\in(0,1),0<\varrho_1<\frac{1}{2},\delta\in(0,1),\zeta^0\in\mathbb{Z}$.
		
 \For {$l=0,1,2,\ldots$}
		
 \State Given a maximum number of CG iterations $n_{l}$. Apply the practical CG $(\eta_{l},n_l)$ \hspace*{0.56cm}in \cite[Algorithm 1]{ZhaoSun2010} with $\eta_{l}\!:=\min \{\overline{\eta},\|\nabla\Phi _{k,j}(\zeta^{l})\|^{1+\tau}\}$ to find an approximate \hspace*{0.56cm} solution $d^{l}$ to \eqref{linear-system} with $\mathcal{V}_l\in\widehat{\partial}^2\Phi_{k,j}(\zeta^{l})$ and $\varepsilon_{l}\!:=\tau_1\min\{\tau_2,\|\nabla\Phi_{k,j}(\zeta^{l})\|\}$:
 \begin{equation}\label{linear-system}
 (\mathcal{V}_{l}+\varepsilon_{l}\mathcal{I})d=-\nabla\Phi_{k,j}(\zeta^{l}).
 \end{equation}

\State  Let $m_{l}$ be the first nonnegative integer $m$ such that
		\begin{equation*}
			\Phi_{k,j}(\zeta^{l}+\delta^{m}d^{l})\leq \Phi_{k,j}(\zeta^{l}) + \varrho_1\delta^{m}\langle\nabla\Phi_{k,j}(\zeta^{l}),d^{l}\rangle.
		\end{equation*}
		\State Set $\zeta^{l+1}=\zeta^{l}+\delta^{m_{l}}d^{l}$.
		\EndFor	
	\end{algorithmic}
\end{algorithm}
\begin{remark}\label{remark-AlgA}
 {\bf(a)} The convexity of $\vartheta$ means that every element of $\partial_C \mathcal{P}_{\!\beta_{k}^{-1}\vartheta}(F(x^k)\!+\!\beta_{k}^{-1}\zeta)$ is positive semidefinite. Therefore, for  $\mathcal{V}_l\in\widehat{\partial}^2\Phi_{k,j}(\zeta^{l})$, the linear mapping $\mathcal{V}_{l}+\varepsilon_{l}\mathcal{I}$ from $\mathbb{Z}$ to $\mathbb{Z}$ is positive definite, so the conjugate gradient (CG) method is able to solve \eqref{linear-system} efficiently. 

 \noindent
 {\bf(b)} From the above discussion, when Algorithm \ref{SNCG} is applied to solve the subproblems, we terminate its iteration at the iterate $\zeta^{l}$ whenever the condition \eqref{stop-algA} is satisfied. Its parameters are chosen to be $n_{l}\equiv 100,\overline{\eta}=10^{-2},\tau=0.1$, $\tau_1=1,\tau_2=10^{-3},\varrho_1=10^{-4}$ and $\delta=0.5$.   
\end{remark} 
 
Now we take a look at the choice of $\mathcal{V}_{l}\in\widehat{\partial}^2\Phi_{k,j}(\zeta^{l})$, which is used in the numerical experiments of Section \ref{sec7.2} for Examples \ref{Example1}-\ref{Example3}. Fix any $d\in\mathbb{Z}$. From the above \eqref{GHessian}, the element $\mathcal{V}_{l}d$ is the sum of $\alpha_{k,j}^{-1}F'(x^k)\Pi_{T_{x^k}\mathcal{M}}(\nabla\!F(x^k)d)$ and $\beta_{k}^{-1}\mathcal{W}_{l}d$ with some $\mathcal{W}_{l}$ from the set $\partial_C \mathcal{P}_{\!\beta_{k}^{-1}\vartheta}(F(x^k)\!+\beta_{k}^{-1}\zeta)$. Observe that $\alpha_{k,j}^{-1}F'(x^k)\Pi_{T_{x^k}\mathcal{M}}(\nabla\! F(x^k)d)$ is known once the expression of the mapping $\Pi_{T_{x^k}\mathcal{M}}(\cdot)$ is available. For the manifold $\mathcal{M}$ in Examples \ref{Example1}-\ref{Example2}, by \cite[Example 3.6.2]{Absil2008} the operator $\Pi_{T_{x^k}\mathcal{M}}(\cdot)$ is specified as
\[
\Pi_{T_{x^k}\mathcal{M}}(u)=(I_{n}-x^k(x^k)^{\top})u+x^k{\rm skew}((x^k)^{\top}u)\ \ {\rm for}\ u\in\mathbb{R}^{n\times r};
\]
while for the $\mathcal{M}$ in Example \ref{Example3}, by virtue of \cite[Proposition 2]{GaoSon2021} it takes the form of
\[
\Pi_{T_{x^k}\mathcal{M}}(u)=u+J_{2n}x^ku^{*},
\]
where $u^*$ is the unique solution of the following Lyapunov equation of variable $u$:
\[
(x^{k})^{\top}x^ku+u(x^{k})^{\top}x^k =u^{\top}J_{2n}x^k-(J_{2n}x^k)^{\top}u.
\]
Next we focus on the choice of $\mathcal{W}_{l}\in\partial_C \mathcal{P}_{\!\beta_{k}^{-1}\vartheta}(F(x^k)+\beta_{k}^{-1}\zeta)$. For Example \ref{Example1}, recall that $\vartheta(Z)=\lambda\|Z\|_1$ for $Z\in\mathbb{S}^n$, so we choose $\mathcal{W}_{l}$ in the following way 
\[
 [\mathcal{W}^ld]_{ij}=\left\{\begin{array}{cl}
 d_{ij} &\ {\rm if}\ [F(X^k)\!+\!\beta_{k}^{-1}\zeta^l]_{ij}\ge\beta_{k}^{-1}\lambda,\\
 0 &\ {\rm otherwise}
 \end{array}\right.\ \ {\rm for}\ i,j\in[n]_{+}.
\]
For Example \ref{Example2}, since $\vartheta(X;Z)=\lambda\|X\|_{2,1}+\rho\|Z\|_1$ and $F(X)=(X;E\circ(X^{\top}B^{\top}BX))$ for $X\in\mathbb{R}^{n\times r}$ and $Z\in\mathbb{S}^r$, it holds $\mathcal{W}_{l}d=[\mathcal{W}_{l}^{1}d^1;\mathcal{W}_{l}^{2}d^2]$ for $d=[d^1;d^2]\in\mathbb{R}^{n\times r}\times\mathbb{S}^r$ with $\mathcal{W}_{l}^{1}\in\mathcal{P}_{\!\beta_{k}^{-1}\lambda\|\cdot\|_{2,1}}(X^k\!+\beta_{k}^{-1}\zeta^l_1)$ and $\mathcal{W}_{l}^{2}\in\mathcal{P}_{\!\beta_{k}^{-1}\rho\|\cdot\|_{1}}(Y^k\!+\beta_{k}^{-1}\zeta^l_2)$ for $\zeta^{l}=[\zeta_1^{l};\zeta_2^{l}]\in\mathbb{R}^{n\times r}\times\mathbb{S}^r$, where $Y^k=E\circ((X^k)^{\top}B^{\top}BX^k)$. We choose $\mathcal{W}_{l}^{1}$ and $\mathcal{W}_{l}^{2}$ by
\begin{align*}
 [\mathcal{W}_{l}^{1}d^1]_{i\cdot}&=\left\{\begin{array}{cl}
  \left(1-\frac{\beta_{k}^{-1}\lambda}{\|H^{k,l}_{i\cdot}\|_2}\right)d^1_{i\cdot}+\frac{\beta_{k}^{-1}\lambda H^{k,l}_{i\cdot}(d^1_{i\cdot})^{\top}}{\|H^{k,l}_{i\cdot}\|_2^3}H^{k,l}_{i\cdot}&\ {\rm if}\ \|H^{k,l}_{i\cdot}\|_2>\beta_{k}^{-1}\lambda,\\
    0  &\ {\rm otherwise};
  \end{array}\right.\\
 [\mathcal{W}_{l}^{2}d^2]_{ij}&=\left\{\begin{array}{cl} 
  1 &\ {\rm if}\ [Y^k\!+\!\beta_{k}^{-1}\zeta^l_2]_{ij}\ge \beta_{k}^{-1}\rho,\\
  0 &\ {\rm otherwise},
 \end{array}\right.\ \ {\rm for}\ i,j\in[r]_{+},
\end{align*}
where $H^{k,l}\!:=X^k\!+\beta_{k}^{-1}\zeta^l_1$ for each $l\in\mathbb{N}$. For Example \ref{Example3}, since $\vartheta(Z;X)=\|Z\|_{F}+\lambda\|X\|_1$ and $F(X)=(X;XX^{+}A-A)$, it holds $\mathcal{W}_{l}d=[\mathcal{W}_{l}^{1}d^{1};\mathcal{W}_{l}^{2}d^2]$ for $d=[d^1;d^2]\in\mathbb{R}^{2n\times 2m}\times\mathbb{R}^{2n\times 2r}$ with $\mathcal{W}_{l}^{1}\in\mathcal{P}_{\!\beta_{k}^{-1}\|\cdot\|_{F}}(X^k\!+\!\beta_{k}^{-1}\zeta^l_1)$ and $\mathcal{W}_{l}^{2}\!\in\mathcal{P}_{\!\beta_{k}^{-1}\lambda\|\cdot\|_{1}}(Y^{k}\!+\!\beta_{k}^{-1}\zeta^l_2)$ for $\zeta^{l}=[\zeta_1^{l};\zeta_2^{l}]\in\mathbb{R}^{2n\times 2m}\times\mathbb{R}^{2n\times 2r}$. The elements $\mathcal{W}_{l}^{1}$ and $\mathcal{W}_{l}^{2}$ are chosen by
\begin{align*}
 \mathcal{W}_{l}^{1}d^1&=\left\{\begin{array}{cl}
  \left(1-\frac{\beta_{k}^{-1}}{\|H^{k,l}\|_F}\right)d^1+\frac{\beta_{k}^{-1}{\rm tr}[(H^{k,l})^{\top}d^1]}{\|H^{k,l}\|_F^3}H&\ {\rm if}\ \|H^{k,l}\|_F\ge\beta_{k}^{-1},\\
    0  &\ {\rm otherwise};
  \end{array}\right.\\
 [\mathcal{W}_{l}^{2}d^2]_{ij}&=\left\{\begin{array}{cl} 
  1 &\ {\rm if}\ [Y^k\!+\!\beta_{k}^{-1}\zeta^l_2]_{ij}\ge \beta_{k}^{-1}\lambda,\\
  0 &\ {\rm otherwise},
 \end{array}\right.\ \ {\rm for}\ (i,j)\in[2n]_{+}\times[2r]_{+},
\end{align*}
where $H^{k,l}\!:=X^k+\beta_{k}^{-1}\zeta^l_1$ for each $l\in\mathbb{N}$ and $Y^k=X^k(X^k)^{+}A-A$.

For the numerical experiments of the next section, we will achieve a pair $(v^{l},\Theta_{k,j}^{\rm LB})$ satisfying \eqref{inexact-cond} with Algorithm \ref{SNCG} applied to system \eqref{root}, rather than Algorithm DFG \cite{Necoara2016} applied to \eqref{EDsubprobj}, though the oracle complexity of Algorithm \ref{iRVM} armed with Algorithm DFG was established in Theorem \ref{oracle}. Figure \ref{DFO-SNCG} accounts for this, where RiVMPL-SNCG and RiVMPL-DFG are Algorithm \ref{iRVM} with the subproblems solved by Algorithm \ref{SNCG} and Algorithm DFG, respectively, and the parameters of Algorithm \ref{iRVM} are chosen in the same way as in Section \ref{sec7.1}. Notice that Algorithm DFG in \cite{Necoara2016} is precisely the FISTA proposed in \cite{beck2009fast} for solving \eqref{EDsubprobj}.
\begin{figure}[h]
 \centering
 \includegraphics[scale=0.45]{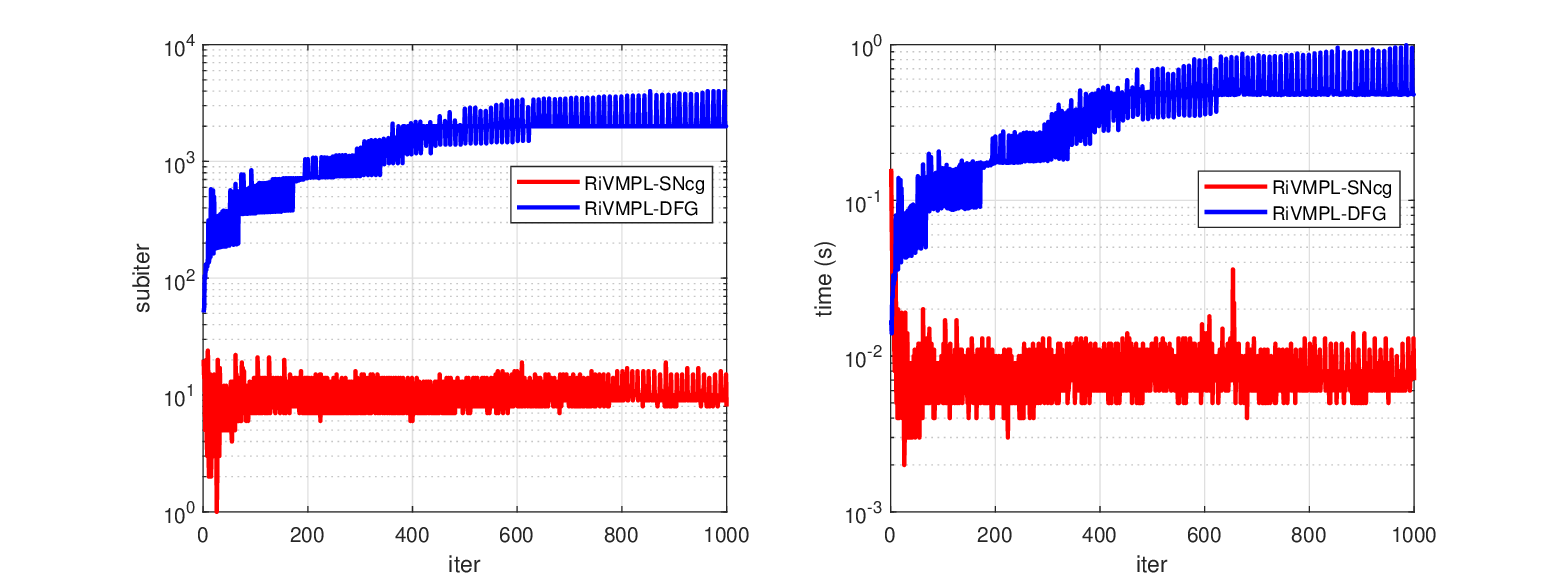}
 \caption{The total iterations and time in seconds for all subproblems at each iteration of RiVMPL-SNCG and RiVMPL-DFG for solving the synthetic example from Section \ref{sec7.3.2} with $(m,n,q)=(50,10^3,5)$ and $(\lambda,\rho)=(2.05,0.5)$.}
 \label{DFO-SNCG}
\end{figure}   
\section{Numerical experiments}\label{sec7}

This section tests the performance of Algorithm \ref{iRVM} armed with Algorithm \ref{SNCG} (known as RiVMPL) for solving Examples \ref{Example1}-\ref{Example3}, and compare its performance with that of RADMM \cite{LiMa2024} and RiALM \cite{Xu2024}. Our code can be downloaded from \url{https://github.com/SCUT-OptGroup/iVM_ManComp}. Since Algorithm \ref{iRVM} is an inexact and variable metric version of ManPL \cite{Wang2022}, we do not make comparison with it. 
\subsection{Choice of parameters}\label{sec7.1}

A good initial estimation for $\overline{L}_k$ in Section \ref{sec3} can reduce the computation cost of the inner loop greatly. Consider that ${\rm lip}\,\vartheta(F(x^k))$ is usually known. We choose
\[
\alpha_{k,0}=0.2\min\big\{\max\big\{{\rm lip}\,\vartheta(F(x^k))L_{\nabla\! F,k}\!+\!L_{\nabla\!f,k},\alpha_{\rm min}\big\},\alpha_{\rm max}\big\}\ \ {\rm for}\ k\in\mathbb{N}^*,
\]
where the coefficient $0.2$ aims at capturing tighter estimation, and $L_{\nabla\!F,k}$ and $L_{\nabla\!f,k}$ are the estimation for ${\rm lip}\,F'(x^k)$ and ${\rm lip}\,\nabla\!f(x^k)$ by the Barzilai-Borwein rule \cite{Barzilai1988}:
\begin{align}\label{BB-Lf0}
 L_{\nabla F,k}=\frac{\|\nabla\!F(x^k)\zeta^{k,j}\|}{\|\zeta^{k,j}\|}\ \ {\rm and}\ \
 L_{\nabla\!f,k}=\frac{\|\Delta y^k\|^2}{|\langle\Delta x^k,\Delta y^k\rangle|},
\end{align}
with $\zeta^{k,j}$ being the output of Algorithm \ref{SNCG} to solve \eqref{subprobj}, $\Delta x^k:=x^k-x^{k-1}$ and $\Delta y^k:=\nabla\!f(x^k)-\nabla\!f(x^{k-1})$. For the initial $\alpha_{0,0}$, we use $0.5\|A\|$ for the experiments in Section \ref{sec7.3.1},  $0.5\|B^{\top}B\|$ for those in Section \ref{sec7.3.2}, and $10^{-5}$ for those in Section \ref{sec7.3.3}. 

The choice of $\mu_k$ involves a trade-off between the computation cost of the inner loop and the quality of the iterates. It is reasonable to require the iterates to satisfy a more stringent inexactness restriction and then have better quality as the iteration proceeds. This inspires us to choose $\mu_k=\max\big\{\frac{\mu_{\rm max}}{\sqrt{k}},1\big\}$. Figure \ref{SSC_mu_max} below shows that the objective values yielded by Algorithm \ref{iRVM} with such $\mu_k$ are almost not affected by $\mu_{\rm max}\in[1,1000]$. The NMI scores (see Section \ref{sec7.3.1} for its definition) with $\mu_{\rm max}\in[1,1000]$ are basically the same, but the objective value with $\mu_{\rm max}\in[1,50]$ is a little higher than those with $\mu_{\rm max}\in[100,1000]$ and the running time of the former is more than that of the latter. Therefore, we always choose $\mu_{\rm max}=500$ for the subsequent tests. 
\begin{figure}[h]
 \centering
\includegraphics[scale=0.38]{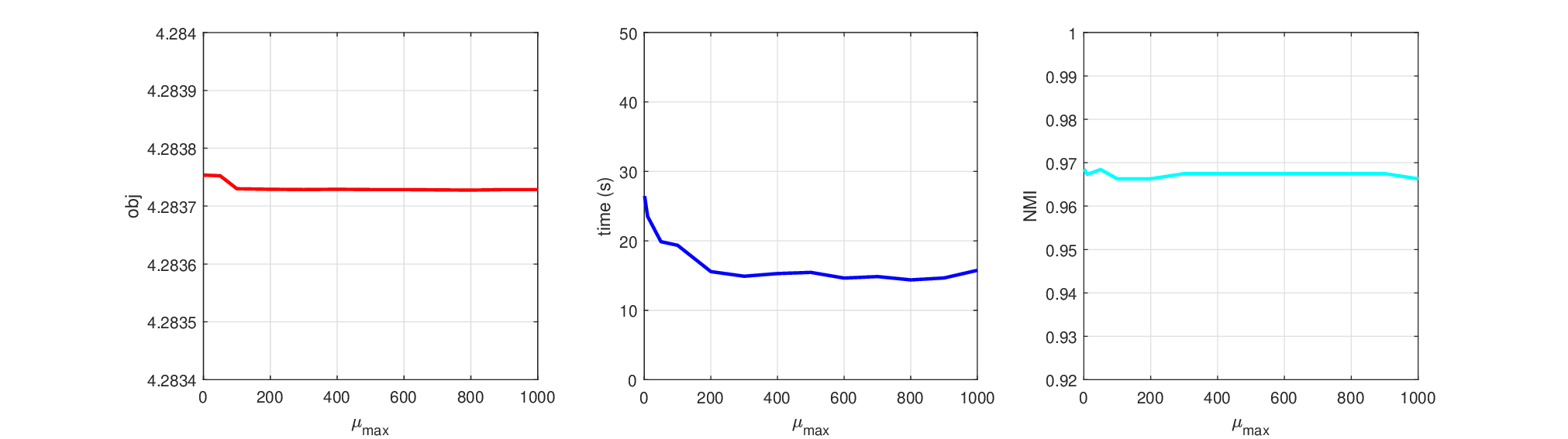}
\caption{The objective values, running time and NMI scores by RiVMPL with different $\mu_{\max}$ for the synthetic example in \cite[Part E.1]{Park2018} with $(n,q)=(300,6)$ and $\lambda=5\times 10^{-5}$.}
\label{SSC_mu_max}
\end{figure}  
\begin{figure}[h]
\centering
\includegraphics[scale=0.45]{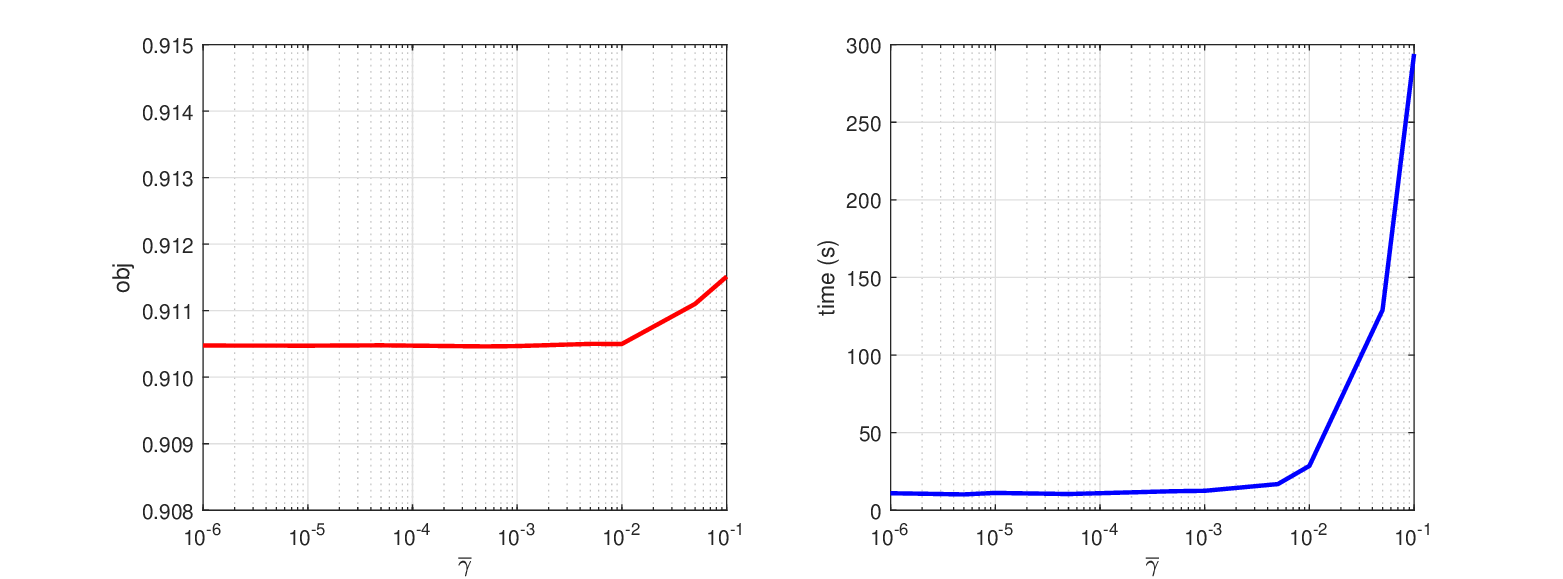}
\caption{ The objective value and running time of RiVMPL for problem \eqref{reg-psdprob} with $\lambda=0$ and the data matrix $A$ of type I in Section \ref{sec7.3.3}.}
\label{SDec_gamma}
\end{figure}

For the parameter $\overline{\gamma}$ in Step 6, Figure \ref{SDec_gamma} shows that it has a tiny influence on the objective value and the running time for $\overline{\gamma}<0.01$, but if $\overline{\gamma}>0.01$ the running time increases sharply. Inspired by this, we choose $\overline{\gamma}=10^{-5}$ for the subsequent tests. 

In addition, we choose the QR decomposition to be the retraction for the tests in Sections \ref{sec7.3.1}-\ref{sec7.3.2}, and the Cayley transformation in \cite{Bendokat2021} to be the retraction for the tests in Section \ref{sec7.3.3}. The other parameters of Algorithm \ref{iRVM} are chosen to be $\alpha_{\rm min}=10^{-6},\,\alpha_{\rm max}=10^{6},\,\overline{\alpha}=10^6$ and $\sigma=2.5$. 
For the parameter $\beta_{k}$ involved in the linear mapping $\mathcal{Q}_{k,j}=\alpha_{k,j}\mathcal{I}+\beta_{k}\nabla F(x^k)F'(x^k)$, we update it by the rule 
\[
\beta_{k+1}=\left\{\begin{array}{cl}
	\max\big\{\beta_k/1.1,10^{-6}\big\} &{\rm if}\ {\rm mod}(k,50)=0,\\
	\beta_k &{\rm otherwise}
\end{array}\right.\ \ {\rm with}\ \beta_0=0.01. 
\]
The smaller $\beta_0$ means that the subproblems are closer to model \eqref{prob}, but its solving becomes more difficult due to the weaker role of the proximal term $\frac{\beta_k}{2}\|F'(x^k)v\|^2$. 
\subsection{Validation of convergence rate}\label{sec7.2}

We validate the convergence rate of Algorithm \ref{iRVM} applied to Examples \ref{Example1}–\ref{Example3}. Figure \ref{con_rate} depicts its convergence behavior for solving these examples. In each subfigure, the blue and red curves plot the change of the function value and the iterate, respectively, as the iteration proceeds. Figure \ref{con_rate} shows that Algorithm \ref{iRVM} has a local linear convergence rate, most notably for Examples \ref{Example1}–\ref{Example2}, which is consistent with our theoretical results.

\begin{figure}[h]
\centering
\includegraphics[scale=0.38]{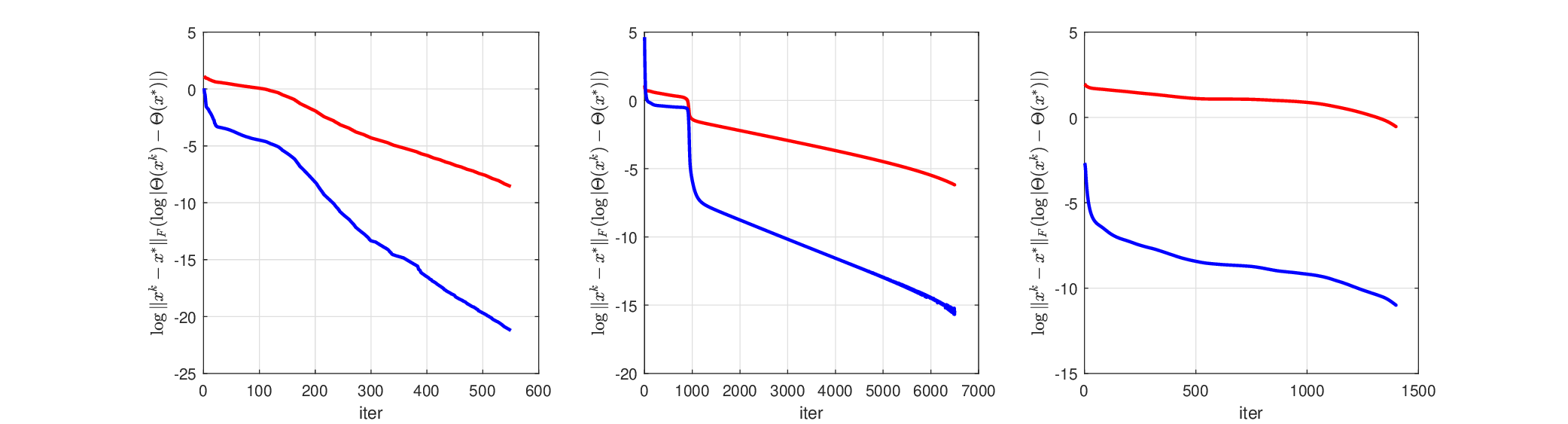}
\caption{\small Convergence behavior of Algorithm~\ref{iRVM} for Examples~\ref{Example1}–\ref{Example3}. Here $x^*$ is the final output of Algorithm~\ref{iRVM}. For Example~\ref{Example1} (left subfigure), we use the synthetic example from \cite[Part E.1]{Park2018} with $(n,q)=(300,6)$ and $\lambda=5\times10^{-5}$; for Example~\ref{Example2} (middle subfigure), we use the synthetic example in Section~\ref{sec7.3.2} with $(m,n,q)=(50,10^3,5)$ and $(\lambda,\rho)=(2.05,0.5)$; and for Example~\ref{Example3} (right subfigure), we use the data $A$ of type I in Section~\ref{sec7.3.3} with $\lambda=10^{-4}$.}
\label{con_rate}
\end{figure}
\subsection{Numerical comparisons}\label{sec7.3} 

The RADMM in \cite{LiMa2024} is originally proposed for problem \eqref{prob} with a linear $F$ but is applicable to \eqref{prob} itself (see \url{https://github.com/JasonJiaxiangLi/RADMM} for the code). Its basic idea is to replace the nonsmooth $\vartheta$ by its Moreau envelope $e_{\gamma\vartheta}$ and apply the ADMM to the resulting problem. The detailed iterations are seen in \cite[Algorithm 1]{LiMa2024}. For the subsequent tests, the parameters $\rho$ and $\gamma$ of RADMM are chosen to be $\rho=50$ and $\gamma=10^{-8}$ as used in \cite{LiMa2024}. Similar to \cite{LiMa2024}, the constant step-size $\eta_k\equiv \eta$ is used for the tests and its value is specified in the experiments. For the RiALM \cite{Xu2024}, since its code is unavailable, we implement it in the same way as described in \cite{Xu2024}, i.e., to seek the approximate $x^{k+1}$ by using the Riemannian gradient descent method with a backtrack line search and a Riemannian BB initial step-size as in \cite{Wen2013}. For the parameters $\sigma_1,b_1$ and $\varepsilon_1$ of RiALM \cite[Algorithm 1]{Xu2024}, we choose $\sigma_1=1.5$ and $b=1.5$ as suggested in \cite{Xu2024}, and $\varepsilon_1$ by the type of test problems to achieve better numerical results.  

Since the code of RADMM uses the objective values as the stopping criterion, to keep in step with it, we terminate the iterations of three methods at $x^k$ when
\begin{equation}\label{stop-cond}
\frac{|\Theta(x^{k})-\Theta(x^{k-1})|}{\max\{1,|\Theta(x^k)|\}}\le\epsilon^*.
\end{equation}
For each $k\in \mathbb{N}$ and $j\in[j_k]$, according to Step 6 of Algorithm \ref{iRVM}, its inner loop is terminated whenever $\Theta(R_{x^k}(v^{k,j})) \le \Theta_{k,j}(v^{k,j}) - \frac{\overline{\gamma}}{2} \|v^{k,j}\|^2$.
Figure \ref{r21PCA_obj_time} shows that the objective values by RiVMPL and RADMM have slower decrease, and their running times also have a gentle growth as the iteration proceeds; RiALM yields a lower objective value within the first $15$ iterations, but as the iteration proceeds, the objective value has a tiny improvement whereas its running time increases quickly. In view of this, we terminate the three methods under \eqref{stop-cond} with different $\epsilon^*$, i.e., $\epsilon_{\rm RiVMPL}^*,\epsilon_{\rm RiALM}^*$ and $\epsilon_{\rm RADMM}^*$ respectively for RiVMPL, RiALM and RADMM. We set the maximum number of iterations $k_{\rm RiVMPL}\!=5000,k_{\rm RiALM}=100$ and $k_{\rm RADMM}=10^5$ for RiVMPL, RiALM and RADMM. 
\begin{figure}[h]
\centering
\includegraphics[scale=0.5]{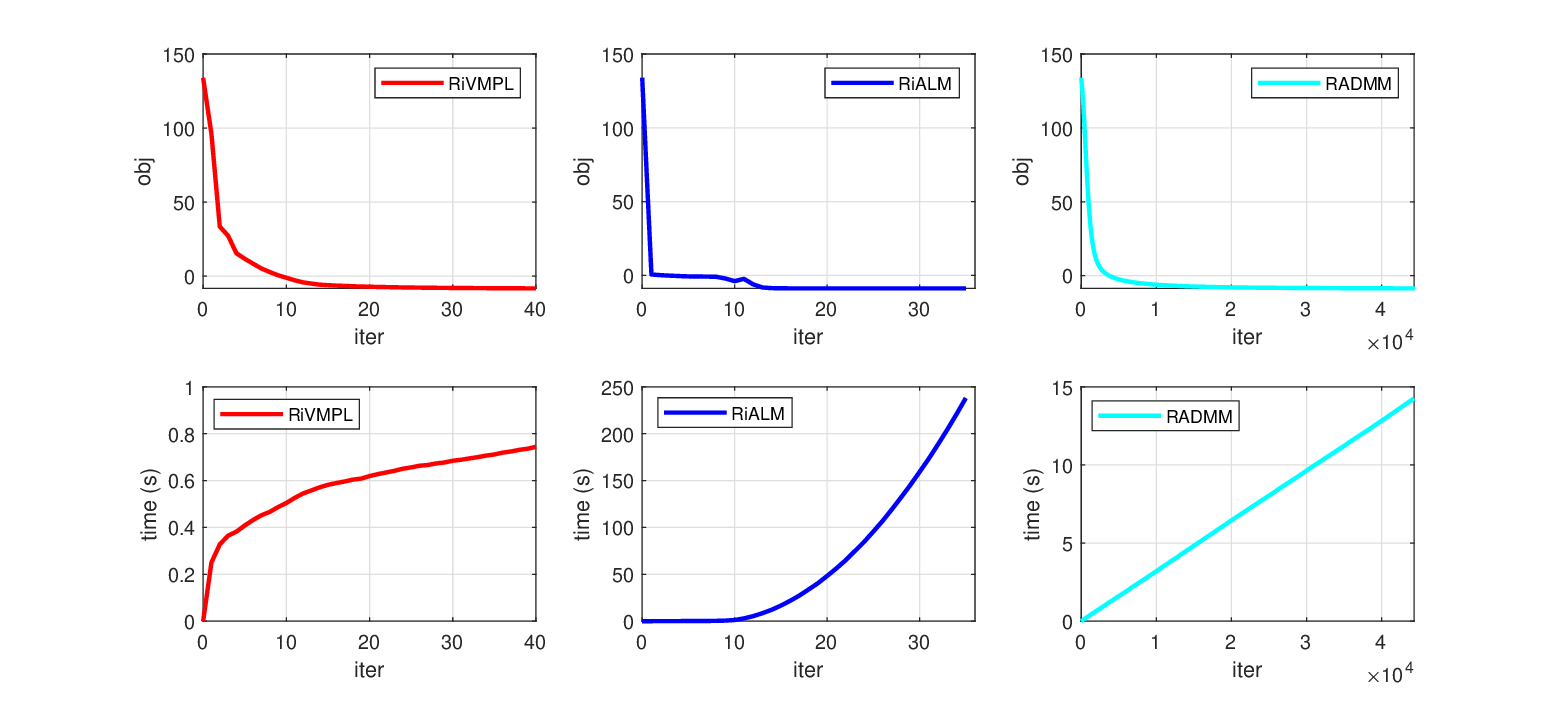}
\caption{The objective values and running time of three methods for the synthetic example in Section \ref{sec7.3.2} with $(m,n,q)=(50,10^3,5)$ and $(\lambda,\rho)=(2.05,0.5)$.}
	\label{r21PCA_obj_time}
\end{figure}

The next three subsections report the numerical results of RiVMPL, RiALM and RADMM for solving the problems from Examples \ref{Example1}-\ref{Example3} with synthetic and real data. All numerical tests are conducted on a desktop running 64-bit Windows System with an Intel(R) Core(TM) i5-8400 CPU 2.80GHz and 8.00 GB RAM on Matlab 2024b. All figures in this section are plotted with the average results of the total $\textbf{10}$ runs.
\subsubsection{Sparse spectral clustering}\label{sec7.3.1}

We apply the three methods to solve problem \eqref{ssc} with the real data sets used in \cite{Park2018}. For the data sets ``Macosko'' and ``Tasic'', we  extract those samples corresponding to the true label not greater than $\textbf{4}$ and $\textbf{9}$ respectively for testing. We follow \cite{Park2018} to construct the similarity matrices and compute the normalized Laplacian matrices $A$. To measure the effect of the clustering, we use the normalized mutual information (NMI) scores in \cite{Strehl2002}, and directly call the ``kmeans'' of Matlab to compute the label corresponding to the outputs of three methods. The higher NMI scores indicate the better clustering performance. We run RiVMPL with $\epsilon_{\rm RiVMPL}^*=10^{-8}$, RADMM with $\eta=0.1$ and $\epsilon_{\rm RADMM}^*=10^{-10}$, and RiALM with $\varepsilon_1=10^{-3}$ and $\epsilon_{\rm RiALM}^*=10^{-8}$.   
\begin{table*}[h]
 \centering
\caption{\centering Numerical results of the three methods for SSC problems with real data.}\label{realSSC-tab}
	\scalebox{0.79}[0.79]{
		\begin{tabular}{l|l|ll|ll|ll|ll} 
			\hline    
			\multicolumn{2}{c|}{Data}& \multicolumn{2}{c|}{Buettner}& \multicolumn{2}{c|}{Deng}& \multicolumn{2}{c|}{Schlitzer}& \multicolumn{2}{c}{Pollen}\\
			\hline
			\multicolumn{2}{c|}{$\lambda$}&$10^{-4}$&$10^{-5}$&$10^{-4}$&$10^{-5}$&$10^{-4}$&$10^{-5}$&$10^{-4}$&$10^{-5}$\\
			\hline
			\multirow{7}*{RiVMPL}
			&obj& 1.9920& 1.9717& 5.7533& 5.7365& 1.9937& 1.9661&9.8836&9.8513\\ 
			&NMI& \textbf{0.4834}& \textbf{0.8679}& 0.6876& 0.7217& 0.5000& 0.5580&0.9051&\textbf{0.9387}\\
			&spar& 0.0540& 0.0138& 0.1324& 0.0154& 0.2327& 0.0136&0.4509&0.0640\\
			&time(s)& 8.74& 7.29& 3.40& 3.61& 19.54& 6.46&14.45&12.56\\
            &iter&1490.1&2036.5&991.4&1513.2&1167.0&923.9&1285.8&1916.6\\
			&nsub& 1.96& 1.99&1.97& 1.99& 1.95& 1.98&1.97&1.99\\
			&nSN& 2.22& 1.99& 2.04& 1.99& 2.80& 1.98&2.35&1.99\\
			\hline
			\multirow{5}*{RADMM}
			&obj& 1.9927& 1.9717& 5.7533& 5.7365& 1.9945& 1.9661&9.8850&9.8512\\ 
			&NMI& 0.4620& 0.8458& 0.6738& \textbf{0.7238}& 0.4675& \textbf{0.5678}&0.8923&0.9316\\
			&spar& 0.0388& 0.0140& 0.1251& 0.0178& 0.1619& 0.0142&0.4216& 0.0664\\
			&time(s)& 2.70& 9.69& 2.21& 6.04& 11.08& 5.25&5.14& 14.00\\
            &iter&9194.2&35243.7&9128.7&24841.7&28500.5&13840.3&11549.8&31866.0\\
			\hline
			\multirow{6}*{RiALM}
			&obj& 1.9922& 1.9717& 5.7532& 5.7365& 1.9936& 1.9661&9.8802&9.8512\\ 
			&NMI& 0.4806& 0.8056& \textbf{0.6993}& 0.7098& \textbf{0.5078}& 0.5647&\textbf{0.9075}&0.9339\\
			&spar& 0.0725& 0.0155&0.1348& 0.0178& 0.2575& 0.0141&0.4972&0.0671\\
			&time(s)& 7.61& 3.69& 3.21& 2.01& 12.09& 1.90&25.75&23.38\\
            &iter&13.4&14.1&12.7&12.0&14.9&12.8&17.0&18.0\\
            &iter(sub)&1520.7&1014.4&993.9&788.0&1723.9&389.8&2855.7&3387.7\\
			\hline  
			\hline
			\multicolumn{2}{c|}{Data}& \multicolumn{2}{c|}{Ting}& \multicolumn{2}{c|}{Treutlin}& \multicolumn{2}{c|}{Macosko}& \multicolumn{2}{c}{Tasic}\\
			\hline
			\multicolumn{2}{c|}{$\lambda$}&$10^{-4}$&$10^{-5}$&$10^{-4}$&$10^{-5}$&$10^{-4}$&$10^{-5}$&$10^{-4}$&$10^{-5}$\\
			\hline
			\multirow{7}*{RiVMPL}
			&obj& 3.8529& 3.8399& 3.9205& 3.9098& 2.9578& 2.9003& 7.9883& 7.9441\\ 
			&NMI& \textbf{0.9780}& \textbf{0.9755}& 0.6267& 0.7395& \textbf{0.5713}& \textbf{0.7408}& 0.2130&\textbf{0.3048}\\
			&spar& 0.0714& 0.0187& 0.0302& 0.0059& 0.2619& 0.0501& 0.5932& 0.0904\\
			&time(s)& 1.33& 0.82& 0.81& 0.75& 241.42& 50.60& 75.50&56.49\\
            &iter&379.3&433.4&757.2&835.3&440.0&341.2&1524.6&2698.5\\
			&nsub& 1.97& 1.98& 1.98& 1.99& 1.95& 1.98& 1.94&1.98\\
			&nSN& 1.99& 1.98& 1.98& 1.99& 4.27& 1.98& 3.60&2.00\\
			\hline
			\multirow{5}*{RADMM}
			&obj& 3.8529& 3.8399& 3.9205& 3.9098& 2.9581& 2.9003& 7.9908&7.9440\\ 
			&NMI& 0.9755& \textbf{0.9755}& \textbf{0.6330}& \textbf{0.7455}& 0.5431& \textbf{0.7408}& 0.2043&0.2949\\
			&spar& 0.0703& 0.0187& 0.0290& 0.0061& 0.2616& 0.0505& 0.5069&0.0952\\
			&time(s)& 0.84& 1.10&0.89& 1.20& 190.87& 25.26& 51.00&43.14\\
            &iter&4092.4&5553.9&8099.4&11234.1&34896.0&4657.1&53273.1&41650.1\\
			\hline
			\multirow{6}*{RiALM}
			&obj& 3.8529& 3.8399& 3.9205& 3.9098& 2.9579& 2.9003& 7.9866&7.9440\\ 
			&NMI& 0.9755& \textbf{0.9755}& 0.6300& 0.7184&0.5585 & \textbf{0.7408}& \textbf{0.2337}&0.3001\\
			&spar& 0.0744& 0.0189& 0.0315& 0.0062& 0.2633& 0.0505&0.6377&0.0984\\
			&time(s)& 0.29& 0.22& 0.29& 0.34& 33.37& 4.60& 44.79&30.77\\
            &iter&10.5&10.3&11.0&11.9&12.2&10.0&14.1&14.2\\
            &iter(sub)&120.9&91.7&254.0&267.0&506.9&59.2&2371.8&2270.5\\
			\hline       
		\end{tabular}
	}
\end{table*}

Table \ref{realSSC-tab} reports the average results of the three methods for the $\textbf{10}$ trials from the same starting point $x^0$, generated by the Matlab function ``$\textrm{orth(randn(n,r))}$''. The ``spar'' rows report the approximate sparsity for the outputs of three solvers. Let $Z^{\rm out}=X^{\rm out}(X^{\rm out})^{\top}$ where $X^{\rm out}$ is the output of a solver. The approximate sparsity is defined as $|\{j\in[n^2]\ |\ |[{\rm vec}(Z^{\rm out})]_j|\le 10^{-4}\|{\rm vec}(Z^{\rm out})\|_{\infty}\}|/n^2$. The ``iter" rows report the average number of outer iterations (which corresponds to the total iterations for RADMM) and the ``iter(sub)" rows report the average number of iterations per subproblem for RiALM. The ``nsub'' and ``nSN'' rows report the average number of subproblems and iterations of Algorithm \ref{SNCG} required by each iteration of Algorithm \ref{iRVM}. RiVMPL yields the best NMIs for $\textbf{8}$ instances, while RADMM and RiALM return the best NMIs respectively for $\textbf{6}$ and $\textbf{6}$ instances; it yields the worst NMIs for $\textbf{2}$ instances, while RADMM and RiALM return the worst NMIs respectively for $\textbf{10}$ and $\textbf{4}$ instances. This shows that RiVMPL is superior to them in terms of NMIs. The running time of RiVMPL is comparable with that of RADMM, more than that of RiALM. The objective values by RiVMPL are better than those by RADMM but a little worse than those by RiALM. The sparsity by RiVMPL is comparable with the one yielded by RiALM and RADMM. 
\subsubsection{The $\ell_1$-norm penalty for constrained sparse PCA}\label{sec7.3.2}

We apply the three methods to solve problem \eqref{pen-rspca} with synthetic data $B$. The matrix $B$ is generated randomly in the same way as in \cite[Section 6.3]{Chen2020}, i.e., we first generate a random matrix $B$ using the MATLAB function ``$\textrm{randn(m,n)}$'', then shift the columns of $B$ so that their mean equals $0$, and lastly normalize the columns so that their Euclidean norms are $1$. For all tests in this subsection, the three methods start from the same point $x^0\in\mathcal{M}$ generated by the Matlab command ``$\textrm{orth(randn(n,r))}$'', and stop with $\epsilon_{\rm RiVMPL}^*=5\times 10^{-8},\epsilon_{\rm RiALM}^*=10^{-6}$, and $\epsilon_{\rm RADMM}^*=10^{-9},k_{\rm RADMM}=2\times 10^5$. 

We first examine the relation between the penalty parameter $\rho$ and the $\ell_1$-norm constraint violation of \eqref{rspca} for $\lambda=2.05$ and $(m,n,r)=(50,1000,5)$. Figure \ref{r21PCA_rho} shows that the three methods return the comparable constraint violation for $\rho\in[0.2,1]$, but for $\rho\in[0.05,0.2)$ the constraint violation by RiVMPL is much less than the one by RiALM and RADMM. This means that RiVMPL is more robust than the other two methods. Furthermore, the running time of RiVMPL is the least.  
\begin{figure}[h]
 \centering
 \includegraphics[scale=0.45]{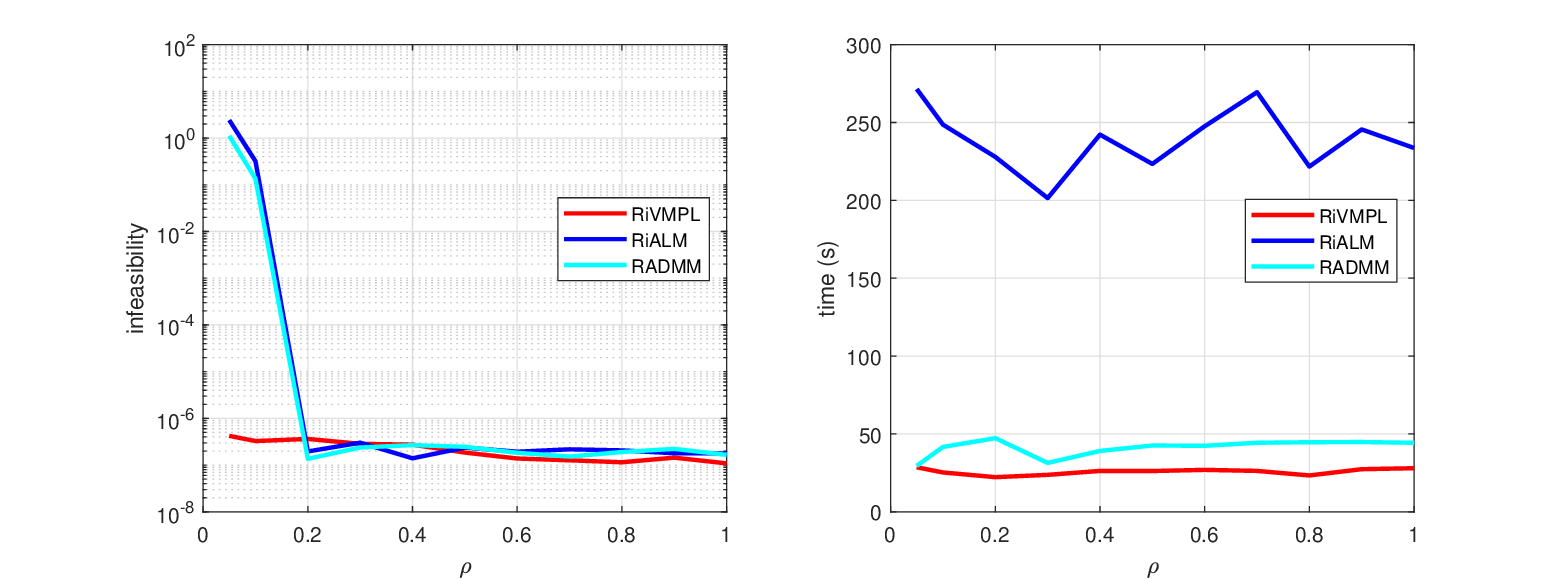}
 \caption{\small The infeasibility and running time of three methods for \eqref{pen-rspca} under different $\rho$.}
 \label{r21PCA_rho}
\end{figure}

Next we take problem \eqref{rspca} with $\rho=0.5$ and $(m,n,r)=(50,1000,5)$ for example to look at how the approximate row sparsity by the three methods varies with $\lambda$. Let $X^{\rm out}$ denote the output of a solver. The approximate row sparsity is defined as $|\{i\in[n]\ |\ \|X_{i\cdot}^{\rm out}\|\le 10^{-4}\max_{i\in[n]}\|X_{i\cdot}^{\rm out}\|\}|/n$. Figure \ref{r21PCA_lambda} indicates that the approximate row sparsity by the three methods is close to $0$ for $\lambda\le 1.0$, and increases gradually as $\lambda$ increases in $[1,2.5]$; when $\lambda$ is greater than $2.5$, the one by RiVMPL and RiALM increases to $1$ sharply, but the one by RADMM almost drops to zero. We also observe that as $\lambda$ increases in $[0.5,3.2]$, the running time of RiALM increases remarkably, but that of RiVMPL and RADMM has no too much variation. 
\begin{figure}[h]
	\centering
	\includegraphics[scale=0.45]{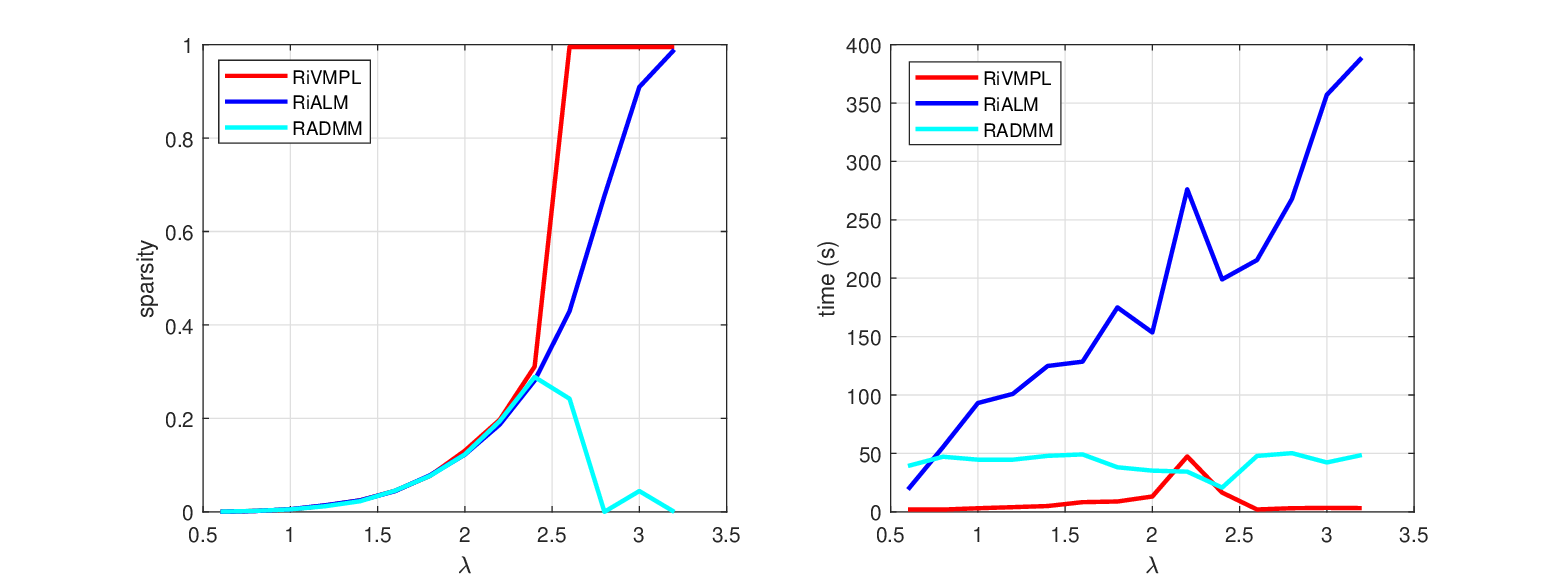}
	\caption{\small The row sparsity and running time of three methods under different $\lambda$.}
	\label{r21PCA_lambda}
\end{figure}
\begin{table*}[h]
	\centering
	\caption{\centering Numerical results of three methods for problem \eqref{pen-rspca} with $(m,\lambda,\rho)=(50,2.0,0.5)$.}\label{r21PCA-tab}
	\scalebox{0.8}[0.80]{
		\begin{tabular}{l|l|lll|lll} 
			\hline    
			\multicolumn{2}{c|}{$(n,r)$}& $(800,5)$& $(1000,5)$& $(1200,5)$& $(1000,3)$& $(1000,7)$& $(1000,9)$\\
			\hline
			\multirow{7}*{RiVMPL}
			&obj& \textbf{-2.6366}& \textbf{-11.7912}& \textbf{-22.2646}& 3.4671& \textbf{-34.5408}&\textbf{-60.8013}\\ 
			&rspar& 0.1964& 0.1307& 0.0917& 0.6040& 0.0266& 0.0045\\
			&time(s)& 17.85& 14.28& 20.58& 14.56&22.13&18.84\\
			&infeas& 5.55e-8& 3.03e-7& 7.92e-7&1.45e-8& 1.28e-6&4.42e-6\\
            &iter&2395.1&1631.0&1881.9&1430.9&1875.6&1249.1\\
			&nsub& 1.11& 1.19& 1.28& 1.00& 1.42&1.65\\
			&nSN& 8.86& 8.69& 9.49& 9.25&8.98&8.48\\
			\hline
			\multirow{5}*{RADMM}
			&obj& -2.3480& -11.6008& -22.1232& 3.7466& -34.4900&-60.7205\\ 
			&rspar& 0.1859& 0.1229& 0.0889& 0.4333& 0.0247&0.0045\\
			&time(s)& 30.88& 37.82& 40.75& 24.70& 72.15&109.11\\
			&infeas& 2.22e-7& 1.95e-7& 1.53e-7&1.22e-1& 1.01e-7&1.41e-7\\
            &iter&113156.1&115311.0&104598.8&200000.0&174445.9&200000.0\\
			\hline
			\multirow{6}*{RiALM}
			&obj& -2.5377& -11.6578& -22.0316& \textbf{3.4304}& -34.3402&-60.5792\\ 
			&rspar& 0.1843& 0.1225& 0.0866& 0.4554& 0.0230&0.0048\\
			&time(s)& 192.95& 168.42& 199.65& 198.35& 215.37&203.59\\
			&infeas& 8.50e-8& 2.13e-7& 4.00e-7& 5.80e-8& 5.87e-7&1.41e-6\\
            &iter&34.6&30.5&30.6&33.4&32.3&30.2\\
            &iter(sub)&3640.4&3475.7&3488.2&3662.2&3560.1&3459.5\\
			\hline   
		\end{tabular}
	}
\end{table*}

In view of Figures \ref{r21PCA_rho}-\ref{r21PCA_lambda}, we test the three methods for solving \eqref{pen-rspca} with $(m,\lambda,\rho)=(50,2.0,0.5)$ for different $n$ and $r$. Table \ref{r21PCA-tab} reports their average results for the $\textbf{10}$ trials, where the ``infeas" rows report the average constraint violations, and the terms ``iter'' and ``iter(sub)'' have the same meaning as in Table \ref{realSSC-tab}. We see that the objective values and the approximate row sparsity by RiVMPL are significantly better than those by RiALM and RADMM. Its running time is about half of that of RADMM and one-tenth that of RiALM. The advantage of RiVMPL in terms of computation time is attributed to its fewer total iterations, a count that is an order of magnitude lower than those of RiALM and RADMM. Notice that the primary computational cost per inner iteration of RiVMPL is incurred by one step of Algorithm~\ref{SNCG}, which requires \(O(nr^2)\) operations; whereas that for RiALM involves a Riemannian gradient descent step with a backtracking line search, costing approximately \(O(nr^2)\). This explains why RiVMPL  requires less computation times than RiALM in these tests. Comparing with the ``nsub'' and ``nSN'' rows of Table \ref{realSSC-tab}, every iteration of Algorithm \ref{iRVM} solves fewer subproblems but needs more semismooth Newton steps. The latter means that problem \eqref{pen-rspca} is more difficult than \eqref{ssc} due to the more complicated mapping $F$. 
\subsubsection{Proper symplectic decomposition}\label{sec7.3.3}

We apply the three methods to solve problem \eqref{reg-psdprob} with synthetic data $A$. The matrix $A$ is generated randomly in the following two ways: (I) to generate $A'=\textrm{randn}(2m,2n)$ and then set $A=A'/\|A'\|_F$; (II) to follow the procedure described in \cite{Jensen2024} to generate $A$. For all the tests of this subsection, the three methods start from the same point $x^0\in\mathcal{M}$ generated by using the same way as in \cite{Jensen2024}. Table \ref{SDec-tab} reports the average results of three methods for solving problem \eqref{reg-psdprob} with $\lambda=0$ in  the $\textbf{10}$ trials, under the stopping condition \eqref{stop-cond} with $\epsilon_{\rm RiVMPL}^*=10^{-7},\epsilon_{\rm RiALM}^*=5\times 10^{-5}$ and $\epsilon_{\rm RADMM}^*=10^{-7}$. We see that for the type I matrix \(A\), RiVMPL yields much better objective values than RiALM and RADMM in much shorter computation time; furthermore, the computational time advantage of RiVMPL stems from its lower total iteration count, which is an order of magnitude lower than those of RiALM and RADMM. While for the type II matrix $A$, it returns a little worse objective values than RiALM within comparable running time. Comparing with the ``nsub'' and ``nSN'' rows of Table \ref{r21PCA-tab}, we conclude that for this class of problems, each iteration of Algorithm \ref{iRVM} solves more subproblems and needs more semismooth Newton steps. This attributes to the high nonlinearity of $F$. 
\begin{table*}[h]
	\centering
	\caption{\centering Numerical results of three methods for problem \eqref{reg-psdprob} with $(m,\lambda)=(50,0)$.}\label{SDec-tab}
	\scalebox{0.8}[0.80]{
		\begin{tabular}{l|l|lll|lll} 
			\hline    
			Type I &$(n,r)$& $(500,5)$& $(500,10)$& $(500,15)$& $(400,10)$& $(600,10)$ & $(700,10)$\\
			\hline
			\multirow{5}*{RiVMPL}
			&obj& \textbf{0.92395}& \textbf{0.84952}& \textbf{0.77522} &\textbf{0.84487} &\textbf{0.85323} & \textbf{0.85611}\\
			&time(s)& 40.35& 48.89& 72.10 &31.12 &73.27 & 164.23\\
            &iter&145.8&165.7&204.2&156.7&167.9&176.3\\
			&nsub& 1.56& 1.32& 1.23 &1.22 &1.46 & 1.54\\
			&nSN& 38.65& 42.25& 47.25 &43.25 &41.80 & 42.30\\
			\hline
			\multirow{3}*{RADMM}
			&obj& 0.92661& 0.85232& 0.77859&0.84762 &0.85628 & 0.85958\\
			&time(s)& 171.09& 233.77& 289.34 &161.13 &353.88 & 484.61\\
            &iter&25140.2&32147.5&37176.0&30891.1&33803.9&35108.3\\
			\hline
			\multirow{4}*{RiALM}
			&obj& 0.92406& 0.84967& 0.77539 &0.84493 &0.85337 & 0.85633\\
			&time(s)& 197.94& 239.47& 276.52 &148.34 &329.64 & 499.28\\
            &iter&26.6&27.0&27.0&26.3&26.5&27.0\\
            &iter(sub)&1122.6&1227.3&1309.3&1106.3&1187.7&1308.3\\
			\hline 
			\hline    
			Type II &$(n,r)$& $(500,5)$& $(500,10)$& $(500,15)$& $(400,10)$& $(600,10)$ & $(700,10)$\\
			\hline
			\multirow{5}*{RiVMPL}
			&obj& 0.76047& 0.54545& 0.33098 &0.54594 &0.54852 & 0.54736\\
			&time(s)& 4.67& 9.30& 22.73 &6.33 &13.38 & 39.82\\
            &iter&50.9&62.6&78.0&61.7&79.1&87.2\\
			&nsub& 1.64& 1.38& 1.31 &1.39 &1.37 & 1.39\\
			&nSN& 18.37& 33.27& 56.82 &32.69 &31.72 & 32.75\\
			\hline
			\multirow{3}*{RADMM}
			&obj& 0.76058& 0.54568& 0.33114 &0.54613 &0.54878 & 0.54762\\
			&time(s)& 34.94& 43.47& 50.81 &31.28 &58.51 & 97.28\\
            &iter&5111.1&5485.6&6549.4&6057.6&6104.2&6226.0\\
			\hline
			\multirow{4}*{RiALM}
			&obj& \textbf{0.76022}& \textbf{0.54494}& \textbf{0.33042} &\textbf{0.54542} &\textbf{0.54805} & \textbf{0.54693}\\
			&time(s)& 4.61& 7.33& 7.29 &4.97 &13.08 & 15.00\\
            &iter&20.5&21.0&21.1&21.1&21.7&21.2\\
            &iter(sub)&32.8&45.8&44.4&45.8&57.1&44.2\\
			\hline 
		\end{tabular}
	}
\end{table*}

Table \ref{sSDec-tab1} reports the average results of three methods for solving problem \eqref{reg-psdprob} with $\lambda=10^{-4}$ in the $\textbf{10}$ trials under the stop condition \eqref{stop-cond} with $\epsilon_{\rm RiVMPL}^*=10^{-6},\epsilon_{\rm RiALM}^*=10^{-4}$ and $\epsilon_{\rm RADMM}^*=10^{-7}$. The objective values obtained by RiVMPL are superior to those of RADMM and comparable to those of RiALM, while RiVMPL achieves the highest sparsity. In terms of running time, RiVMPL performs comparably to RADMM, requiring at most one-half and one-tenth of the time consumed by RiALM for type I and type II data \(A\), respectively.  
\begin{table*}[h]
	\centering
	\caption{\centering Numerical results of three methods for problem \eqref{reg-psdprob} with $(m,\lambda)=(50,10^{-4})$.}\label{sSDec-tab1}
	\scalebox{0.8}[0.80]{
		\begin{tabular}{l|l|lll|lll} 
			\hline    
			Type I &$(n,r)$& $(500,5)$& $(500,10)$& $(500,15)$& $(400,10)$& $(600,10)$ & $(700,10)$\\
			\hline
			\multirow{6}*{RiVMPL}
			&obj& \textbf{0.96006}& 0.92129& 0.88301 & 0.90940&0.93188 & \textbf{0.93987}\\
			&spar&0.3068 & 0.3220 &0.3291& 0.2870& 0.3573& 0.3804\\
			&time(s)& 71.79& 128.95& 186.40 &93.08 &144.82 & 278.53\\
            &iter&220.5&374.5&490.8&349.6&365.9&350.4\\
			&nsub& 1.00& 1.00& 1.00 &1.00 &1.00 & 1.00\\
			&nSN& 77.76& 72.72& 69.55 &74.49 &72.03 & 70.37\\
			\hline
			\multirow{4}*{RADMM}
			&obj& 0.96318& 0.92606& 0.88958 &0.91364 &0.93698 & 0.94536\\
			&spar&0.2799 & 0.2815 & 0.2822& 0.2445& 0.3217& 0.3515\\
			&time(s)& 114.53& 143.52& 165.17 &97.30 &189.64 &304.74\\
            &iter&16137.1&18120.4&19196.5&17638.6&18925.1&20556.6\\
			\hline
			\multirow{5}*{RiALM}
			&obj& \textbf{0.96006}& \textbf{0.92112}& \textbf{0.88224} &\textbf{0.90902} &\textbf{0.93183} & 0.93993\\
			&spar&0.3002 & 0.3046 & 0.3042& 0.2695& 0.3442& 0.3698\\
			&time(s)& 228.18& 357.18& 560.88 &271.60 &413.34 & 574.05\\
            &iter&7.3&9.9&14.2&11.0&9.0&8.2\\
            &iter(sub)&4643.3&4797.0&4887.2&4793.8&4787.1&4811.3\\
			\hline 
			\hline    
			Type II &$(n,r)$& $(500,5)$& $(500,10)$& $(500,15)$& $(400,10)$& $(600,10)$ & $(700,10)$\\
			\hline
			\multirow{6}*{RiVMPL}
			&obj& 0.76535& \textbf{0.55391}& 0.34172 &0.55459 &0.55688 & \textbf{0.55583}\\
			&spar&0.9078 & 0.8651 & 0.7572& 0.7814& 0.9208& 0.9494\\
			&time(s)& 13.83& 25.23& 45.67 &18.24 &33.27 & 60.42\\
            &iter&56.9&63.9&61.3&56.3&68.6&65.9\\
			&nsub& 1.46& 1.44& 1.33 &1.38 &1.47 & 1.42\\
			&nSN& 41.01& 68.01& 92.82 &66.60 &67.84 & 67.34\\
			\hline
			\multirow{4}*{RADMM}
			&obj& 0.76585& 0.55502& 0.34306 &0.55563 &0.55822 & 0.55710\\
			&spar&0.9067 & 0.8631 &  0.7550& 0.7803& 0.9179& 0.9473\\
			&time(s)& 30.39& 31.43& 39.35 &27.05 &39.41 & 61.24\\
            &iter&4049.9&3898.1&4566.0&4642.5&3743.8&3890.2\\
			\hline
			\multirow{5}*{RiALM}
			&obj& \textbf{0.76520}& 0.55423& \textbf{0.34087} &\textbf{0.55403} &\textbf{0.55644} & 0.55599\\
			&spar&0.9048 & 0.8637 & 0.7599& 0.7822& 0.9194& 0.9484\\
			&time(s)& 294.68& 372.32& 462.99 &263.47 &506.35 & 706.82\\
            &iter&10.7&10.7&12.2&11.3&11.4&10.8\\
            &iter(sub)&3867.4&4283.7&4614.4&4354.9&4372.6&4327.8\\
			\hline 
		\end{tabular}
	}
\end{table*}

To sum up, for the SSC problem, RiVMPL outperforms both RADMM and RiALM in terms of NMIs, while showing slightly inferior performance to RiALM with respect to objective values, sparsities, and running times. For the constrained row-sparse PCA problem, RiVMPL is remarkably superior to RADMM and RiALM in both objective values and execution efficiency. Finally, for the proper symplectic decomposition, RiVMPL surpasses RADMM and RiALM in terms of sparsity and running time, while yielding objective values comparable to those of RiALM and superior to those obtained by RADMM.    
\section{Conclusion}\label{sec8}

For the composite problem \eqref{prob} with a $\mathcal{C}^2$-smooth embedded closed submanifold $\mathcal{M}$, we proposed an inexact variable metric proximal linearization method RiVMPL, and proved the full convergence of the iterate sequence under the boundedness assumption on the iterate sequence and the KL property of the potential function $\Xi_{\widetilde{c}}$. The idea of introducing a variable metric linear operator to tackle the nonlinearity of $F$ and using the inexactness criterion for the solving of subproblems is not new, but its full convergence analysis on the iterate sequence is nontrivial and needs to overcome new difficulty caused by the manifold constraint. As far as we know, this is the first algorithm with a full convergence certificate for \eqref{prob}. Moreover, under the boundedness assumption on the iterate sequence, we established the $O(\epsilon^{-2})$ iteration complexity and the $O(\epsilon^{-2})$ calls to the subproblem solver for returning an $\epsilon$-stationary point defined with a direct measure, and if the DFG in \cite{Necoara2016} to return the average primal iterate is used as an inner solver, the $O(\epsilon^{-3})$ oracle complexity bound was derived for the RiVMPL with $\mathcal{Q}_{k}$ specified as in Theorem \ref{oracle} to return an $\epsilon$-stationary point. We also provide a checkable condition for the potential function $\Xi_{\widetilde{c}}$ to satisfy the KL property with exponent $q\in[\frac{1}{2},1)$. Numerical tests validated the efficiency of the RiVMPL armed with the dual semismooth Newton method. 

\backmatter





\bmhead{Acknowledgements}

This work was supported by the National Natural Science Foundation of China under project No.12371299. The authors are highly grateful to the two anonymous reviewers for their invaluable comments and suggestions, which have significantly improved the quality of this paper.

\section*{Declarations}


\begin{itemize}
\item The authors declare that they have no conflict of interest to this work.
\item The data used for the test problems in Section \ref{sec7.3.1} is from \url{https://github.com/ishspsy/project/tree/master/MPSSC}, and the data used for the test problems in Section \ref{sec7.3.2}-\ref{sec7.3.3} are generated randomly. 
\end{itemize}







\begin{appendices}






\end{appendices}


\bibliography{sn-bibliography}


\begin{thebibliography}{53}
\ifx \bisbn   \undefined \def \bisbn  #1{ISBN #1}\fi
\ifx \binits  \undefined \def \binits#1{#1}\fi
\ifx \bauthor  \undefined \def \bauthor#1{#1}\fi
\ifx \batitle  \undefined \def \batitle#1{#1}\fi
\ifx \bjtitle  \undefined \def \bjtitle#1{#1}\fi
\ifx \bvolume  \undefined \def \bvolume#1{\textbf{#1}}\fi
\ifx \byear  \undefined \def \byear#1{#1}\fi
\ifx \bissue  \undefined \def \bissue#1{#1}\fi
\ifx \bfpage  \undefined \def \bfpage#1{#1}\fi
\ifx \blpage  \undefined \def \blpage #1{#1}\fi
\ifx \burl  \undefined \def \burl#1{\textsf{#1}}\fi
\ifx \doiurl  \undefined \def \doiurl#1{\url{https://doi.org/#1}}\fi
\ifx \betal  \undefined \def \betal{\textit{et al.}}\fi
\ifx \binstitute  \undefined \def \binstitute#1{#1}\fi
\ifx \binstitutionaled  \undefined \def \binstitutionaled#1{#1}\fi
\ifx \bctitle  \undefined \def \bctitle#1{#1}\fi
\ifx \beditor  \undefined \def \beditor#1{#1}\fi
\ifx \bpublisher  \undefined \def \bpublisher#1{#1}\fi
\ifx \bbtitle  \undefined \def \bbtitle#1{#1}\fi
\ifx \bedition  \undefined \def \bedition#1{#1}\fi
\ifx \bseriesno  \undefined \def \bseriesno#1{#1}\fi
\ifx \blocation  \undefined \def \blocation#1{#1}\fi
\ifx \bsertitle  \undefined \def \bsertitle#1{#1}\fi
\ifx \bsnm \undefined \def \bsnm#1{#1}\fi
\ifx \bsuffix \undefined \def \bsuffix#1{#1}\fi
\ifx \bparticle \undefined \def \bparticle#1{#1}\fi
\ifx \barticle \undefined \def \barticle#1{#1}\fi
\bibcommenthead
\ifx \bconfdate \undefined \def \bconfdate #1{#1}\fi
\ifx \botherref \undefined \def \botherref #1{#1}\fi
\ifx \url \undefined \def \url#1{\textsf{#1}}\fi
\ifx \bchapter \undefined \def \bchapter#1{#1}\fi
\ifx \bbook \undefined \def \bbook#1{#1}\fi
\ifx \bcomment \undefined \def \bcomment#1{#1}\fi
\ifx \oauthor \undefined \def \oauthor#1{#1}\fi
\ifx \citeauthoryear \undefined \def \citeauthoryear#1{#1}\fi
\ifx \endbibitem  \undefined \def \endbibitem {}\fi
\ifx \bconflocation  \undefined \def \bconflocation#1{#1}\fi
\ifx \arxivurl  \undefined \def \arxivurl#1{\textsf{#1}}\fi
\csname PreBibitemsHook\endcsname

\bibitem[\protect\citeauthoryear{Lu et~al.}{2016}]{Lu2016}
\begin{barticle}
\bauthor{\bsnm{Lu}, \binits{C.}},
\bauthor{\bsnm{Yan}, \binits{S.}},
\bauthor{\bsnm{Lin}, \binits{Z.}}:
\batitle{Convex sparse spectral clustering: single-view to multi-view}.
\bjtitle{IEEE Trans. Image Process.}
\bvolume{25}(\bissue{6}),
\bfpage{2833}--\blpage{2843}
(\byear{2016})
\doiurl{10.1109/TIP.2016.2553459}
\end{barticle}
\endbibitem

\bibitem[\protect\citeauthoryear{Park and Zhao}{2018}]{Park2018}
\begin{barticle}
\bauthor{\bsnm{Park}, \binits{S.}},
\bauthor{\bsnm{Zhao}, \binits{H.}}:
\batitle{Spectral clustering based on learning similarity matrix}.
\bjtitle{Bioinformatics}
\bvolume{34}(\bissue{12}),
\bfpage{2069}--\blpage{2076}
(\byear{2018})
\doiurl{10.1093/bioinformatics/bty050}
\end{barticle}
\endbibitem

\bibitem[\protect\citeauthoryear{Lu and Zhang}{2012}]{Lu2012}
\begin{barticle}
\bauthor{\bsnm{Lu}, \binits{Z.}},
\bauthor{\bsnm{Zhang}, \binits{Y.}}:
\batitle{An augmented {L}agrangian approach for sparse principal component
  analysis}.
\bjtitle{Math. Program.}
\bvolume{135}(\bissue{1}),
\bfpage{149}--\blpage{193}
(\byear{2012})
\doiurl{10.1007/s10107-011-0452-4}
\end{barticle}
\endbibitem

\bibitem[\protect\citeauthoryear{Peng and Mohseni}{2016}]{Peng2016}
\begin{barticle}
\bauthor{\bsnm{Peng}, \binits{L.}},
\bauthor{\bsnm{Mohseni}, \binits{K.}}:
\batitle{Symplectic model reduction of {H}amiltonian systems}.
\bjtitle{SIAM J. Sci. Comput.}
\bvolume{38}(\bissue{1}),
\bfpage{1}--\blpage{27}
(\byear{2016})
\doiurl{10.1137/140978922}
\end{barticle}
\endbibitem

\bibitem[\protect\citeauthoryear{Gao et~al.}{2024}]{Gao2024}
\begin{barticle}
\bauthor{\bsnm{Gao}, \binits{B.}},
\bauthor{\bsnm{Son}, \binits{N.T.}},
\bauthor{\bsnm{Stykel}, \binits{T.}}:
\batitle{Optimization on the symplectic {S}tiefel manifold: {SR}
  decomposition-based retraction and applications}.
\bjtitle{Linear Algebra Appl.}
\bvolume{682},
\bfpage{50}--\blpage{85}
(\byear{2024})
\doiurl{10.1016/j.laa.2023.10.025}
\end{barticle}
\endbibitem

\bibitem[\protect\citeauthoryear{Jensen and Zimmermann}{2024}]{Jensen2024}
\begin{botherref}
\oauthor{\bsnm{Jensen}, \binits{R.}},
\oauthor{\bsnm{Zimmermann}, \binits{R.}}:
Riemannian optimization on the symplectic {S}tiefel manifold using second-order
  information.
arXiv preprint
(2024)
{\href{https://arxiv.org/abs/arXiv:2404.08463}{{arXiv:2404.08463}}}
\end{botherref}
\endbibitem

\bibitem[\protect\citeauthoryear{Wang et~al.}{2025}]{Wang2025}
\begin{botherref}
\oauthor{\bsnm{Wang}, \binits{L.}},
\oauthor{\bsnm{Liu}, \binits{X.}},
\oauthor{\bsnm{Chen}, \binits{X.}}:
The distributionally robust optimization model of sparse principal component
  analysis.
arXiv preprint
(2025)
{\href{https://arxiv.org/abs/arXiv:2503.02494}{{arXiv:2503.02494}}}
\end{botherref}
\endbibitem

\bibitem[\protect\citeauthoryear{Si et~al.}{2024}]{Si2024}
\begin{barticle}
\bauthor{\bsnm{Si}, \binits{W.}},
\bauthor{\bsnm{Absil}, \binits{P.-A.}},
\bauthor{\bsnm{Huang}, \binits{W.}},
\bauthor{\bsnm{Jiang}, \binits{R.}},
\bauthor{\bsnm{Vary}, \binits{S.}}:
\batitle{A {R}iemannian proximal {N}ewton method}.
\bjtitle{SIAM J. Optim.}
\bvolume{34}(\bissue{1}),
\bfpage{654}--\blpage{681}
(\byear{2024})
\doiurl{10.1137/23M1565097}
\end{barticle}
\endbibitem

\bibitem[\protect\citeauthoryear{Xiao et~al.}{2021}]{XiaoLiu2021}
\begin{barticle}
\bauthor{\bsnm{Xiao}, \binits{N.}},
\bauthor{\bsnm{Liu}, \binits{X.}},
\bauthor{\bsnm{Yuan}, \binits{Y.-X.}}:
\batitle{Exact penalty function for $\ell_{2,1}$ norm minimization over the
  {S}tiefel manifold}.
\bjtitle{SIAM J. Optim.}
\bvolume{31}(\bissue{4}),
\bfpage{3097}--\blpage{3126}
(\byear{2021})
\doiurl{10.1137/20M1354313}
\end{barticle}
\endbibitem

\bibitem[\protect\citeauthoryear{Zhang et~al.}{2024}]{Zhang2024}
\begin{barticle}
\bauthor{\bsnm{Zhang}, \binits{C.}},
\bauthor{\bsnm{Chen}, \binits{X.}},
\bauthor{\bsnm{Ma}, \binits{S.}}:
\batitle{A {R}iemannian smoothing steepest descent method for non-{L}ipschitz
  optimization on embedded submanifolds of $\mathbb{R}^n$}.
\bjtitle{Math. Oper. Res.}
\bvolume{49}(\bissue{3}),
\bfpage{1710}--\blpage{1733}
(\byear{2024})
\doiurl{10.1287/moor.2022.0286}
\end{barticle}
\endbibitem

\bibitem[\protect\citeauthoryear{Hosseini et~al.}{2018}]{Hosseini2018}
\begin{barticle}
\bauthor{\bsnm{Hosseini}, \binits{S.}},
\bauthor{\bsnm{Huang}, \binits{W.}},
\bauthor{\bsnm{Yousefpour}, \binits{R.}}:
\batitle{Line search algorithms for locally {L}ipschitz functions on
  {R}iemannian manifolds}.
\bjtitle{SIAM J. Optim.}
\bvolume{28}(\bissue{1}),
\bfpage{596}--\blpage{619}
(\byear{2018})
\doiurl{10.1137/16M1108145}
\end{barticle}
\endbibitem

\bibitem[\protect\citeauthoryear{Liu et~al.}{2024}]{LiuXiao2024}
\begin{barticle}
\bauthor{\bsnm{Liu}, \binits{X.}},
\bauthor{\bsnm{Xiao}, \binits{N.}},
\bauthor{\bsnm{Yuan}, \binits{Y.-X.}}:
\batitle{A penalty-free infeasible approach for a class of nonsmooth
  optimization problems over the {S}tiefel manifold}.
\bjtitle{J. Sci. Comput.}
\bvolume{99}(\bissue{2}),
\bfpage{30}
(\byear{2024})
\doiurl{10.1007/s10915-024-02495-4}
\end{barticle}
\endbibitem

\bibitem[\protect\citeauthoryear{Chen et~al.}{2020}]{Chen2020}
\begin{barticle}
\bauthor{\bsnm{Chen}, \binits{S.}},
\bauthor{\bsnm{Ma}, \binits{S.}},
\bauthor{\bsnm{Man-Cho~So}, \binits{A.}},
\bauthor{\bsnm{Zhang}, \binits{T.}}:
\batitle{Proximal gradient method for nonsmooth optimization over the {S}tiefel
  manifold}.
\bjtitle{SIAM J. Optim.}
\bvolume{30}(\bissue{1}),
\bfpage{210}--\blpage{239}
(\byear{2020})
\doiurl{10.1137/18M122457X}
\end{barticle}
\endbibitem

\bibitem[\protect\citeauthoryear{Huang and Wei}{2022}]{Huang2022a}
\begin{barticle}
\bauthor{\bsnm{Huang}, \binits{W.}},
\bauthor{\bsnm{Wei}, \binits{K.}}:
\batitle{Riemannian proximal gradient methods}.
\bjtitle{Math. Program.}
\bvolume{194}(\bissue{1}),
\bfpage{371}--\blpage{413}
(\byear{2022})
\doiurl{10.1007/s10107-021-01632-3}
\end{barticle}
\endbibitem

\bibitem[\protect\citeauthoryear{Wang and Yang}{2023}]{Wang2023}
\begin{barticle}
\bauthor{\bsnm{Wang}, \binits{Q.}},
\bauthor{\bsnm{Yang}, \binits{W.H.}}:
\batitle{Proximal quasi-{N}ewton method for composite optimization over the
  {S}tiefel manifold}.
\bjtitle{J. Sci. Comput.}
\bvolume{95}(\bissue{2}),
\bfpage{39}
(\byear{2023})
\doiurl{10.1007/s10915-023-02165-x}
\end{barticle}
\endbibitem

\bibitem[\protect\citeauthoryear{Wang and Yang}{2024}]{Wang2024}
\begin{barticle}
\bauthor{\bsnm{Wang}, \binits{Q.}},
\bauthor{\bsnm{Yang}, \binits{W.H.}}:
\batitle{An adaptive regularized proximal {N}ewton-type methods for composite
  optimization over the {S}tiefel manifold}.
\bjtitle{Comput. Optim. Appl.}
\bvolume{89}(\bissue{2}),
\bfpage{419}--\blpage{457}
(\byear{2024})
\doiurl{10.1007/s10589-024-00595-3}
\end{barticle}
\endbibitem

\bibitem[\protect\citeauthoryear{Deng and Peng}{2023}]{deng2023manifold}
\begin{barticle}
\bauthor{\bsnm{Deng}, \binits{K.}},
\bauthor{\bsnm{Peng}, \binits{Z.}}:
\batitle{A manifold inexact augmented {L}agrangian method for nonsmooth
  optimization on {R}iemannian submanifolds in {E}uclidean space}.
\bjtitle{IMA J. Numer. Anal.}
\bvolume{43}(\bissue{3}),
\bfpage{1653}--\blpage{1684}
(\byear{2023})
\doiurl{10.1093/imanum/drac018}
\end{barticle}
\endbibitem

\bibitem[\protect\citeauthoryear{Deng et~al.}{2025}]{deng2025oracle}
\begin{barticle}
\bauthor{\bsnm{Deng}, \binits{K.}},
\bauthor{\bsnm{Hu}, \binits{J.}},
\bauthor{\bsnm{Wu}, \binits{J.}},
\bauthor{\bsnm{Wen}, \binits{Z.}}:
\batitle{Oracle complexities of augmented {L}agrangian methods for nonsmooth
  composite optimization on a compact submanifold}.
\bjtitle{Math. Oper. Res.}
(\byear{2025})
\doiurl{10.1287/moor.2024.0498}
\end{barticle}
\endbibitem

\bibitem[\protect\citeauthoryear{Beck and Rosset}{2023}]{Beck2023}
\begin{barticle}
\bauthor{\bsnm{Beck}, \binits{A.}},
\bauthor{\bsnm{Rosset}, \binits{I.}}:
\batitle{A dynamic smoothing technique for a class of nonsmooth optimization
  problems on manifolds}.
\bjtitle{SIAM J. Optim.}
\bvolume{33}(\bissue{3}),
\bfpage{1473}--\blpage{1493}
(\byear{2023})
\doiurl{10.1137/22M1489447}
\end{barticle}
\endbibitem

\bibitem[\protect\citeauthoryear{Li et~al.}{2024}]{LiMa2024}
\begin{barticle}
\bauthor{\bsnm{Li}, \binits{J.}},
\bauthor{\bsnm{Ma}, \binits{S.}},
\bauthor{\bsnm{Srivastava}, \binits{T.}}:
\batitle{A {R}iemannian alternating direction method of multipliers}.
\bjtitle{Math. Oper. Res.}
(\byear{2024})
\doiurl{10.1287/moor.2023.0068}
\end{barticle}
\endbibitem

\bibitem[\protect\citeauthoryear{Wang et~al.}{2022}]{Wang2022}
\begin{barticle}
\bauthor{\bsnm{Wang}, \binits{Z.}},
\bauthor{\bsnm{Liu}, \binits{B.}},
\bauthor{\bsnm{Chen}, \binits{S.}},
\bauthor{\bsnm{Ma}, \binits{S.}},
\bauthor{\bsnm{Xue}, \binits{L.}},
\bauthor{\bsnm{Zhao}, \binits{H.}}:
\batitle{A manifold proximal linear method for sparse spectral clustering with
  application to single-cell {RNA} sequencing data analysis}.
\bjtitle{INFORMS J. Optim.}
\bvolume{4}(\bissue{2}),
\bfpage{200}--\blpage{214}
(\byear{2022})
\doiurl{10.1287/ijoo.2021.0064}
\end{barticle}
\endbibitem

\bibitem[\protect\citeauthoryear{Xu et~al.}{2025}]{Xu2024}
\begin{botherref}
\oauthor{\bsnm{Xu}, \binits{M.}},
\oauthor{\bsnm{Jiang}, \binits{B.}},
\oauthor{\bsnm{Liu}, \binits{Y.-F.}},
\oauthor{\bsnm{So}, \binits{A.M.-C.}}:
On the oracle complexity of a {R}iemannian inexact augmented {L}agrangian
  method for nonsmooth composite problems over {R}iemannian submanifolds.
Optim. Lett.,
1--19
(2025)
\doiurl{10.1007/s11590-025-02210-8}
\end{botherref}
\endbibitem

\bibitem[\protect\citeauthoryear{Xu et~al.}{2026}]{xu2026riemannian}
\begin{barticle}
\bauthor{\bsnm{Xu}, \binits{M.}},
\bauthor{\bsnm{Jiang}, \binits{B.}},
\bauthor{\bsnm{Liu}, \binits{Y.-F.}},
\bauthor{\bsnm{So}, \binits{A.M.-C.}}:
\batitle{A {R}iemannian alternating descent ascent algorithmic framework for
  nonconvex-linear minimax problems on riemannian manifolds}.
\bjtitle{Math. Oper. Res.}
(\byear{2026})
\doiurl{10.1287/moor.2025.1055}
\end{barticle}
\endbibitem

\bibitem[\protect\citeauthoryear{Zhou et~al.}{2023}]{Zhou2023}
\begin{barticle}
\bauthor{\bsnm{Zhou}, \binits{Y.}},
\bauthor{\bsnm{Bao}, \binits{C.}},
\bauthor{\bsnm{Ding}, \binits{C.}},
\bauthor{\bsnm{Zhu}, \binits{J.}}:
\batitle{A semismooth {N}ewton based augmented {L}agrangian method for
  nonsmooth optimization on matrix manifolds}.
\bjtitle{Math. Program.}
\bvolume{201}(\bissue{1}),
\bfpage{1}--\blpage{61}
(\byear{2023})
\doiurl{10.1007/s10107-022-01898-1}
\end{barticle}
\endbibitem

\bibitem[\protect\citeauthoryear{Li et~al.}{2024}]{LiZhang2024}
\begin{botherref}
\oauthor{\bsnm{Li}, \binits{Q.}},
\oauthor{\bsnm{Zhang}, \binits{N.}},
\oauthor{\bsnm{Yan}, \binits{H.}},
\oauthor{\bsnm{Feng}, \binits{J.}}:
Proximal methods for structured nonsmooth optimization over {R}emannian
  submanifolds.
arXiv preprint
(2024)
{\href{https://arxiv.org/abs/arXiv:2411.15776}{{arXiv:2411.15776}}}
\end{botherref}
\endbibitem

\bibitem[\protect\citeauthoryear{Huang and Wei}{2023}]{Huang2023}
\begin{barticle}
\bauthor{\bsnm{Huang}, \binits{W.}},
\bauthor{\bsnm{Wei}, \binits{K.}}:
\batitle{An inexact {R}iemannian proximal gradient method}.
\bjtitle{Comput. Optim. Appl.}
\bvolume{85}(\bissue{1}),
\bfpage{1}--\blpage{32}
(\byear{2023})
\doiurl{10.1007/s10589-023-00451-w}
\end{barticle}
\endbibitem

\bibitem[\protect\citeauthoryear{Zheng et~al.}{2025}]{zheng2025new}
\begin{botherref}
\oauthor{\bsnm{Zheng}, \binits{Z.}},
\oauthor{\bsnm{Yu}, \binits{X.}},
\oauthor{\bsnm{Ma}, \binits{S.}},
\oauthor{\bsnm{Xue}, \binits{L.}}:
A new inexact manifold proximal linear algorithm with adaptive stopping
  criteria.
arXiv preprint
(2025)
{\href{https://arxiv.org/abs/arXiv:2508.19234}{{arXiv:2508.19234}}}
\end{botherref}
\endbibitem

\bibitem[\protect\citeauthoryear{Tao et~al.}{2023}]{tao2023inexact}
\begin{botherref}
\oauthor{\bsnm{Tao}, \binits{T.}},
\oauthor{\bsnm{Liu}, \binits{R.}},
\oauthor{\bsnm{Pan}, \binits{S.}}:
An inexact {LPA} for {DC} composite optimization and application to matrix
  completions with outliers.
arXiv preprint
(2023)
{\href{https://arxiv.org/abs/arXiv:2303.16822}{{arXiv:2303.16822}}}
\end{botherref}
\endbibitem

\bibitem[\protect\citeauthoryear{Necoara and Patrascu}{2016}]{Necoara2016}
\begin{barticle}
\bauthor{\bsnm{Necoara}, \binits{I.}},
\bauthor{\bsnm{Patrascu}, \binits{A.}}:
\batitle{Iteration complexity analysis of dual first-order methods for conic
  convex programming}.
\bjtitle{Optim. Methods Softw.}
\bvolume{31}(\bissue{3}),
\bfpage{645}--\blpage{678}
(\byear{2016})
\doiurl{10.1080/10556788.2016.1161763}
\end{barticle}
\endbibitem

\bibitem[\protect\citeauthoryear{Li and Pong}{2018}]{LiPong2018}
\begin{barticle}
\bauthor{\bsnm{Li}, \binits{G.}},
\bauthor{\bsnm{Pong}, \binits{T.K.}}:
\batitle{Calculus of the exponent of {K}urdyka-{{\L}}ojasiewicz inequality and
  its applications to linear convergence of first-order methods}.
\bjtitle{Found. Comput. Math.}
\bvolume{18}(\bissue{5}),
\bfpage{1199}--\blpage{1232}
(\byear{2018})
\doiurl{10.1007/s10208-017-9366-8}
\end{barticle}
\endbibitem

\bibitem[\protect\citeauthoryear{Rockafellar and Wets}{1998}]{RW98}
\begin{bbook}
\bauthor{\bsnm{Rockafellar}, \binits{R.T.}},
\bauthor{\bsnm{Wets}, \binits{R.J.}}:
\bbtitle{Variational Analysis}.
\bpublisher{Springer},
\blocation{Berlin, Germany}
(\byear{1998})
\end{bbook}
\endbibitem

\bibitem[\protect\citeauthoryear{Boumal}{2023}]{Boumal2023}
\begin{bbook}
\bauthor{\bsnm{Boumal}, \binits{N.}}:
\bbtitle{An Introduction to Optimization on Smooth Manifolds}.
\bpublisher{Cambridge University Press},
\blocation{Cambridge, UK}
(\byear{2023})
\end{bbook}
\endbibitem

\bibitem[\protect\citeauthoryear{Absil et~al.}{2008}]{Absil2008}
\begin{bbook}
\bauthor{\bsnm{Absil}, \binits{P.-A.}},
\bauthor{\bsnm{Mahony}, \binits{R.}},
\bauthor{\bsnm{Sepulchre}, \binits{R.}}:
\bbtitle{Optimization Algorithms on Matrix Manifolds}.
\bpublisher{Princeton University Press},
\blocation{Princeton, NJ}
(\byear{2008})
\end{bbook}
\endbibitem

\bibitem[\protect\citeauthoryear{Boumal et~al.}{2019}]{Boumal2019}
\begin{barticle}
\bauthor{\bsnm{Boumal}, \binits{N.}},
\bauthor{\bsnm{Absil}, \binits{P.-A.}},
\bauthor{\bsnm{Cartis}, \binits{C.}}:
\batitle{Global rates of convergence for nonconvex optimization on manifolds}.
\bjtitle{IMA J. Numer. Anal.}
\bvolume{39}(\bissue{1}),
\bfpage{1}--\blpage{33}
(\byear{2019})
\doiurl{10.1093/imanum/drx080}
\end{barticle}
\endbibitem

\bibitem[\protect\citeauthoryear{Attouch et~al.}{2010}]{Attouch2010}
\begin{barticle}
\bauthor{\bsnm{Attouch}, \binits{H.}},
\bauthor{\bsnm{Bolte}, \binits{J.}},
\bauthor{\bsnm{Redont}, \binits{P.}},
\bauthor{\bsnm{Soubeyran}, \binits{A.}}:
\batitle{Proximal alternating minimization and projection methods for nonconvex
  problems: An approach based on the {K}urdyka-{{\L}}ojasiewicz inequality}.
\bjtitle{Math. Oper. Res.}
\bvolume{35}(\bissue{2}),
\bfpage{438}--\blpage{457}
(\byear{2010})
\doiurl{10.1287/moor.1100.0449}
\end{barticle}
\endbibitem

\bibitem[\protect\citeauthoryear{Beck}{2017}]{Beck2017}
\begin{bbook}
\bauthor{\bsnm{Beck}, \binits{A.}}:
\bbtitle{First-Order Methods in Optimization}.
\bpublisher{Society for Industrial and Applied Mathematics},
\blocation{Philadelphia, US}
(\byear{2017})
\end{bbook}
\endbibitem

\bibitem[\protect\citeauthoryear{Nesterov}{2005}]{Nesterov2005}
\begin{barticle}
\bauthor{\bsnm{Nesterov}, \binits{Y.}}:
\batitle{Smooth minimization of non-smooth functions}.
\bjtitle{Math. {P}rogram.}
\bvolume{103}(\bissue{1}),
\bfpage{127}--\blpage{152}
(\byear{2005})
\doiurl{10.1007/s10107-004-0552-5}
\end{barticle}
\endbibitem

\bibitem[\protect\citeauthoryear{Attouch et~al.}{2013}]{Attouch2013}
\begin{barticle}
\bauthor{\bsnm{Attouch}, \binits{H.}},
\bauthor{\bsnm{Bolte}, \binits{J.}},
\bauthor{\bsnm{Svaiter}, \binits{B.F.}}:
\batitle{Convergence of descent methods for semi-algebraic and tame problems:
  proximal algorithms, forward--backward splitting, and regularized
  {G}auss-{S}eidel methods}.
\bjtitle{Math. Program.}
\bvolume{137}(\bissue{1}),
\bfpage{91}--\blpage{129}
(\byear{2013})
\doiurl{10.1007/s10107-011-0484-9}
\end{barticle}
\endbibitem

\bibitem[\protect\citeauthoryear{Bolte et~al.}{2014}]{Bolte2014}
\begin{barticle}
\bauthor{\bsnm{Bolte}, \binits{J.}},
\bauthor{\bsnm{Sabach}, \binits{S.}},
\bauthor{\bsnm{Teboulle}, \binits{M.}}:
\batitle{Proximal alternating linearized minimization for nonconvex and
  nonsmooth problems}.
\bjtitle{Math. Program.}
\bvolume{146}(\bissue{1}),
\bfpage{459}--\blpage{494}
(\byear{2014})
\doiurl{10.1007/s10107-013-0701-9}
\end{barticle}
\endbibitem

\bibitem[\protect\citeauthoryear{Attouch and Bolte}{2009}]{Attouch2009}
\begin{barticle}
\bauthor{\bsnm{Attouch}, \binits{H.}},
\bauthor{\bsnm{Bolte}, \binits{J.}}:
\batitle{On the convergence of the proximal algorithm for nonsmooth functions
  involving analytic features}.
\bjtitle{Math. Program.}
\bvolume{116}(\bissue{1}),
\bfpage{5}--\blpage{16}
(\byear{2009})
\doiurl{10.1007/s10107-007-0133-5}
\end{barticle}
\endbibitem

\bibitem[\protect\citeauthoryear{Liu et~al.}{2019}]{LiuSo2019}
\begin{barticle}
\bauthor{\bsnm{Liu}, \binits{H.}},
\bauthor{\bsnm{So}, \binits{A.M.-C.}},
\bauthor{\bsnm{Wu}, \binits{W.}}:
\batitle{Quadratic optimization with orthogonality constraint: explicit
  {{\L}}ojasiewicz exponent and linear convergence of retraction-based
  line-search and stochastic variance-reduced gradient methods}.
\bjtitle{Math. Program.}
\bvolume{178}(\bissue{1}),
\bfpage{215}--\blpage{262}
(\byear{2019})
\doiurl{10.1007/s10107-018-1285-1}
\end{barticle}
\endbibitem

\bibitem[\protect\citeauthoryear{Ioffe}{2009}]{Ioffe2009}
\begin{barticle}
\bauthor{\bsnm{Ioffe}, \binits{A.D.}}:
\batitle{An invitation to tame optimization}.
\bjtitle{SIAM J. Optim.}
\bvolume{19}(\bissue{4}),
\bfpage{1894}--\blpage{1917}
(\byear{2009})
\doiurl{10.1137/080722059}
\end{barticle}
\endbibitem

\bibitem[\protect\citeauthoryear{Bolte et~al.}{2009}]{Bolte2009}
\begin{barticle}
\bauthor{\bsnm{Bolte}, \binits{J.}},
\bauthor{\bsnm{Daniilidis}, \binits{A.}},
\bauthor{\bsnm{Lewis}, \binits{A.}}:
\batitle{Tame functions are semismooth}.
\bjtitle{Math. Program.}
\bvolume{117}(\bissue{1}),
\bfpage{5}--\blpage{19}
(\byear{2009})
\doiurl{10.1007/s10107-007-0166-9}
\end{barticle}
\endbibitem

\bibitem[\protect\citeauthoryear{Facchinei and Pang}{2003}]{Facchinei2003}
\begin{bbook}
\bauthor{\bsnm{Facchinei}, \binits{F.}},
\bauthor{\bsnm{Pang}, \binits{J.-S.}}:
\bbtitle{Finite-Dimensional Variational Inequalities and Complementarity
  Problems}.
\bpublisher{Springer},
\blocation{New York, US}
(\byear{2003})
\end{bbook}
\endbibitem

\bibitem[\protect\citeauthoryear{Hiriart-Urruty
  et~al.}{1984}]{Hiriart-Urruty1984}
\begin{barticle}
\bauthor{\bsnm{Hiriart-Urruty}, \binits{J.-B.}},
\bauthor{\bsnm{Strodiot}, \binits{J.-J.}},
\bauthor{\bsnm{Nguyen}, \binits{V.H.}}:
\batitle{Generalized {H}essian matrix and second-order optimality conditions
  for problems with {$C^{1,1}$} data}.
\bjtitle{Appl. math. optim.}
\bvolume{11}(\bissue{1}),
\bfpage{43}--\blpage{56}
(\byear{1984})
\doiurl{10.1007/BF01442169}
\end{barticle}
\endbibitem

\bibitem[\protect\citeauthoryear{Clarke}{1983}]{Clarke1983}
\begin{bbook}
\bauthor{\bsnm{Clarke}, \binits{F.H.}}:
\bbtitle{Optimization and Nonsmooth Analysis}.
\bpublisher{Society for Industrial and Applied Mathematics},
\blocation{Philadelphia, US}
(\byear{1983})
\end{bbook}
\endbibitem

\bibitem[\protect\citeauthoryear{Zhao et~al.}{2010}]{ZhaoSun2010}
\begin{barticle}
\bauthor{\bsnm{Zhao}, \binits{X.-Y.}},
\bauthor{\bsnm{Sun}, \binits{D.}},
\bauthor{\bsnm{Toh}, \binits{K.-C.}}:
\batitle{A {N}ewton-{CG} augmented {L}agrangian method for semidefinite
  programming}.
\bjtitle{SIAM J. Optim.}
\bvolume{20}(\bissue{4}),
\bfpage{1737}--\blpage{1765}
(\byear{2010})
\doiurl{10.1137/080718206}
\end{barticle}
\endbibitem

\bibitem[\protect\citeauthoryear{Gao et~al.}{2021}]{GaoSon2021}
\begin{bchapter}
\bauthor{\bsnm{Gao}, \binits{B.}},
\bauthor{\bsnm{Son}, \binits{N.T.}},
\bauthor{\bsnm{Absil}, \binits{P.-A.}},
\bauthor{\bsnm{Stykel}, \binits{T.}}:
\bctitle{Geometry of the symplectic {S}tiefel manifold endowed with the
  {E}uclidean metric}.
In: \bbtitle{International Conference on Geometric Science of Information},
pp. \bfpage{789}--\blpage{796}
(\byear{2021}).
\bcomment{Springer}
\end{bchapter}
\endbibitem

\bibitem[\protect\citeauthoryear{Beck and Teboulle}{2009}]{beck2009fast}
\begin{barticle}
\bauthor{\bsnm{Beck}, \binits{A.}},
\bauthor{\bsnm{Teboulle}, \binits{M.}}:
\batitle{A fast iterative shrinkage-thresholding algorithm for linear inverse
  problems}.
\bjtitle{SIAM J. Imaging Sci.}
\bvolume{2}(\bissue{1}),
\bfpage{183}--\blpage{202}
(\byear{2009})
\end{barticle}
\endbibitem

\bibitem[\protect\citeauthoryear{Barzilai and Borwein}{1988}]{Barzilai1988}
\begin{barticle}
\bauthor{\bsnm{Barzilai}, \binits{J.}},
\bauthor{\bsnm{Borwein}, \binits{J.M.}}:
\batitle{Two-point step size gradient methods}.
\bjtitle{IMA J. Numer. Anal.}
\bvolume{8}(\bissue{1}),
\bfpage{141}--\blpage{148}
(\byear{1988})
\doiurl{10.1093/imanum/8.1.141}
\end{barticle}
\endbibitem

\bibitem[\protect\citeauthoryear{Bendokat and Zimmermann}{2021}]{Bendokat2021}
\begin{botherref}
\oauthor{\bsnm{Bendokat}, \binits{T.}},
\oauthor{\bsnm{Zimmermann}, \binits{R.}}:
The real symplectic {S}tiefel and {G}rassmann manifolds: metrics, geodesics and
  applications.
arXiv preprint
(2021)
{\href{https://arxiv.org/abs/arXiv:2108.12447}{{arXiv:2108.12447}}}
\end{botherref}
\endbibitem

\bibitem[\protect\citeauthoryear{Wen and Yin}{2013}]{Wen2013}
\begin{barticle}
\bauthor{\bsnm{Wen}, \binits{Z.}},
\bauthor{\bsnm{Yin}, \binits{W.}}:
\batitle{A feasible method for optimization with orthogonality constraints}.
\bjtitle{Math. Program.}
\bvolume{142}(\bissue{1}),
\bfpage{397}--\blpage{434}
(\byear{2013})
\doiurl{10.1007/s10107-012-0584-1}
\end{barticle}
\endbibitem

\bibitem[\protect\citeauthoryear{Strehl and Ghosh}{2002}]{Strehl2002}
\begin{barticle}
\bauthor{\bsnm{Strehl}, \binits{A.}},
\bauthor{\bsnm{Ghosh}, \binits{J.}}:
\batitle{Cluster ensembles---a knowledge reuse framework for combining multiple
  partitions}.
\bjtitle{J. Mach. Learn. Res.}
\bvolume{3}(\bissue{Dec}),
\bfpage{583}--\blpage{617}
(\byear{2002})
\end{barticle}
\endbibitem

\end{thebibliography}

\end{document}